\documentclass[10pt]{article}
   \usepackage{graphicx}
   \usepackage{epsfig}
   \usepackage{amssymb}
   \usepackage{amsmath}
   \usepackage{amsthm}
   \newtheorem{lemma}{Lemma}[section]
   \newtheorem{theorem}{Theorem}[section]

   {\end{description}}

   {\hfill $\bullet$ \\}

   \newcommand{\be}{\begin{equation}}
   \newcommand{\ee}{\end{equation}}

\begin{document}
    \title{Unconditional Stability Of A Two-Step Fourth-Order Modified Explicit Euler/Crank-Nicolson Approach For Solving Time-Variable Fractional Mobile-Immobile Advection-Dispersion Equation}
   \author{Eric Ngondiep$^{\text{\,a\,b}}$}
   \date{$^{\text{\,a\,}}$\small{Department of Mathematics and Statistics, College of Science, Imam Mohammad Ibn Saud\\ Islamic University
        (IMSIU), $90950$ Riyadh $11632,$ Saudi Arabia.}\\
     \text{\,}\\
       $^{\text{\,b\,}}$\small{Hydrological Research Centre, Institute for Geological and Mining Research, 4110 Yaounde-Cameroon.}\\
     \text{,}\\
        \textbf{Email addresses:} ericngondiep@gmail.com/engondiep@imamu.edu.sa}
   \maketitle

   \textbf{Abstract.}
   This paper considers a two-step fourth-order modified explicit Euler/Crank-Nicolson numerical method for solving the time-variable fractional 
   mobile-immobile advection-dispersion model subjects to suitable initial and boundary conditions. Both stability and error estimates of the new 
   approach are deeply analyzed in the $L^{\infty}(0,T;L^{2})$-norm. The theoretical studies show that the proposed technique is unconditionally stable 
   with convergence of order $O(k+h^{4})$, where $h$ and $k$ are space step and time step, respectively. This result indicate that the two-step 
   fourth-order formulation is more efficient than a broad range of numerical schemes widely studied in the literature for the considered problem. 
   Numerical experiments are performed to verify the unconditional stability and convergence rate of the developed algorithm.\\
    \text{\,}\\

   \\
   {\bf AMS Subject Classification (MSC). 65M12, 65M06}.

  \section{Introduction}\label{sec1}
   In the last twenty years, fractional calculus and fractional differential equations have found a broad range of applications in fluid flow, sound, electrodynamics, elasticity, biology, finances, geology, electrostatics, heat, hydrology and medical problems \cite{3mc,3zzy,1zzy,2zzy,7zjl,6ww,25ww,27ll}. A large class of complex models are deeply described via variable-order derivatives. Most recently, the time variable fractional order telegraph equation has been shown to be a suitable model for various physical phenomena. Since the equations modeled by the time fractional partial differential equations (FPDEs) are highly complex, there is no method that can compute an exact solution. The big challenge with such a set of equations is the development of fast and efficient numerical approaches in an approximate solution. In the literature, abundant numerical schemes have been proposed for solving time variable (or constant) order FPDEs such as: finite difference schemes of first order accuracy in time and spatial second-order convergence, compact finite difference scheme with convergence order $O(\tau+h^{4})$ and some numerical procedures of higher order $O(\tau^{2-\gamma}+h^{4})$. All these techniques were one-step methods \cite{17mc,27mc,6mc,12mc,28mc,15mc,31mc}. The author \cite{entlfcdr} developed a two-step numerical scheme with convergence order $O(\tau^{2-\frac{\gamma}{2}}+h^{4})$ for solving the time-fractional convection-diffusion-reaction equation with constant order derivative. For classical integer order ordinary/partial differential equations such as: Navier-Stokes equations, systems of ODEs, mixed Stokes-Darcy model, Shallow water problem, convection-diffusion-reaction equation, advection-diffusion model and conduction equation \cite{en1,22zzy,en2,en3,21zzy,en4,en11,5zzy,en12,en5,7mc,en9,en10,6mc}, a wide set of numerical techniques have been developed and deeply analyzed. For more details, we refer the readers to \cite{en6,22mc,en14,24mc,en17,en19,38mc,en7} and references therein. This class of partial differential equations (PDEs) also have a large set of applications. Furthermore, some methods developed for solving inter order PDEs can be used to efficiently compute approximate solutions of time FPDEs with low computational costs \cite{15mc,12mc,17mc}. In this work, we develop a two-step modified explicit Euler/Crank-Nicolson approach in an approximate solution of time-variable fractional mobile-immobile advection-dispersion model with subjects to suitable initial and boundary conditions. The proposed technique is unconditionally stable, convergence with order $O(\tau+h^{4})$, fast and more efficient than a large class of numerical schemes widely studied in the literature for the considered problem \cite{27mc,28ll,29ll,30zjl,22zjl,10zjl}. A time variable fractional mobile-immobile advection-dispersion equation describes a broad range of problems in physical or mathematical systems which include ocean acoustic propagation and heat conduction through a solid \cite{2zjl}. In such an equation, the integer order time derivative term is added to describing the motion of particles conveniently \cite{2zjl}. Furthermore, this equation lies in a class of second order PDEs that govern continuous time random walks with heavy tailed random waiting times.\\

   In this paper, we propose a two-step fourth-order modified explicit Euler/Crank-Nicolson formulation for the time-variable fractional mobile-immobile
   advection-dispersion equation \cite{27mc} and describing by the following initial-boundary value problem

     \begin{equation}\label{1}
      u_{t}(x,t)=-cD_{0t}^{\beta(x,t)}u(x,t)+u_{xx}(x,t)-u_{x}(x,t)+f(x,t),\text{\,\,\,\,on\,\,\,\,}\Omega=(L_{0},L)\times(t_{0},T),
     \end{equation}
     with initial condition
      \begin{equation}\label{2}
      u(x,0)=u_{0}(x),\text{\,\,\,\,on\,\,\,\,}[L_{0},L],
     \end{equation}
     and boundary condition
      \begin{equation}\label{3}
      u(L_{0},t)=g_{1}(t)\text{\,\,\,\,and\,\,\,\,}u(L,t)=g_{2}(t),\text{\,\,\,\,on\,\,\,\,}[0,T],
     \end{equation}
     where $u_{t}$, $u_{x}$, and $u_{2x}$ denote $\frac{\partial u}{\partial t}$, $\frac{\partial u}{\partial x}$, and $\frac{\partial^{2}u}{\partial x^{2}}$, respectively. $f=f(x,t)$ represents the source term, $u_{0}$ is the initial condition whereas $g_{1}$ and $g_{2}$ designate the boundary conditions. $cD_{0t}^{\beta(x,t)}u$, where $(0<\beta_{1}\leq\beta(x,t)\leq\beta_{2}<1)$, is denotes the variable-order fractional derivative which has been introduced in various physical fields and defined (for example, in \cite{22ll}) as
     \begin{equation}\label{4}
      cD_{0t}^{\beta(x,t)}u(x,t)=\frac{1}{\Gamma(1-\lambda)}\int_{0}^{t}\frac{u_{s}(x,s)}{(t-s)^{\beta(x,t)}}ds.
     \end{equation}
     For the sake of discretization and error estimates, we assume that the analytical solution of the initial-boundary value problem $(\ref{1})$-$(\ref{3})$ is regular enough. We remind that the goal of this work is to develop an efficient numerical method for solving the time-variable fractional equation $(\ref{1})$ subjects to initial condition $(\ref{2})$ and boundary condition $(\ref{3})$. Specifically, the attention is focused on the following three items:
     \begin{description}
      \item[(i1)] full description of a two-step fourth-order modified explicit Euler/Crank-Nicolson approach for time-variable FPDE $(\ref{1})$
      with initial-boundary conditions given by $(\ref{2})$ and $(\ref{3})$, respectively,
      \item[(i2)] analysis of the unconditional stability and convergence rate of the new technique,
      \item[(i3)] a broad range of numerical evidences which confirm the theoretical results.
     \end{description}

     The outline of the paper is as follows. In Section $\ref{sec2}$, we construct the two-step fourth-order modified explicit Euler/Crank-Nicolson
      numerical scheme for computing an approximate solution to the initial-boundary value problem $(\ref{1})$-$(\ref{3})$. Section $\ref{sec3}$
      analyzes both stability and error estimates of the new approach using the $L^{\infty}(0,T;L^{2})$-norm. Some numerical examples that confirm
      the theoretical study are presented and discussed in Section $\ref{sec4}$. Finally, in Section $\ref{sec5}$ we draw the general conclusions
      and provide our future investigations.

    \section{Development of the two-step modified explicit Euler/crank-Nicolson numerical approach}\label{sec2}
    In this section we develop a two-step fourth-order modified explicit Euler/crank-Nicolson numerical method for solving the initial-boundary value problem $(\ref{1})$-$(\ref{3})$. Starting with the explicit Euler scheme, the new approach is a two-step implicit method which approximates the time-variable fractional operator $(\ref{4})$ using both forward and backward difference approximations in each step whereas the convection and diffusion terms are approximated using the central difference formulation. To compute a numerical solution of problem $(\ref{1})$-$(\ref{3})$, we introduce a uniform grid of mesh points $(x_{j},t_{n})$, where $x_{j}=L_{0}+jh$, $j=0,1,2,...,M$, and $t_{n}=nk$, for $n=0,1,2,...,N$. In this discretization, $M$ and $N$ are two positive integers, $h=\frac{L-L_{0}}{M}$ and $k\frac{T}{N}$ are the space step and the time step, respectively. The space of grid functions is defined as $\mathcal{U}_{hk}=\{u(x_{j},t_{n}),\text{\,}0\leq j\leq M;\text{\,}0\leq n\leq N\}$. For the convenience of writing, we set $u(x_{j},t_{n})=u^{n}_{j}$. The exact solution and the computed one at the grid point $(x_{j},t_{n})$ are denoted by $u_{j}^{n}$ and $U_{j}^{n}$, respectively. Furthermore, the values of the functions $\beta$ and $f$ at the mesh point $(x_{j},t_{n})$ are given by $\beta_{j}^{n}$ and $f_{j}^{n}$, respectively. Finally we introduce the positive parameter $\alpha\in(0,\frac{1}{2})$.\\

     Furthermore, we define the following operators
       \begin{equation}\label{2a}
        \delta_{t} u_{j}^{n}=\frac{u_{j}^{n+\frac{1}{2}}-u_{j}^{n}}{k/2},\text{\,}\delta_{x}u_{j-\frac{1}{2}}^{n}=\frac{u_{j}^{n}
        -u_{j-1}^{n}}{h},\text{\,}\delta_{x}u_{j+\frac{1}{2}}^{n}=\frac{u_{j+1}^{n}-u_{j}^{n}}{h},\text{\,}\delta_{x}^{2}u_{j}^{n}=
        \frac{u_{j+1}^{n}-2u_{j}^{n}+u_{j-1}^{n}}{h^{2}},
       \end{equation}
       and we introduce the given norms
       \begin{equation}\label{dn}
        \|u^{n}\|_{2}=\left(h\underset{j=2}{\overset{M-2}\sum}|u_{j}^{n}|^{2}\right)^{\frac{1}{2}},\text{\,\,\,}
        \||u|\|_{\mathcal{C}^{6,3}_{D}}=\underset{0\leq r\leq 6}{\underset{0\leq s\leq 3}\max}\left\{\underset{0\leq n\leq
        N}{\max}\left\|\frac{\partial^{r+s}u}{\partial x^{r}\partial t^{s}}(t_{n})\right\|_{L^{2}}\right\}
        \text{\,\,\,and\,\,\,}\||u|\|_{\infty,2}=\underset{0\leq n\leq N}{\max}\|u^{n}\|_{2},
       \end{equation}
       together with the inner products
       \begin{equation*}
        \left(u^{n},v^{n}\right)=h\underset{j=2}{\overset{M-2}\sum}u_{j}^{n}v_{j}^{n},\text{\,\,}\left(\delta_{x}u^{n},v^{n}\right)=
       h\underset{j=1}{\overset{M-2}\sum}\delta_{x}u_{j+\frac{1}{2}}^{n}v_{j}^{n}=h\underset{j=2}{\overset{M-1}\sum}\delta_{x}u_{j-\frac{1}{2}}^{n}v_{j}^{n},
       \text{\,\,}\left(\delta_{x}u^{n},\delta_{x}v^{n}\right)=h\underset{j=1}{\overset{M-2}\sum}\delta_{x}u_{j+\frac{1}{2}}^{n}\delta_{x}v_{j+\frac{1}{2}}^{n},
     \end{equation*}
       \begin{equation}\label{sp}
       \left(\delta_{x}u^{n},\delta_{x}v^{n}\right)=h\underset{j=2}{\overset{M-1}\sum}\delta_{x}u_{j-\frac{1}{2}}^{n}\delta_{x}v_{j-\frac{1}{2}}^{n},
       \end{equation}
       where $D=[L_{0},L]\times[0,T]$, $|\cdot|$ is the $\mathbb{C}$-norm. The spaces $L^{2}(L_{0},L)$, $L^{\infty}(0,T;L^{2})$ and $\mathcal{C}^{6,3}_{D}$ are equipped with the norms $\|\cdot\|_{2}$, $\||\cdot|\|_{\infty,2}$ and $\||\cdot|\|_{\mathcal{C}^{6,3}_{D}}$, respectively, whereas the Hilbert space $L^{2}(L_{0},L)$ is endowed with the scalar product $\left(\cdot,\cdot\right)$.\\

     The application of the Taylor series for $u$ at the grid point $(x_{j},t_{n+\frac{1}{2}+\alpha})$ with time step $\frac{k}{2}$ using forward difference representation gives
     \begin{equation*}
        u^{n+\frac{1}{2}+\alpha}_{j}=u_{j}^{n+\alpha}+\frac{k}{2}u_{t,j}^{n+\alpha}+O(k^{2}).
     \end{equation*}
     Utilizing equation $(\ref{1})$, this becomes
      \begin{equation}\label{5}
      u^{n+\frac{1}{2}+\alpha}_{j}=u^{n+\alpha}_{j}+\frac{k}{2}[-cD_{0t}^{\beta_{j}^{n+\alpha}}u_{j}^{n+\alpha}-u_{x,j}^{n+\alpha}+u_{2x,j}^{n+\alpha}+f_{j}^{n+\alpha}]+
      O(k^{2}).
      \end{equation}
      Expanding the Taylor series for $u$ at the mesh point $(x_{j},t_{n+\frac{1}{2}+\alpha})$ with time step $\frac{k}{2}$ using both forward and backward difference formulations to get
      \begin{equation*}
        u^{n+1+\alpha}_{j}=u^{n+\frac{1}{2}+\alpha}_{j}+\frac{k}{2}u^{n+\frac{1}{2}+\alpha}_{t,j}+\frac{k^{2}}{8}u^{n+\frac{1}{2}+\alpha}_{2t,j}+O(k^{3}),
     \end{equation*}
     \begin{equation*}
     u^{n+\frac{1}{2}+\alpha}_{j}=u^{n+1+\alpha}_{j}-\frac{k}{2}u^{n+1+\alpha}_{t,j}+\frac{k^{2}}{8}u^{n+1+\alpha}_{2t,j}+O(k^{3}).
     \end{equation*}
     Using equation $(\ref{1})$, we obtain
     \begin{equation*}
     u^{n+1+\alpha}_{j}=u^{n+\frac{1}{2}+\alpha}_{j}+\frac{k}{2}[-cD_{0t}^{\beta_{j}^{n+\frac{1}{2}+\alpha}}u_{j}^{n+\frac{1}{2}+\alpha}-
     u_{x,j}^{n+\frac{1}{2}+\alpha}+u_{2x,j}^{n+\frac{1}{2}+\alpha}+f_{j}^{n+\frac{1}{2}+\alpha}]+\frac{k^{2}}{8}u^{n+\frac{1}{2}+\alpha}_{2t,j}+O(k^{3}),
     \end{equation*}
     \begin{equation*}
     u^{n+\frac{1}{2}+\alpha}_{j}=u^{n+1+\alpha}_{j}-\frac{k}{2}[-cD_{0t}^{\beta_{j}^{n+1+\alpha}}u_{j}^{n+1+\alpha}-u_{x,j}^{n+1+\alpha}
     +u_{2x,j}^{n+1+\alpha}+f_{j}^{n+1+\alpha}]+\frac{k^{2}}{8}u^{n+1+\alpha}_{2t,j}+O(k^{3}).
     \end{equation*}
     Subtracting the second equation from the first one and using the Mean value theorem, it is easy to see that
      \begin{equation*}
     u^{n+1+\alpha}_{j}-u^{n+\frac{1}{2}+\alpha}_{j}=\frac{k}{4}\left[-(cD_{0t}^{\beta_{j}^{n+1+\alpha}}u_{j}^{n+1+\alpha}+cD_{0t}^{\beta_{j}^{n+\frac{1}{2}+
     \alpha}}u_{j}^{n+\frac{1}{2}+\alpha})-(u_{x,j}^{n+1+\alpha}+u_{x,j}^{n+\frac{1}{2}+\alpha})+\right.
     \end{equation*}
      \begin{equation*}
     \left.(u_{2x,j}^{n+1+\alpha}+u_{2x,j}^{n+\frac{1}{2}+\alpha})+(f_{j}^{n+1+\alpha}+f_{j}^{n+\frac{1}{2}+\alpha})\right]+O(k^{3}),
     \end{equation*}
     which is equivalent to
      \begin{equation*}
     u^{n+1+\alpha}_{j}=u^{n+\frac{1}{2}+\alpha}_{j}+\frac{k}{4}\left[-(cD_{0t}^{\beta_{j}^{n+1+\alpha}}u_{j}^{n+1+\alpha}+cD_{0t}^{\beta_{j}^{n+\frac{1}{2}+
     \alpha}}u_{j}^{n+\frac{1}{2}+\alpha})-(u_{x,j}^{n+1+\alpha}+u_{x,j}^{n+\frac{1}{2}+\alpha})+\right.
     \end{equation*}
      \begin{equation}\label{6}
     \left.(u_{2x,j}^{n+1+\alpha}+u_{2x,j}^{n+\frac{1}{2}+\alpha})+(f_{j}^{n+1+\alpha}+f_{j}^{n+\frac{1}{2}+\alpha})\right]+O(k^{3}).
     \end{equation}
      For the convenience of writing, we set $\gamma_{j}^{n+\frac{1}{2}+\alpha}=\Gamma(1-\beta_{j}^{n+\frac{1}{2}+\alpha})^{-1}$. So the variable-order fractional time derivative given by $(\ref{4})$ at the grid point $(x_{j},t_{n+\frac{1}{2}+\alpha})$ can be rewritten as
       \begin{equation*}
        cD_{0t}^{\beta_{j}^{n+\frac{1}{2}+\alpha}}u_{j}^{n+\frac{1}{2}+\alpha}=\gamma_{j}^{n+\frac{1}{2}+\alpha}\int_{0}^{t_{n+\frac{1}{2}+\alpha}}
        \frac{u_{s,j}(s)}{(t_{n+\frac{1}{2}+\alpha}-s)^{\beta_{j}^{n+\frac{1}{2}+\alpha}}}ds=\gamma_{j}^{n+\frac{1}{2}+\alpha}\underset{l=0}{\overset{n-1}\sum}
        \int_{t_{l+\frac{1}{2}}}^{t_{l+\frac{3}{2}}}\frac{u_{s,j}(s)}{(t_{n+\frac{1}{2}+\alpha}-s)^{\beta_{j}^{n+\frac{1}{2}+\alpha}}}ds
       \end{equation*}
       \begin{equation}\label{8}
      +\gamma_{j}^{n+\frac{1}{2}+\alpha}\left[\int_{0}^{t_{\frac{1}{2}}}\frac{u_{s,j}(s)}{(t_{n+\frac{1}{2}+\alpha}-s)^{\beta_{j}^{n+\frac{1}{2}+\alpha}}}ds
      +\int_{t_{n+\frac{1}{2}}}^{t_{n+\frac{1}{2}+\alpha}}\frac{u_{s,j}(s)}{(t_{n+\frac{1}{2}+\alpha}-s)^{\beta_{j}^{n+\frac{1}{2}+\alpha}}}d\tau\right].
       \end{equation}

       Let $q_{2,j}^{1,u_{j}}$ be the polynomial of degree two approximating $u_{j}(t)$ at the mesh points $(t_{l},u_{j}^{l})$, $(t_{l+\frac{1}{2}},u_{j}^{l+\frac{1}{2}})$ and $(t_{l+\frac{3}{2}},u_{j}^{l+\frac{3}{2}})$ and let $E^{1,u_{j}}_{j}(t)$ be the corresponding error. Replacing $\lambda$ by $\beta_{j}^{n+\frac{1}{2}+\alpha}$ in equation $(53)$ provided in \cite{entlfcdr} to get
       \begin{equation}\label{9}
        cD_{0t}^{\beta_{j}^{n+\frac{1}{2}+\alpha}}u_{j}^{n+\frac{1}{2}+\alpha}=c\Delta_{0t}^{\beta_{j}^{n+\frac{1}{2}+\alpha}}u_{j}^{n+\frac{1}{2}+\alpha}+
        I_{j}^{\beta_{j}^{n+\frac{1}{2}+\alpha}},
       \end{equation}
       where $c\Delta_{0t}^{\beta_{j}^{n+\frac{1}{2}+\alpha}}u_{j}^{n+\frac{1}{2}+\alpha}$ and $I_{j}^{\beta_{j}^{n+\frac{1}{2}+\alpha}}$ are defined by
       \begin{equation}\label{10}
        c\Delta_{0t}^{\beta_{j}^{n+\frac{1}{2}+\alpha}}u_{j}^{n+\frac{1}{2}+\alpha}=k^{1-\beta_{j}^{n+\frac{1}{2}+\alpha}}(1-\beta_{j}^{n+\frac{1}{2}+\alpha})^{-1}
        \gamma_{j}^{n+\frac{1}{2}+\alpha}\underset{l=0}{\overset{n}\sum}a_{n+\frac{1}{2},l+\frac{1}{2}}^{\alpha,\beta_{j}^{n+\frac{1}{2}+\alpha}}\delta_{t}u_{j}^{l},
       \end{equation}
       \begin{equation*}
        I_{j}^{\beta_{j}^{n+\frac{1}{2}+\alpha}}=\gamma_{j}^{n+\frac{1}{2}+\alpha}\left[\int_{0}^{t_{\frac{1}{2}}}\frac{u_{s,j}(s)}{(t_{n+\frac{1}{2}+\alpha}-
       s)^{\beta_{j}^{n+\frac{1}{2}+\alpha}}}ds+\int_{t_{n+\frac{1}{2}}}^{t_{n+\frac{1}{2}+\alpha}}\frac{u_{s,j}(s)}
       {(t_{n+\frac{1}{2}+\alpha}-s)^{\beta_{j}^{n+\frac{1}{2}+\alpha}}}ds\right]+
       \end{equation*}
       \begin{equation}\label{11}
       \gamma_{j}^{n+\frac{1}{2}+\alpha}\underset{i=0}{\overset{n-1}\sum}\int_{t_{i+\frac{1}{2}}}^{t_{i+\frac{3}{2}}}
       \frac{E^{1,u_{j}}_{s,j}(s)}{(t_{n+\frac{1}{2}+\alpha}-s)^{\beta_{j}^{n+\frac{1}{2}+\alpha}}}ds.
       \end{equation}
       The summation index in equation $(\ref{10})$ varies in the range: $l=0,\frac{1}{2},1,\frac{3}{2},2,...,n$, whereas the summation index in $(\ref{11})$
       satisfies: $i=0,1,2,3,...,n-1$. In relation $(\ref{10})$, the coefficients $a_{n+\frac{1}{2},l+\frac{1}{2}}^{\alpha,\beta_{j}^{n+\frac{1}{2}+\alpha}}$
       form a generalized sequence defined in \cite{entlfcdr}, page $11$, when replacing $\lambda$ with $\beta_{j}^{n+\frac{1}{2}+\alpha}$ by
    \begin{equation}\label{12}
    a_{\frac{1}{2},\frac{1}{2}}^{\alpha,\beta_{j}^{\alpha}}=\frac{1}{2}\alpha^{-\beta_{j}^{\alpha}},
    \text{\,\,\,and\,\,\,for\,\,\,}n\geq1,\text{\,\,\,\,}a_{n+\frac{1}{2},\frac{1}{2}}^{\alpha,\beta_{j}^{n+\frac{1}{2}+\alpha}}=
    \widetilde{\text{\,\.{f}}}_{n+\frac{1}{2},0}^{\alpha,\beta_{j}^{n+\frac{1}{2}+\alpha}},
   \end{equation}
   \begin{equation}\label{13}
    a_{n+\frac{1}{2},l+\frac{1}{2}}^{\alpha,\beta_{j}^{n+\frac{1}{2}+\alpha}}=\begin{array}{c}
                                          \left\{
                                            \begin{array}{ll}
                                              \widetilde{d}_{n+\frac{1}{2},l}^{\alpha,\beta_{j}^{n+\frac{1}{2}+\alpha}}-
                                              \widetilde{f}_{n+\frac{1}{2},l}^{\alpha,\beta_{j}^{n+\frac{1}{2}+\alpha}},
                                              & \hbox{if $l=1,2,3,...,n-1$,} \\
                                              \text{\,}\\
                                              \widetilde{f}_{n+\frac{1}{2},l-\frac{1}{2}}^{\alpha,\beta_{j}^{n+\frac{1}{2}+\alpha}}, & \hbox{if
                                                $l=\frac{1}{2},\frac{3}{2},...,n-\frac{1}{2}$,} \\
                                               \text{\,}\\
                                              \widetilde{f}_{n+\frac{1}{2},n-1}^{\alpha,\beta_{j}^{n+\frac{1}{2}+\alpha}}+
                                              \widetilde{f}_{n+\frac{1}{2},n}^{\alpha,\beta_{j}^{n+\frac{1}{2}+\alpha}},
                                              & \hbox{if $l=n$,} \\
                                            \end{array}
                                          \right.
                                        \end{array}
         \end{equation}
          where the terms $\widetilde{f}_{\frac{1}{2},0}^{\alpha,\beta_{j}^{n+\frac{1}{2}+\alpha}}$,$\widetilde{\text{\,\.{f}}}_{n+\frac{1}{2},0}^{\alpha,
          \beta_{j}^{n+\frac{1}{2}+\alpha}}$, $\widetilde{f}_{n+\frac{1}{2},s}^{\alpha,\beta_{j}^{n+\frac{1}{2}+\alpha}}$ and $\widetilde{d}_{n+\frac{1}{2},s}^{\alpha,\beta_{j}^{n+\frac{1}{2}+\alpha}}$ (for $s=0,1,2,...,n$) are given in \cite{entlfcdr}, page $5$-$6$, by replacing $\lambda$ by $\beta_{j}^{n+\frac{1}{2}+\alpha}$ by
         \begin{equation*}
        \widetilde{f}^{\alpha,\beta_{j}^{n+\frac{1}{2}+\alpha}}_{n+\frac{1}{2},n}=\alpha^{1-\beta_{j}^{n+\frac{1}{2}+\alpha}},\text{\,\,}
        \widetilde{f}^{\alpha,\beta_{j}^{\frac{1}{2}+\alpha}}_{\frac{1}{2},0}=(\frac{1}{2}+\alpha)^{1-\beta_{j}^{\frac{1}{2}+\alpha}},
        \text{\,\,}\widetilde{\text{\,\.{f}}}^{\alpha,\beta_{j}^{n+\frac{1}{2}+\alpha}}_{n+\frac{1}{2},0}=(n+\frac{1}{2}+\alpha)^{1-\beta_{j}^{n+\frac{1}{2}+\alpha}}
        -(n+\alpha)^{1-\beta_{j}^{n+\frac{1}{2}+\alpha}},
       \end{equation*}
        \begin{equation*}
        \widetilde{d}^{\alpha,\beta_{j}^{n+\frac{1}{2}+\alpha}}_{n+\frac{1}{2},i}=(n+\alpha-i)^{1-\beta_{j}^{n+\frac{1}{2}+\alpha}}-(n+\alpha-i-1)^{1-
        \beta_{j}^{n+\frac{1}{2}+\alpha}},\text{\,\,\,}\widetilde{f}^{\alpha,\beta_{j}^{n+\frac{1}{2}+\alpha}}_{n+\frac{1}{2},i}=
        \frac{2}{2-\beta_{j}^{n+\frac{1}{2}+\alpha}}\left[(n+\alpha-i)^{2-\beta_{j}^{n+\frac{1}{2}+\alpha}}\right.
       \end{equation*}
       \begin{equation}\label{14}
        \left.-(n+\alpha-i-1)^{2-\beta_{j}^{n+\frac{1}{2}+\alpha}}\right]-\frac{1}{2}\left[(n+\alpha-i)^{1-\beta_{j}^{n+\frac{1}{2}+\alpha}}
        +3(n+\alpha-i-1)^{1-\beta_{j}^{n+\frac{1}{2}+\alpha}}\right],
       \end{equation}
        for $n\geq1$ and $i=0,1,2,...,n-1$. Furthermore,
        \begin{equation*}
        cD_{0t}^{\beta_{j}^{\alpha}}u_{j}^{\alpha}=\frac{1}{\Gamma(1-\beta_{j}^{\alpha})}\int_{0}^{t_{\alpha}}\frac{u_{s,j}(s)}
        {(t_{\alpha}-s)^{\beta_{j}^{\alpha}}}ds=\frac{1}{\Gamma(1-\lambda)}\left[\int_{0}^{t_{\alpha}}\frac{u_{j}^{\alpha}-u_{j}^{0}}
        {\alpha k(t_{\alpha}-s)^{\beta_{j}^{\alpha}}}ds+\int_{0}^{t_{\alpha}}\frac{(u_{s,j}(s)-P_{1,s}^{0,u_{j}}(s)}
        {(t_{\alpha}-s)^{\beta_{j}^{\alpha}}ds}\right]
       \end{equation*}
       \begin{equation}\label{15}
       =c\Delta_{0t}^{\beta_{j}^{\alpha}}u_{j}^{\alpha}+J_{j}^{\beta_{j}^{\alpha}},
       \end{equation}
       where
        \begin{equation}\label{16}
      c\Delta_{0t}^{\beta_{j}^{\alpha}}u_{j}^{\alpha}=k^{1-\beta_{j}^{\alpha}}\Gamma(2-\beta_{j}^{\alpha})^{-1}
      a_{\frac{1}{2},\frac{1}{2}}^{\alpha,\beta_{j}^{\alpha}}\delta_{t}^{\alpha}u_{j}^{0},
       \end{equation}
       \begin{equation}\label{16a}
      \delta_{t}^{\alpha}u_{j}^{0}=\frac{2}{k}(u_{j}^{\alpha}-u_{j}^{0})\text{\,\,\,and\,\,\,}J_{j}^{\beta_{j}^{\alpha}}=
      \frac{1}{\Gamma(1-\beta_{j}^{\alpha})}\int_{0}^{t_{\alpha}}\frac{u_{s,j}(s)-P_{1,s}^{0,u_{j}}(s)}{(t_{\alpha}-s)^{\beta_{j}^{\alpha}}}ds.
       \end{equation}
       $P_{1,s}^{0,u_{j}}$ is the first-order polynomial interpolating $u_{j}$ at the mesh point $(u_{j}^{0},t_{0})$ and $(u_{j}^{\alpha},t_{\alpha})$.\\

       In a similar manner, setting $\gamma_{j}^{n+1+\alpha}=\Gamma(1-\beta_{j}^{n+1+\alpha})^{-1}$ and replacing $\lambda$ by $\beta_{j}^{n+1+\alpha}$
       in the formulas obtained in \cite{entlfcdr}, pages: $8$, $12$, and $17$, results in
       \begin{equation}\label{17}
        cD_{0t}^{\beta_{j}^{n+1+\alpha}}u_{j}^{n+1+\alpha}=c\Delta_{0t}^{\beta_{j}^{n+1+\alpha}}u_{j}^{n+1+\alpha}+J_{j}^{\beta_{j}^{n+1+\alpha}},
       \end{equation}
       where
       \begin{equation}\label{18}
        c\Delta_{0t}^{\beta_{j}^{n+1+\alpha}}u_{j}^{n+1+\alpha}=k^{1-\beta_{j}^{n+1+\alpha}}(1-\beta_{j}^{n+1+\alpha})^{-1}
        \gamma_{j}^{n+1+\alpha}\underset{l=0}{\overset{n+\frac{1}{2}}\sum}a_{n+1,l+\frac{1}{2}}^{\alpha,\beta_{j}^{n+1+\alpha}}\delta_{t}u_{j}^{l},
       \end{equation}
        with
       \begin{equation}\label{19}
       a_{1,\frac{1}{2}}^{\alpha,\beta_{j}^{1+\alpha}}=\widetilde{d}_{1,0}^{\alpha,\beta_{j}^{1+\alpha}}-\widetilde{f}_{1,0}^{\alpha,\beta_{j}^{1+\alpha}},
       \text{\,\,\,}a_{1,1}^{\alpha,\beta_{j}^{1+\alpha}}=\widetilde{f}_{1,0}^{\alpha,\beta_{j}^{1+\alpha}}+\widetilde{f}_{1,1}^{\alpha,\beta_{j}^{1+\alpha}},
      \end{equation}
       and for $n\geq1$
      \begin{equation}\label{20}
      a_{n+1,l+\frac{1}{2}}^{\alpha,\beta_{j}^{n+1+\alpha}}=\begin{array}{c}
                                          \left\{
                                            \begin{array}{ll}
                                              \widetilde{d}_{n+1,l}^{\alpha,\beta_{j}^{n+1+\alpha}}-\widetilde{f}_{n+1,l}^{\alpha,\beta_{j}^{n+1+\alpha}},
                                              & \hbox{if $l=0,1,2,3,...,n$,} \\
                                              \text{\,}\\
                                              \widetilde{f}_{n+1,l-\frac{1}{2}}^{\alpha,\beta_{j}^{n+1+\alpha}}, & \hbox{if
                                                $l=\frac{1}{2},\frac{3}{2},...,n-\frac{1}{2}$,} \\
                                               \text{\,}\\
                                              \widetilde{f}_{n+1,n}^{\alpha,\beta_{j}^{n+1+\alpha}}+\widetilde{f}_{n+1,n+1}^{\alpha,\beta_{j}^{n+1+\alpha}},
                                              & \hbox{if $l=n+\frac{1}{2}$.} \\
                                            \end{array}
                                          \right.
                                        \end{array}
         \end{equation}
          Here the terms $\widetilde{f}_{n+1,r}^{\alpha,\beta_{j}^{n+1+\alpha}}$ and $\widetilde{d}_{n+1,r}^{\alpha,\beta_{j}^{n+1+\alpha}}$, for $n\geq1$ and $r=0,1,2,...,n+1$, are obtained by replacing $\lambda$ by $\beta_{j}^{n+1+\alpha}$ in \cite{entlfcdr}, page $8$-$9$, by
         \begin{equation*}
         \widetilde{d}^{\alpha,\beta_{j}^{n+1+\alpha}}_{n+1,i}=(n+1+\alpha-i)^{1-\beta_{j}^{n+1+\alpha}}-(n+\alpha-i)^{1-\beta_{j}^{n+1+\alpha}},
        \text{\,\,}\widetilde{f}^{\alpha,\beta_{j}^{n+1+\alpha}}_{n+1,n+1}=\alpha^{1-\beta_{j}^{n+1+\alpha}},
       \end{equation*}
       \begin{equation*}
        \widetilde{f}^{\alpha,\beta_{j}^{n+1+\alpha}}_{n+1,i}=\frac{2}{2-\beta_{j}^{n+1+\alpha}}\left[(n+1+\alpha-i)^{2-\beta_{j}^{n+1+\alpha}}
        -(n+\alpha-i)^{2-\beta_{j}^{n+1+\alpha}}\right]
       \end{equation*}
        \begin{equation}\label{21}
        -\frac{1}{2}\left[(n+1+\alpha-i)^{1-\beta_{j}^{n+1+\alpha}}+3(n+\alpha-i)^{1-\beta_{j}^{n+1+\alpha}}\right].
       \end{equation}
       We recall that in equation $(\ref{18})$ the summation index varies in the range $l=0,\frac{1}{2},1,\frac{3}{2},2,...,n+\frac{1}{2}$. Furthermore, replacing in relation $(52)$ given in \cite{entlfcdr}, $\lambda$ with $\beta_{j}^{n+1+\alpha}$, we obtain
       \begin{equation}\label{22}
       J_{j}^{\beta_{j}^{n+1+\alpha}}=\gamma_{j}^{n+1+\alpha}\left[\int_{t_{n+1}}^{t_{n+1+\alpha}}\frac{u_{s,j}(s)}
       {(t_{n+1+\alpha}-s)^{\beta_{j}^{n+1+\alpha}}}ds+\underset{i=0}{\overset{n}\sum}\int_{t_{i}}^{t_{i+1}}
       \frac{E^{2,u_{j}}_{s,j}(s)}{(t_{n+1+\alpha}-s)^{\beta_{j}^{n+1+\alpha}}}ds\right],
       \end{equation}
       where $E^{2,u_{j}}_{s,j}(s)$ is the error associated with the quadratic polynomial interpolating $u_{j}(t)$ at the grid points $(u_{j}^{l},t_{l})$,
       $(u_{j}^{l+\frac{1}{2}},t_{l+\frac{1}{2}})$ and $(u_{j}^{l+1},t_{l+1})$.\\

       Now, we should find the spatial fourth-order approximations of the terms $u_{x,j}^{n+s+\alpha}$ and $u_{2x,j}^{n+s+\alpha}$, for $s=0,\frac{1}{2}$ and $1$. Using the Taylor series expansion for $u$ about the grid point $(x_{j},t_{n+s+\alpha})$ with step size $h$ using both forward and backward
        representations, the author \cite{entlfcdr} has shown that
      \begin{equation}\label{23}
       u_{2x,j}^{n+s+\alpha}=\frac{1}{12h^{2}}\left[-u_{j+2}^{n+s+\alpha}+16u_{j+1}^{n+s+\alpha}-
      30u_{j}^{n+s+\alpha}+16u_{j-1}^{n+s+\alpha}-u_{j-2}^{n+s+\alpha}\right]+O(h^{4})=\delta_{2x}^{4}u_{j}^{n+s+\alpha}+O(h^{4}),
      \end{equation}
       \begin{equation}\label{24}
       u_{x,j}^{n+s+\alpha}=\frac{1}{12h}\left[-u_{j+2}^{n+s+\alpha}+8u_{j+1}^{n+s+\alpha}-
      8u_{j-1}^{n+s+\alpha}+u_{j-2}^{n+s+\alpha}\right]+O(h^{4})=\delta_{x}^{4}u_{j}^{n+s+\alpha}+O(h^{4})
      \end{equation}
      for $s\in\{0,\frac{1}{2},1\}$, where
       \begin{equation}\label{23a}
       \delta_{2x}^{4}u_{j}^{n+s+\alpha}=\frac{1}{12h^{2}}\left[-u_{j+2}^{n+s+\alpha}+16u_{j+1}^{n+s+\alpha}-
      30u_{j}^{n+s+\alpha}+16u_{j-1}^{n+s+\alpha}-u_{j-2}^{n+s+\alpha}\right],
       \end{equation}
        \begin{equation}\label{24a}
       \delta_{x}^{4}u_{j}^{n+s+\alpha}=\frac{1}{12h}\left[-u_{j+2}^{n+s+\alpha}+8u_{j+1}^{n+s+\alpha}-
      8u_{j-1}^{n+s+\alpha}+u_{j-2}^{n+s+\alpha}\right].
       \end{equation}
      For $s=0$, replacing $n+1$ by $n$ in $(\ref{17})$ and plugging the obtained equation with $(\ref{24})$, $(\ref{23})$ and $(\ref{5})$ yields
       \begin{equation}\label{25a}
      u^{n+\frac{1}{2}+\alpha}_{j}=u^{n+\alpha}_{j}+\frac{k}{2}\left[-c\Delta_{0t}^{\beta_{j}^{n+\alpha}}u_{j}^{n+\alpha}-J_{j}^{\beta_{j}^{n+\alpha}}-
      \delta_{x}^{4}u_{j}^{n+\alpha}+\delta_{2x}^{4}u_{j}^{n+\alpha}+f_{j}^{n+\alpha}\right]+O(k^{2}+kh^{4}).
      \end{equation}
      Replacing $n+1$ by $n$ in $(\ref{18})$, substituting the new equation into $(\ref{25a})$  and rearranging terms results in
      \begin{equation*}
       u^{n+\frac{1}{2}+\alpha}_{j}=u^{n+\alpha}_{j}-\frac{1}{2}k^{2-\beta_{j}^{n+\alpha}}(1-\beta_{j}^{n+\alpha})^{-1}\gamma_{j}^{n+\alpha}
      \underset{l=0}{\overset{n-\frac{1}{2}}\sum}a_{n,l+\frac{1}{2}}^{\alpha,\beta_{j}^{n+\alpha}}\delta_{t}u_{j}^{l}
      +\frac{k}{2}\left[-\delta_{x}^{4}u_{j}^{n+\alpha}+\delta_{2x}^{4}u_{j}^{n+\alpha}+f_{j}^{n+\alpha}\right]
       \end{equation*}
      \begin{equation}\label{25}
      -\frac{k}{2}J_{j}^{\beta_{j}^{n+\alpha}}+O(k^{2}+kh^{4}).
      \end{equation}
      Since $\gamma_{j}^{n+\alpha}=\Gamma(1-\beta_{j}^{n+\alpha})^{-1}$ so, $(1-\beta_{j}^{n+\alpha})^{-1}\gamma_{j}^{n+\alpha}=\Gamma(2-\beta_{j}^{n+\alpha})^{-1}$. This fact, together with $(\ref{25})$ provides
      \begin{equation*}
       u^{n+\frac{1}{2}+\alpha}_{j}=u^{n+\alpha}_{j}-\frac{1}{2}k^{2-\beta_{j}^{n+\alpha}}\Gamma(2-\beta_{j}^{n+\alpha})^{-1}
      \underset{l=0}{\overset{n-\frac{1}{2}}\sum}a_{n,l+\frac{1}{2}}^{\alpha,\beta_{j}^{n+\alpha}}\delta_{t}u_{j}^{l}
      +\frac{k}{2}\left[-\delta_{x}^{4}u_{j}^{n+\alpha}+\delta_{2x}^{4}u_{j}^{n+\alpha}+f_{j}^{n+\alpha}\right]
       \end{equation*}
      \begin{equation}\label{26}
      -\frac{k}{2}J_{j}^{\beta_{j}^{n+\alpha}}+O(k^{2}+kh^{4}).
      \end{equation}
      Setting $u_{0x}(x,t)=u(x,t)$ and applying the Taylor series expansion for the functions $u_{mx}$ ($m=0,1,2$) at the mesh points $(x_{j},t_{n})$, $(x_{j},t_{n+\frac{1}{2}+\alpha})$ and $(x_{j},t_{n+1+\alpha})$ with time step $\frac{k}{2}$ gives
      \begin{equation*}
       u_{mx,j}^{n+\frac{1}{2}+\alpha}=(1+2\alpha)u_{mx,j}^{n+\frac{1}{2}}-2\alpha u_{mx,j}^{n}+O(k^{2})=
         2\alpha u_{mx,j}^{n+1}+(1-2\alpha)u_{mx,j}^{n+\frac{1}{2}}+O(k^{2}),
       \end{equation*}
        \begin{equation}\label{27}
         u_{mx,j}^{n+1+\alpha}=(1+2\alpha)u_{mx,j}^{n+1}-2\alpha u_{mx,j}^{n+\frac{1}{2}}+O(k^{2}),\text{\,\,\,}
         u^{n+\alpha}_{mx,j}=2\alpha u^{n+\frac{1}{2}}_{mx,j}+(1-2\alpha)ku^{n}_{mx,j}+O(k^{2}).
      \end{equation}
      Combining $(\ref{27})$ and $(\ref{23})$-$(\ref{24})$ to get
       \begin{equation*}
       \delta_{mx}^{4}u_{j}^{n+\frac{1}{2}+\alpha}=(1+2\alpha)\delta_{mx}^{4}u_{j}^{n+\frac{1}{2}}-2\alpha\delta_{mx}^{4}u_{j}^{n}+O(k^{2}+h^{4})=
         2\alpha\delta_{mx}^{4}u_{j}^{n+1}+(1-2\alpha)\delta_{mx}^{4}u_{j}^{n+\frac{1}{2}}+O(k^{2}+h^{4}),
       \end{equation*}
        \begin{equation}\label{28}
         \delta_{mx}^{4}u_{j}^{n+1+\alpha}=(1+2\alpha)\delta_{mx}^{4}u_{j}^{n+1}-2\alpha\delta_{mx}^{4}u_{j}^{n+\frac{1}{2}}+O(k^{2}+h^{4})
         ,\text{\,\,}\delta_{mx}^{4}u^{n+\alpha}_{j}=2\alpha\delta_{mx}^{4}u_{j}^{n+\frac{1}{2}}+(1-2\alpha)\delta_{mx}^{4}u_{j}^{n}+O(k^{2}+h^{4}),
      \end{equation}
      where we set $\delta_{0x}^{4}u^{n+\alpha}_{j}=u^{n+\alpha}_{j}$, $\delta_{0x}^{4}u^{n+\frac{1}{2}}_{j}=u^{n+\frac{1}{2}}_{j}$ and $\delta_{0x}^{4}u^{n}_{j}=u^{n}_{j}$. For $m=0,1,2$, substituting the last equation of $(\ref{28})$ into $(\ref{26})$ and performing simple computations to obtain
       \begin{equation*}
       u^{n+\frac{1}{2}}_{j}-\alpha k(\delta_{2x}^{4}-\delta_{x}^{4})u^{n+\frac{1}{2}}_{j}= u^{n}-2^{-1}k^{2-\beta_{j}^{n+\alpha}}
       \Gamma(2-\beta_{j}^{n+\alpha})^{-1}\underset{l=0}{\overset{n-\frac{1}{2}}\sum}a_{n,l+\frac{1}{2}}^{\alpha,\beta_{j}^{n+\alpha}}
       \delta_{t}u_{j}^{l}+
       \end{equation*}
      \begin{equation}\label{29}
      (\frac{1}{2}-\alpha)k(\delta_{2x}^{4}-\delta_{x}^{4})u^{n}_{j}+\frac{k}{2}f_{j}^{n+\alpha}-\frac{k}{2}J_{j}^{\beta_{j}^{n+\alpha}}+O(k^{2}+k^{3}+kh^{4}),
      \text{\,\,\,\,\,for\,\,\,\,}n\geq1.
      \end{equation}
      For $n=0$, utilizing relation $(\ref{16a})$ and substituting $(\ref{15})$ into $(\ref{5})$ to get
      \begin{equation}\label{30a}
      u^{\frac{1}{2}+\alpha}_{j}=u^{\alpha}_{j}+\frac{k}{2}\left[-c\Delta_{0t}^{\beta_{j}^{\alpha}}u_{j}^{\alpha}-J_{j}^{\beta_{j}^{\alpha}}-
      u_{x,j}^{n+\alpha}+u_{2x,j}^{\alpha}+f_{j}^{\alpha}\right]+O(k^{2}).
      \end{equation}
      Since $(1-\beta_{j}^{\alpha})^{-1}\gamma^{\alpha}_{j}=\Gamma(2-\beta_{j}^{\alpha})^{-1}$, using $(\ref{16})$, equation $(\ref{30a})$ is equivalent to
      \begin{equation}\label{31a}
      u^{\frac{1}{2}+\alpha}_{j}=u^{\alpha}_{j}+\frac{k}{2}\left[-k^{1-\beta_{j}^{\alpha}}\Gamma(2-\beta_{j}^{\alpha})^{-1}
      a_{\frac{1}{2},\frac{1}{2}}^{\alpha,\beta_{j}^{\alpha}}\delta_{t}^{\alpha}u_{j}^{0}-J_{j}^{\beta_{j}^{\alpha}}-
      u_{x,j}^{\alpha}+u_{2x,j}^{\alpha}+f_{j}^{\alpha}\right]+O(k^{2}).
      \end{equation}
      Replacing $n$ and $s$ by $0$ into $(\ref{23})$-$(\ref{24})$ and $(\ref{27})$-$(\ref{28})$ and $m$ by $0,1,2$, into the first equations of $(\ref{27})$ and $(\ref{28})$, combining the obtained equations together with $(\ref{31a})$ and utilizing equality $\delta_{t}^{\alpha}u_{j}^{0}=\frac{2}{k}(u^{\alpha}_{j}-u_{j}^{0})$, it is not difficult to see that
      \begin{equation}\label{32a}
       u^{\frac{1}{2}}_{j}-\alpha k(\delta_{2x}^{4}-\delta_{x}^{4})u^{\frac{1}{2}}_{j}=u^{0}-\alpha k\theta_{0j}^{\alpha}\delta_{t}u_{j}^{0}+
      (\frac{1}{2}-\alpha)k(\delta_{2x}^{4}-\delta_{x}^{4})u^{0}_{j}+\frac{k}{2}f_{j}^{\alpha}-\frac{k}{2}J_{j}^{\beta_{j}^{\alpha}}+O(k^{2}+k^{3}+kh^{4}),
      \end{equation}
      where
      \begin{equation}\label{33a}
       \theta_{0j}^{\alpha}=k^{1-\beta_{j}^{\alpha}}\Gamma(2-\beta_{j}^{\alpha})^{-1}a_{\frac{1}{2},\frac{1}{2}}^{\alpha,\beta_{j}^{\alpha}}.
      \end{equation}
      Suppose $U^{n}=(U^{n}_{2},U^{n}_{3},...,U^{n}_{M-2})$ be the approximate solution vector at time level $n$ and $u^{n}=(u^{n}_{2},u^{n}_{3},...,u^{n}_{M-2})$ be the analytical one at time $t_{n}$. Truncating the error terms in both equations $(\ref{29})$ and $(\ref{32a})$, we obtain the first-step of the new approach
      \begin{equation}\label{34a}
       U^{\frac{1}{2}}_{j}-\alpha k(\delta_{2x}^{4}-\delta_{x}^{4})U^{\frac{1}{2}}_{j}=U^{0}-\alpha k\theta_{0j}^{\alpha}\delta_{t}U_{j}^{0}+
      (\frac{1}{2}-\alpha)k(\delta_{2x}^{4}-\delta_{x}^{4})U^{0}_{j}+\frac{k}{2}f_{j}^{\alpha},
      \end{equation}
      \begin{equation*}
       U^{n+\frac{1}{2}}_{j}-\alpha k(\delta_{2x}^{4}-\delta_{x}^{4})U^{n+\frac{1}{2}}_{j}=U^{n}-2^{-1}k^{2-\beta_{j}^{n+\alpha}}
       \Gamma(2-\beta_{j}^{n+\alpha})^{-1}\underset{l=0}{\overset{n-\frac{1}{2}}\sum}a_{n,l+\frac{1}{2}}^{\alpha,\beta_{j}^{n+\alpha}}
       \delta_{t}U_{j}^{l}+
       \end{equation*}
      \begin{equation}\label{30}
      (\frac{1}{2}-\alpha)k(\delta_{2x}^{4}-\delta_{x}^{4})U^{n}_{j}+\frac{k}{2}f_{j}^{n+\alpha},\text{\,\,\,\,\,for\,\,\,\,}n\geq1.
      \end{equation}
      In addition, setting
       \begin{equation}\label{36a}
      f^{n+s+\alpha}=(f^{n+s+\alpha}_{2},f^{n+s+\alpha}_{3},...,f^{n+s+\alpha}_{M-2}) \text{\,\,\,and\,\,\,} \theta_{l+\frac{1}{2}}^{n+s+\alpha}=(\theta_{l+\frac{1}{2},2}^{n+s+\alpha},\theta_{l+\frac{1}{2},3}^{n+s+\alpha},...,
      \theta_{l+\frac{1}{2},M-2}^{n+s+\alpha}),
      \end{equation}
      where
      \begin{equation}\label{35a}
      \theta_{l+\frac{1}{2},j}^{n+s+\alpha}=k^{1-\beta_{j}^{n+s+\alpha}}\Gamma(2-\beta_{j}^{n+s+\alpha})^{-1}
      a_{n+s,l+\frac{1}{2}}^{\alpha,\beta_{j}^{n+s+\alpha}},
      \end{equation}
      $s\in\{0,2^{-1},1\}$. Thus, equations $(\ref{34a})$ and $(\ref{35a})$ can be expressed in the matrix form as
      \begin{equation}\label{35aa}
       A_{0}U^{\frac{1}{2}}+k\theta_{0}^{\alpha}*\delta_{t}U^{0}=A_{1}U^{0}+\frac{k}{2}f^{\alpha},
      \end{equation}
      \begin{equation}\label{36aa}
       A_{0}U^{n+\frac{1}{2}}+\frac{k}{2}\underset{l=0}{\overset{n-\frac{1}{2}}\sum}\theta_{l+\frac{1}{2}}^{n+\alpha}*\delta_{t}U^{l}=
       A_{1}U^{n}+\frac{k}{2}f^{n+\alpha},
       \end{equation}
      where "*" denotes the componentwise usual multiplication between two vectors, $A_{0}$ and $A_{1}$ are two $(M-2)\times(M-2)$ "pentadiagonal"
      matrices defined by
      \begin{equation}\label{M1}
        A_{0}=\begin{bmatrix}
                a_{kh}^{0} & b_{kh}^{0} & c_{kh}^{0} & 0 & \cdots & \cdots & 0 \\
                d_{kh}^{0} & a_{kh}^{0} & a_{kh}^{0} & a_{kh}^{0} & 0 & & \vdots \\
                e_{kh}^{0} & d_{kh}^{0} & a_{kh}^{0} & b_{kh}^{0} & c_{kh}^{0} & \ddots & \vdots \\
                0 & \ddots & \ddots & \ddots & \ddots & \ddots & 0\\
                \vdots & \ddots & \ddots & \ddots & \ddots & \ddots & c_{kh}^{0}\\
                \vdots &  & \ddots & \ddots & d_{kh}^{0} & a_{kh}^{0} & b_{kh}^{0}\\
                0 & \cdots & \cdots & 0 & e_{kh}^{0} & d_{kh}^{0} & a_{kh}^{0} \\
              \end{bmatrix}
              \text{\,\,\,and\,\,\,}
         A_{1}=\begin{bmatrix}
                a_{kh}^{1} & b_{kh}^{1} & c_{kh}^{1} & 0 & \cdots & \cdots & 0 \\
                d_{kh}^{1} & a_{kh}^{1} & a_{kh}^{1} & a_{kh}^{1} & 0 & & \vdots \\
                e_{kh}^{1} & d_{kh}^{1} & a_{kh}^{1} & b_{kh}^{1} & c_{kh}^{1} & \ddots & \vdots \\
                0 & \ddots & \ddots & \ddots & \ddots & \ddots & 0\\
                \vdots & \ddots & \ddots & \ddots & \ddots & \ddots & c_{kh}^{1}\\
                \vdots &  & \ddots & \ddots & d_{kh}^{1} & a_{kh}^{1} & b_{kh}^{1}\\
                0 & \cdots & \cdots & 0 & e_{kh}^{1} & d_{kh}^{1} & a_{kh}^{1} \\
              \end{bmatrix},
      \end{equation}
      where
       \begin{equation*}
      a_{kh}^{0}=1+\frac{5}{2}\frac{\alpha k}{h^{2}},\text{\,\,}b_{kh}^{0}=\frac{2\alpha k}{3h}(1-2h^{-1}),\text{\,\,}c_{kh}^{0}=\frac{\alpha k}{12h}(-1+h^{-1}),\text{\,\,}d_{kh}^{0}=\frac{-2\alpha k}{3h}(1+2h^{-1}),\text{\,\,}e_{kh}^{0}=\frac{\alpha k}{12h}(1+h^{-1}),
       \end{equation*}
       \begin{equation*}
      a_{kh}^{1}=1-\frac{5(1-2\alpha)k}{4h^{2}},\text{\,\,}b_{kh}^{1}=\frac{(1-2\alpha)k}{3h}(-1+2h^{-1}),\text{\,\,}
      c_{kh}^{1}=\frac{(1-2\alpha)k}{24h}(1-h^{-1}),\text{\,\,}d_{kh}^{1}=\frac{(1-2\alpha)k}{3h}(1+2h^{-1}),
       \end{equation*}
        \begin{equation}\label{M2}
       e_{kh}^{1}=\frac{(-1+2\alpha)k}{24h}(1+h^{-1}),
       \end{equation}

      To complete the full description of the desired algorithm, we should develop the second-step of the method. Combining equations $(\ref{6})$, $(\ref{9})$ and $(\ref{17})$, direct calculations give
      \begin{equation*}
       u^{n+1+\alpha}_{j}=u^{n+\frac{1}{2}+\alpha}+\frac{k}{4}\left\{-\left(c\Delta_{0t}^{\beta_{j}^{n+1+\alpha}}u_{j}^{n+1+\alpha}+
       c\Delta_{0t}^{\beta_{j}^{n+\frac{1}{2}+\alpha}}u_{j}^{n+\frac{1}{2}+\alpha}\right)-(u_{x,j}^{n+1+\alpha}+u_{x,j}^{n+\frac{1}{2}+\alpha})+\right.
       \end{equation*}
      \begin{equation}\label{31}
      \left.(u_{2x,j}^{n+1+\alpha}+u_{2x,j}^{n+\frac{1}{2}+\alpha})+(f_{j}^{n+1+\alpha}+f_{j}^{n+\frac{1}{2}+\alpha})\right\}-
      \frac{k}{4}(J_{j}^{\beta_{j}^{n+1+\alpha}}+I_{j}^{\beta_{j}^{n+\frac{1}{2}+\alpha}})+O(k^{3}).
      \end{equation}
      Since for $s=\frac{1}{2},1$, $\gamma_{j}^{n+s+\alpha}=\Gamma(1-\beta_{j}^{n+s+\alpha})^{-1}$ so, $(1-\beta_{j}^{n+s+\alpha})^{-1}\gamma_{j}^{n+s+\alpha}=\Gamma(2-\beta_{j}^{n+s+\alpha})^{-1}$. For $r=\frac{1}{2}$ and $m=0,1,2$, plugging equations $(\ref{10})$, $(\ref{18})$, $(\ref{27})$, $(\ref{28})$ and $(\ref{31})$, straightforward computations result in
      \begin{equation*}
       u^{n+1}_{j}=u^{n+\frac{1}{2}}_{j}+\frac{k}{4}\left\{-\frac{k^{1-\beta_{j}^{n+1+\alpha}}}{\Gamma(2-\beta_{j}^{n+1+\alpha})}
       \underset{l=0}{\overset{n+\frac{1}{2}}\sum}a_{n+1,l+\frac{1}{2}}^{\alpha,\beta_{j}^{n+1+\alpha}}
       \delta_{t}u_{j}^{l}-\frac{k^{1-\beta_{j}^{n+\frac{1}{2}+\alpha}}}{\Gamma(2-\beta_{j}^{n+\frac{1}{2}+\alpha})}
       \underset{l=0}{\overset{n}\sum}a_{n+\frac{1}{2},l+\frac{1}{2}}^{\alpha,\beta_{j}^{n+\frac{1}{2}+\alpha}}\delta_{t}u_{j}^{l}\right.
       \end{equation*}
      \begin{equation*}
       -\left[(1+2\alpha)\delta_{x}^{4}u_{j}^{n+1}-2\alpha\delta_{x}^{4}u_{j}^{n+\frac{1}{2}}+2\alpha\delta_{2x}^{4}u_{j}^{n+1}+
       \left.(1-2\alpha)\delta_{2x}^{4}u_{j}^{n+\frac{1}{2}}+f_{j}^{n+1+\alpha}+f_{j}^{n+\frac{1}{2}+\alpha}+O(k^{2}+h^{4})\right]\right\}
      \end{equation*}
       \begin{equation*}
      -\frac{k}{4}(J_{j}^{\beta_{j}^{n+1+\alpha}}+I_{j}^{\beta_{j}^{n+\frac{1}{2}+\alpha}})+O(k^{2}+k^{3}),
      \end{equation*}
      which is equivalent to
       \begin{equation*}
       u^{n+1}_{j}-\frac{1+4\alpha}{4}k(\delta_{2x}^{4}-\delta_{x}^{4})u^{n+1}_{j}+\frac{k}{4}\left[\underset{l=0}{\overset{n+\frac{1}{2}}\sum}
       \theta_{l+\frac{1}{2},j}^{n+1+\alpha}\delta_{t}u_{j}^{l}+\underset{l=0}{\overset{n}\sum}\theta_{l+\frac{1}{2},j}^{n+\frac{1}{2}+\alpha}
       \delta_{t}u_{j}^{l}\right]=u^{n+\frac{1}{2}}_{j}+\frac{1-4\alpha}{4}k(\delta_{2x}^{4}-\delta_{x}^{4})u^{n+\frac{1}{2}}_{j}
       \end{equation*}
      \begin{equation}\label{32}
       \frac{k}{4}(f_{j}^{n+1+\alpha}+f_{j}^{n+\frac{1}{2}+\alpha})-\frac{k}{4}(J_{j}^{\beta_{j}^{n+1+\alpha}}+
       I_{j}^{\beta_{j}^{n+\frac{1}{2}+\alpha}})+O(k^{2}+k^{3}+kh^{4}),
      \end{equation}
      where $\theta_{l+\frac{1}{2},j}^{n+s+\alpha}$ is defined by $(\ref{35a})$. Omitting the error terms $\frac{k}{4}(J_{j}^{\beta_{j}^{n+1+\alpha}}+I_{j}^{\beta_{j}^{n+\frac{1}{2}+\alpha}})+O(k^{2}+k^{3}+kh^{4})$, equation
      $(\ref{32})$ can be approximated as
      \begin{equation*}
       U^{n+1}_{j}-\frac{1+4\alpha}{4}k(\delta_{2x}^{4}-\delta_{x}^{4})U^{n+1}_{j}+\frac{k}{4}\left[\underset{l=0}{\overset{n+\frac{1}{2}}\sum}
       \theta_{l+\frac{1}{2},j}^{n+1+\alpha}\delta_{t}U_{j}^{l}+\underset{l=0}{\overset{n}\sum}\theta_{l+\frac{1}{2},j}^{n+\frac{1}{2}+\alpha}
       \delta_{t}U_{j}^{l}\right]=U^{n+\frac{1}{2}}_{j}+\frac{1-4\alpha}{4}k(\delta_{2x}^{4}-\delta_{x}^{4})U^{n+\frac{1}{2}}_{j}
       \end{equation*}
      \begin{equation}\label{34}
       \frac{k}{4}(f_{j}^{n+1+\alpha}+f_{j}^{n+\frac{1}{2}+\alpha}),\text{\,\,\,\,\,\,\,\,\,for\,\,\,}j=2,3,...M-2.
      \end{equation}
      We introduce the following "pentadiagonal" matrices $A$ and $A_{2}$ of size $(M-2)\times(M-2)$ defined as
            \begin{equation}\label{M3}
        A=\begin{bmatrix}
                a_{kh} & b_{kh} & c_{kh} & 0 & \cdots & \cdots & 0 \\
                d_{kh} & a_{kh} & a_{kh} & a_{kh} & 0 & & \vdots \\
                e_{kh} & d_{kh} & a_{kh} & b_{kh} & c_{kh} & \ddots & \vdots \\
                0 & \ddots & \ddots & \ddots & \ddots & \ddots & 0\\
                \vdots & \ddots & \ddots & \ddots & \ddots & \ddots & c_{kh}\\
                \vdots &  & \ddots & \ddots & d_{kh} & a_{kh} & b_{kh}\\
                0 & \cdots & \cdots & 0 & e_{kh} & d_{kh} & a_{kh} \\
              \end{bmatrix}
              \text{\,\,\,and\,\,\,}
         A_{2}=\begin{bmatrix}
                a_{kh}^{2} & b_{kh}^{2} & c_{kh}^{2} & 0 & \cdots & \cdots & 0 \\
                d_{kh}^{2} & a_{kh}^{2} & a_{kh}^{2} & a_{kh}^{2} & 0 & & \vdots \\
                e_{kh}^{2} & d_{kh}^{2} & a_{kh}^{2} & b_{kh}^{2} & c_{kh}^{2} & \ddots & \vdots \\
                0 & \ddots & \ddots & \ddots & \ddots & \ddots & 0\\
                \vdots & \ddots & \ddots & \ddots & \ddots & \ddots & c_{kh}^{2}\\
                \vdots &  & \ddots & \ddots & d_{kh}^{2} & a_{kh}^{2} & b_{kh}^{2}\\
                0 & \cdots & \cdots & 0 & e_{kh}^{2} & d_{kh}^{2} & a_{kh}^{2} \\
              \end{bmatrix},
      \end{equation}
      where
       \begin{equation*}
      a_{kh}=1+\frac{5(1+4\alpha)k}{8h^{2}},\text{\,\,}b_{kh}=\frac{(1+4\alpha)k}{6h}(1-2h^{-1}),\text{\,\,}c_{kh}=\frac{(1+4\alpha) k}{48h}(-1+h^{-1}),\text{\,\,}d_{kh}=\frac{-(1+4\alpha)k}{6h}(1+2h^{-1}),
       \end{equation*}
       \begin{equation*}
      e_{kh}=\frac{(1+4\alpha)k}{48h}(1+h^{-1}),\text{\,\,}a_{kh}^{2}=1-\frac{5(1-4\alpha)k}{8h^{2}},\text{\,\,}
      b_{kh}^{2}=\frac{(1-4\alpha)k}{6h}(-1+2h^{-1}),\text{\,\,} c_{kh}^{2}=\frac{(1-4\alpha)k}{48h}(1-h^{-1}),
       \end{equation*}
        \begin{equation}\label{M4}
       d_{kh}^{2}=\frac{(1-4\alpha)k}{6h}(1+2h^{-1}),\text{\,\,}e_{kh}^{2}=\frac{(-1+4\alpha)k}{48h}(1+h^{-1}).
       \end{equation}
       A combination of equations $(\ref{36a})$, $(\ref{35a})$ and $(\ref{34})$ provides the following matrix form
       \begin{equation}\label{35}
       AU^{n+1}+\frac{k}{4(1+2\alpha)}\left[\underset{l=0}{\overset{n+\frac{1}{2}}\sum}
       \theta_{l+\frac{1}{2}}^{n+1+\alpha}*\delta_{t}U_{j}^{l}+\underset{l=0}{\overset{n}\sum}\theta_{l+\frac{1}{2}}^{n+\frac{1}{2}+\alpha}*
       \delta_{t}U_{j}^{l}\right]=A_{2}U^{n+\frac{1}{2}}+\frac{k}{4}(f^{n+1+\alpha}+f^{n+\frac{1}{2}+\alpha}),
       \end{equation}
       where "*" represents the componentwise usual multiplication between two vectors. We recall that the summation index $"l"$ varies in the range
       $l=0,\frac{1}{2},1,\frac{3}{2},2,...,n,n+\frac{1}{2}$. Furthermore, equation $(\ref{34})$ denotes the second-step of the proposed two-step fourth-order modified explicit Euler/Crank-Nicolson numerical scheme in a computed solution of the initial-boundary value problem $(\ref{1})$-$(\ref{3})$.\\

       An assembly of equations $(\ref{35aa})$, $(\ref{36aa})$ and $(\ref{35})$ provides the new algorithm for solving the problem $(\ref{1})$-$(\ref{3})$,
       that is, for $n=1,2,...,N-1$,
       \begin{equation}\label{s1}
       A_{0}U^{\frac{1}{2}}+k\theta_{0}^{\alpha}*\delta_{t}U^{0}=A_{1}U^{0}+\frac{k}{2}f^{\alpha},
      \end{equation}
      \begin{equation}\label{s2}
       AU^{1}+\frac{k}{4(1+2\alpha)}\left[\underset{l=0}{\overset{\frac{1}{2}}\sum}
       \theta_{l+\frac{1}{2}}^{1+\alpha}*\delta_{t}U_{j}^{l}+\theta_{\frac{1}{2}}^{\frac{1}{2}+\alpha}*
       \delta_{t}U_{j}^{0}\right]=A_{2}U^{\frac{1}{2}}+\frac{k}{4}(f^{1+\alpha}+f^{\frac{1}{2}+\alpha}),
       \end{equation}
      \begin{equation}\label{s3}
       A_{0}U^{n+\frac{1}{2}}+\frac{k}{2}\underset{l=0}{\overset{n-\frac{1}{2}}\sum}\theta_{l+\frac{1}{2}}^{n+\alpha}*\delta_{t}U^{l}=
       A_{1}U^{n}+\frac{k}{2}f^{n+\alpha},
       \end{equation}
       \begin{equation}\label{s4}
       AU^{n+1}+\frac{k}{4(1+2\alpha)}\left[\underset{l=0}{\overset{n+\frac{1}{2}}\sum}
       \theta_{l+\frac{1}{2}}^{n+1+\alpha}*\delta_{t}U_{j}^{l}+\underset{l=0}{\overset{n}\sum}\theta_{l+\frac{1}{2}}^{n+\frac{1}{2}+\alpha}*
       \delta_{t}U_{j}^{l}\right]=A_{2}U^{n+\frac{1}{2}}+\frac{k}{4}(f^{n+1+\alpha}+f^{n+\frac{1}{2}+\alpha}),
       \end{equation}
        with initial and boundary conditions
       \begin{equation}\label{s5}
       U_{j}^{0}=u_{j}^{0},\text{\,\,}j=0,1,...,M;\text{\,\,\,}U_{0}^{n}=g_{1}^{n}\text{\,\,\,and\,\,\,}U_{M}^{n}=g_{2}^{n},
       \text{\,\,\,for\,\,\,}n=0,1,...,N,
       \end{equation}
       where the matrices $A_{0}$, $A_{1}$, $A$ and $A_{2}$ are given by relations $(\ref{M1})$-$(\ref{M2})$ and $(\ref{M3})$-$(\ref{M4})$. To start the algorithm we should set $U_{1}^{n}=U_{0}^{n}$ and $U_{M-1}^{n}=U_{M}^{n}$, for $n=0,1,...,N$. However, the terms $U_{1}^{n}$ and $U_{M-1}^{n}$ can be obtained by using any one-step fractional approach such as the method analyzed in \cite{9zzy}.\\

        It's worth noticing that the coefficients of the pentadiagonal matrices $A$ and $A_{i}$, for $i=0,1,2$, come from the entries of a $g$-Toeplitz matrix ($T_{n,g}(f))$ or $g$-circulant matrix ($C_{n,g}(f))$, generated by a Lebesgue integrable function $f$ defined over the domain $(-\pi,\pi)$, where $g$ is a nonnegative integer. Specifically, these matrices are called band Toeplitz matrices which represent a subclass of $g$-Toeplitz structures. For more details about $g$-Toeplitz, $g$-circulant and band Toeplitz matrices, we refer the readers to \cite{nss,enss,eng1,eng2} and references therein. Furthermore, since the matrices $A$ and $A_{0}$ are not symmetric, at time level $n$ or $n+\frac{1}{2}$, each system of linear equations $(\ref{s1})$-$(\ref{s4})$ can be efficiently solved using the Preconditioned Generalized Minimal Residual Algorithm \cite{gmres}.

     \section{Stability analysis and error estimates of the proposed two-step approach $(\ref{s1})$-$(\ref{s5})$}\label{sec3}
    In this section we analyze both unconditional stability and convergence order of the new approach $(\ref{s1})$-$(\ref{s5})$ applied to
    the initial-boundary value problem $(\ref{1})$-$(\ref{3})$. In this study we assume that the function $\beta(x,t):=\overline{\beta}$ is constant
    and the parameter $\alpha$ satisfies $0<\alpha<2^{-1}$. These restrictions play crucial roles in the proof of some intermediate results (namely
    Lemma $\ref{l2}$) and the main result of this paper (Theorem $\ref{t}$). Furthermore, the following Lemmas are important in the analysis of
    stability and error estimates of the proposed formulation $(\ref{s1})$-$(\ref{s5})$ for solving the time variable-order fractional mobile-immobile
    equation $(\ref{1})$ subjects to suitable initial condition $(\ref{2})$ and boundary one $(\ref{3})$.

    \begin{lemma}\label{l1}
   For any $\overline{\beta}\in(0,1)$ and $\alpha\in(0,\frac{1}{2})$. Set $D=[L_{0},L]\times[0,T]$ and consider a function
    $u\in\mathcal{C}^{6,3}(D):=\mathcal{C}^{6,3}_{D}$, thus it holds
   \begin{equation}\label{36}
   \underset{0\leq n\leq N-1}{\max}\|cD_{0t}^{\beta_{j}^{n+s+\alpha}}u^{n+s+\alpha}-c\Delta_{0t}^{\beta_{j}^{n+s+\alpha}}u^{n+s+\alpha}\|_{L^{2}}
   \leq C_{s}k^{2-\overline{\beta}}
   \end{equation}
   where $s=\frac{1}{2},1$, $\beta(x,t)=\overline{\beta}$ is constant, $cD_{0t}^{\lambda}u^{n+s+\alpha}$ and
   $c\Delta_{0t}^{\lambda}u^{i+\frac{1}{2}+\alpha}$, for $0\leq n\leq N-1$, are defined by $(\ref{9})$, $(\ref{10})$ and
   $(\ref{15})$-$(\ref{18})$, and $C_{s}$ for $s\in\{\frac{1}{2},1\}$ are positive constants independent of the mesh size $h$ and time step $k$.
   \end{lemma}

   \begin{proof}
    Since $0<\overline{\beta}<1$ is constant and $0<\alpha<\frac{1}{2}$, the proof of this Lemma can be found in \cite{entlfcdr}.
   \end{proof}

   \begin{lemma}\label{l2}\cite{entlfcdr}
   Assume that $0<\overline{\beta}<\frac{2}{3}$ and consider the generalized sequences $\left(a_{n+s,l}^{\alpha,\overline{\beta}}\right)_{2^{-1}
   \leq l\leq n+s}$, where $s=\frac{1}{2},1$, defined by equations $(\ref{12})$-$(\ref{13})$ (for $s=2^{-1}$) and $(\ref{19})$-$(\ref{20})$
   (for $s=1$), thus
   \begin{equation}\label{37}
    a_{n+s,l}^{\alpha,\overline{\beta}}<a_{n+s,l+\frac{1}{2}}^{\alpha,\overline{\beta}},\text{\,\,\,\,for\,\,\,\,}l=\frac{1}{2},1,\frac{3}{2},2,...,n,
    \text{\,\,\,(resp.\,\,\,})n+\frac{1}{2}.
   \end{equation}
   Furthermore,
    \begin{equation}\label{38}
    a_{n+s,l}^{\alpha,\overline{\beta}}>\frac{(2-3\overline{\beta})(1-\overline{\beta})}{2(2-\overline{\beta})}(n+s+\alpha-l)^{-\overline{\beta}},
    \text{\,\,\,\,for\,\,\,\,}l=\frac{1}{2},1,\frac{3}{2},2,...,n,\text{\,\,\,(resp.,\,\,\,}n+\frac{1}{2}).
   \end{equation}
   \end{lemma}

   \begin{lemma}\label{l3}
    Let $(a_{n+s,l}^{\alpha,\overline{\beta}})_{l\leq n+s}$ be the generalized sequences defined by relations $(\ref{12})$-$(\ref{13})$
    and $(\ref{19})$-$(\ref{20})$. For every mesh function $u(\cdot,\cdot)$ defined on the grid space $\mathcal{U}_{hk}=\{u_{j}^{n},
    \text{\,\,}0\leq j\leq M\text{\,\,\,and\,\,\,}n=0,1,2,...,N\}$, setting $W_{j}^{\overline{\beta},l}=\underset{r=0}{\overset{l}\sum}
    a_{n+s,r+\frac{1}{2}}^{\alpha,\overline{\beta}}\delta_{t}u_{j}^{r}$, the following estimates hold for $s=\frac{1}{2},1$,
    \begin{equation*}
    u_{j}^{n+s}(c\Delta_{0t}^{\overline{\beta}}u_{j}^{n+s+\alpha})=\frac{1}{2}c\Delta_{0t}^{\overline{\beta}}(u_{j}^{n+s+\alpha})^{2}+
    \frac{k^{2-\overline{\beta}}}{4\Gamma(2-\overline{\beta})}\left\{(a_{n+s,n+s}^{\alpha,\overline{\beta}})^{-1}
    \left(W_{j}^{\overline{\beta},n+s-\frac{1}{2}}\right)^{2}+\right.
   \end{equation*}
    \begin{equation}\label{39a}
    \left.\left[a_{n+s,\frac{1}{2}}^{\alpha,\overline{\beta}}-(a_{n+s,1}^{\alpha,\overline{\beta}})^{-1}
    (a_{n+s,\frac{1}{2}}^{\alpha,\overline{\beta}})^{2}\right]
    \left(\delta_{t}u_{j}^{0}\right)^{2}+\underset{l=\frac{1}{2}}{\overset{n+s-1}\sum}\left[(a_{n+s,l+\frac{1}{2}}^{\alpha,\overline{\beta}})^{-1}
    -(a_{n+s,l+1}^{\alpha,\overline{\beta}})^{-1}\right]\left(W_{j}^{\overline{\beta},l}\right)^{2}\right\},
   \end{equation}
     Furthermore,
    \begin{equation}\label{39}
    u_{j}^{n+s}(c\Delta_{0t}^{\overline{\beta}}u_{j}^{n+s+\alpha})\geq\frac{1}{2}c\Delta_{0t}^{\overline{\beta}}(u_{j}^{n+s+\alpha})^{2}.
   \end{equation}
   \end{lemma}

   \begin{proof}
   The proof of $(\ref{39a})$ is obtained by replacing $\lambda$ with $\overline{\beta}$ in the proof of Lemma $3.3$ established in \cite{entlfcdr}.
   The proof of $(\ref{39})$ is obvious since $a_{n+s,l}^{\alpha,\overline{\beta}}<a_{n+s,l+\frac{1}{2}}^{\alpha,\overline{\beta}}$, for
   $s=\frac{1}{2},1$, and $l=\frac{1}{2},1,\frac{3}{2},2,...,n$ (resp., $n+\frac{1}{2}$). In addition, it is easy to see that
   $a_{n+s,\frac{1}{2}}^{\alpha,\overline{\beta}}-(a_{n+s,1}^{\alpha,\overline{\beta}})^{-1}(a_{n+s,\frac{1}{2}}^{\alpha,\overline{\beta}})^{2}\geq0$.
   \end{proof}

   \begin{lemma}\label{l4}
     Given $(a_{n+s,l}^{\alpha,\overline{\beta}})_{l}$ be the generalized sequences defined by equations $(\ref{12})$-$(\ref{13})$
    and $(\ref{19})$-$(\ref{20})$, for any grid function $v(\cdot,\cdot)$ defined on the grid space $\mathcal{U}_{hk}$, it holds
    \begin{equation*}
    \underset{l=l_{0}}{\overset{m}\sum}a_{n+s,l+\frac{1}{2}}^{\alpha,\overline{\beta}}[(v_{j}^{l+\frac{1}{2}})^{2}-(v_{j}^{l})^{2}]=
    a_{n+s,m+\frac{1}{2}}^{\alpha,\overline{\beta}}(v_{j}^{m+\frac{1}{2}})^{2}-a_{n+s,l_{0}+\frac{1}{2}}^{\alpha,\overline{\beta}}(v_{j}^{l_{0}})^{2}+
   \end{equation*}
   \begin{equation}\label{42}
    \underset{l=l_{0}}{\overset{m-\frac{1}{2}}\sum}[a_{n+s,l+\frac{1}{2}}^{\alpha,\overline{\beta}}
     -a_{n+s,l+1}^{\alpha,\overline{\beta}}](v_{j}^{l+\frac{1}{2}})^{2},
   \end{equation}
    for $m\in\{n,n+\frac{1}{2}\}$ and $l=l_{0},l_{0}+\frac{1}{2},l_{0}+1,l_{0}+\frac{3}{2},...,m$, where $l_{0}$ is a nonnegative integer satisfying
    $l_{0}\leq m$.
   \end{lemma}

    \begin{proof}
     Expanding the left side of this equality and rearranging terms to obtain the result.
    \end{proof}

    \begin{lemma}\label{l5}
      Consider the following linear operators
     \begin{equation}\label{44}
      L_{h}u_{j}^{q}=(\delta_{2x}^{4}-\delta_{x}^{4})u_{j}^{q}\text{\,\,\,and\,\,\,}Lu_{j}^{q}=[u_{2x}-u_{x}]|_{(x_{j},t_{q})},
    \end{equation}
    for $j=2,3,...,M-2$, where $q$ is any nonnegative rational number. If $u_{0}^{q}=v_{0}^{q}=0$, $u_{1}^{q}=v_{1}^{q}=0$, $u_{M-1}^{q}=v_{M-1}^{q}=0$
    and $u_{M}^{q}=v_{M}^{q}=0$, so it holds
   \begin{equation}\label{45}
    |\left(L_{h}u_{j}^{q},v^{d}\right)|\leq \frac{4}{3}\|\delta_{x}v^{d}\|_{2}[\|\delta_{x}u^{q}\|_{2}+\|u^{q}\|_{2}]\text{\,\,\,and\,\,\,}
    |\left(L_{h}u_{j}^{q},v^{d}\right)|\leq C_{p}\|\delta_{x}v^{d}\|_{2}\|\delta_{x}u^{q}\|_{2},
   \end{equation}
   where $C_{p}$ is a positive constant independent of the time step $k$ and space step $h$.
   \end{lemma}

   \begin{proof}
    In equations $(\ref{23a})$ and $(\ref{24a})$, replacing $n+s+\alpha$ by $q$ to obtain
    \begin{equation*}
       \delta_{2x}^{4}u_{j}^{q}=\frac{1}{12h^{2}}\left[-u_{j+2}^{q}+16u_{j+1}^{q}-30u_{j}^{q}+16u_{j-1}^{q}-u_{j-2}^{q}\right],
       \end{equation*}
        \begin{equation*}
       \delta_{x}^{4}u_{j}^{q}=\frac{1}{12h}\left[-u_{j+2}^{q}+8u_{j+1}^{q}-8u_{j-1}^{q}+u_{j-2}^{q}\right].
       \end{equation*}
     Since $\delta_{x}u_{j}^{q}=\frac{u_{j}^{q}-u_{j-1}^{q}}{h}$, direct calculations provide
     \begin{equation}\label{n1}
       \delta_{2x}^{4}u_{j}^{q}=\frac{1}{12h^{2}}\left[-(\delta_{x}u_{j-\frac{1}{2}}^{q}-\delta_{x}u_{j-\frac{3}{2}}^{q})+
       14(\delta_{x}u_{j+\frac{1}{2}}^{q}-\delta_{x}u_{j-\frac{1}{2}}^{q})-(\delta_{x}u_{j+\frac{3}{2}}^{q}-\delta_{x}u_{j+\frac{1}{2}}^{q})\right],
       \end{equation}
        \begin{equation}\label{n2}
        \delta_{x}^{4}u_{j}^{q}=\frac{1}{12h}\left[-\delta_{x}u_{j-\frac{3}{2}}^{q}+
       7\delta_{x}u_{j-\frac{1}{2}}^{q}+7\delta_{x}u_{j+\frac{1}{2}}^{q}-\delta_{x}u_{j+\frac{3}{2}}^{q}\right].
       \end{equation}
       Performing straightforward computations, it is not difficult to show that
       \begin{equation*}
    \frac{-1}{h}\underset{j=2}{\overset{M-2}\sum}(\delta_{x}u_{j-\frac{1}{2}}^{q}-\delta_{x}u_{j-\frac{3}{2}}^{q})v_{j}^{p}=
    \frac{1}{h}(\delta_{x}u_{\frac{1}{2}}^{q}v_{2}^{p}-\delta_{x}u_{M-\frac{5}{2}}^{q}v_{M-2}^{p})+\underset{j=2}{\overset{M-3}\sum}
    \delta_{x}u_{j-\frac{1}{2}}^{q}\delta_{x}v_{j+\frac{1}{2}}^{q}.
    \end{equation*}
      Utilizing assumption $v_{1}^{q}=0$ and $v_{M-1}^{q}=0$, this becomes
     \begin{equation}\label{n3}
     \frac{-1}{h}\underset{j=2}{\overset{M-2}\sum}(\delta_{x}u_{j-\frac{1}{2}}^{q}-\delta_{x}u_{j-\frac{3}{2}}^{q})v_{j}^{p}=
     \underset{j=1}{\overset{M-2}\sum}\delta_{x}u_{j-\frac{1}{2}}^{q}\delta_{x}v_{j+\frac{1}{2}}^{q}.
     \end{equation}
     \begin{equation}\label{n4}
     \frac{-1}{h}\underset{j=2}{\overset{M-2}\sum}(\delta_{x}u_{j+\frac{3}{2}}^{q}-\delta_{x}u_{j+\frac{1}{2}}^{q})v_{j}^{p}=
    \frac{1}{h}(\delta_{x}u_{\frac{5}{2}}^{q}v_{2}^{p}-\delta_{x}u_{M-\frac{1}{2}}^{q}v_{M-2}^{p})+\underset{j=2}{\overset{M-3}\sum}
    \delta_{x}u_{j+\frac{3}{2}}^{q}\delta_{x}v_{j+\frac{1}{2}}^{q}=\underset{j=1}{\overset{M-2}\sum}\delta_{x}u_{j+\frac{3}{2}}^{q}
    \delta_{x}v_{j+\frac{1}{2}}^{q}.
     \end{equation}
    The last equality follows from the assumption $v_{1}^{q}=0$ and $v_{M-1}^{q}=0$. Analogously, one easily shows that
    \begin{equation}\label{n5}
     \frac{14}{h}\underset{j=2}{\overset{M-2}\sum}(\delta_{x}u_{j+\frac{1}{2}}^{q}-\delta_{x}u_{j-\frac{1}{2}}^{q})v_{j}^{p}=
     -14\underset{j=1}{\overset{M-2}\sum}\delta_{x}u_{j+\frac{1}{2}}^{q}\delta_{x}v_{j+\frac{1}{2}}^{q}.
     \end{equation}
     Plugging equations $(\ref{n1})$ and $(\ref{n3})$-$(\ref{n5})$ yields
     \begin{equation*}
    \underset{j=2}{\overset{M-2}\sum}(\delta_{2x}^{4}u_{j}^{q})v_{j}^{q}=\frac{1}{12}\left[\underset{j=1}{\overset{M-2}\sum}\delta_{x}u_{j-\frac{1}{2}}^{q}
    \delta_{x}v_{j+\frac{1}{2}}^{q}+\underset{j=1}{\overset{M-2}\sum}\delta_{x}u_{j+\frac{3}{2}}^{q}\delta_{x}v_{j+\frac{1}{2}}^{q}-14
    \underset{j=1}{\overset{M-2}\sum}\delta_{x}u_{j+\frac{1}{2}}^{q}\delta_{x}v_{j+\frac{1}{2}}^{q}\right].
    \end{equation*}
    Multiplying both sides of this equation by $h$, applying the H\"{o}lder and Cauchy-Schwarz inequalities, using the definition of $L^{2}$-norm
    and the scalar product $\left(\cdot,\cdot\right)$, this results in
    \begin{equation*}
    \left(\delta_{2x}^{4}u^{q},v^{q}\right)\leq\frac{1}{12}\left[\left(h^{2}\underset{j=1}{\overset{M-2}\sum}
    \left(\delta_{x}u_{j-\frac{1}{2}}^{q}\right)^{2}\underset{j=1}{\overset{M-2}\sum}
    \left(\delta_{x}v_{j+\frac{1}{2}}^{q}\right)^{2}\right)^{\frac{1}{2}}+\left(h^{2}\underset{j=1}{\overset{M-2}\sum}
    \left(\delta_{x}u_{j+\frac{3}{2}}^{q}\right)^{2}\underset{j=1}{\overset{M-2}\sum}
    \left(\delta_{x}v_{j+\frac{1}{2}}^{q}\right)^{2}\right)^{\frac{1}{2}}+\right.
    \end{equation*}
    \begin{equation}\label{n6}
     \left.14\left(h^{2}\underset{j=1}{\overset{M-2}\sum}\left(\delta_{x}u_{j+\frac{1}{2}}^{q}\right)^{2}\underset{j=1}{\overset{M-2}\sum}
     \left(\delta_{x}v_{j+\frac{1}{2}}^{q}\right)^{2}\right)^{\frac{1}{2}}\right]\leq\frac{16}{12}\|\delta_{x}u^{q}\|_{2}\|\delta_{x}v^{q}\|_{2}.
     \end{equation}
     In a similar manner, one easily proves that
     \begin{equation*}
     \underset{j=2}{\overset{M-2}\sum}\delta_{x}u_{j-\frac{3}{2}}^{q}v_{j}^{p}=
     -\underset{j=1}{\overset{M-3}\sum}u_{j}^{q}\delta_{x}v_{j+\frac{3}{2}}^{q},\text{\,\,\,\,\,\,\,}\underset{j=2}{\overset{M-2}\sum}
     \delta_{x}u_{j+\frac{3}{2}}^{q}v_{j}^{p}=-\underset{j=3}{\overset{M-2}\sum}u_{j}^{q}\delta_{x}v_{j-\frac{3}{2}}^{q},
     \end{equation*}
     \begin{equation}\label{n7}
     -7\underset{j=2}{\overset{M-2}\sum}\delta_{x}u_{j-\frac{1}{2}}^{q}v_{j}^{p}=
     7\underset{j=1}{\overset{M-2}\sum}u_{j}^{q}\delta_{x}v_{j+\frac{1}{2}}^{q}\text{\,\,\,and\,\,\,}-7\underset{j=2}{\overset{M-2}\sum}
     \delta_{x}u_{j+\frac{1}{2}}^{q}v_{j}^{p}=7\underset{j=2}{\overset{M-2}\sum}u_{j}^{q}\delta_{x}v_{j-\frac{1}{2}}^{q}.
     \end{equation}
     Combining equations $(\ref{n2})$ and $(\ref{n7})$, it is not hard to observe that
     \begin{equation*}
    \underset{j=2}{\overset{M-2}\sum}(\delta_{x}^{4}u_{j}^{q})v_{j}^{q}=\frac{1}{12}\left[-\underset{j=1}{\overset{M-3}\sum}\delta_{x}u_{j}^{q}
    \delta_{x}v_{j+\frac{3}{2}}^{q}-\underset{j=3}{\overset{M-2}\sum}\delta_{x}u_{j}^{q}\delta_{x}v_{j-\frac{3}{2}}^{q}+7
    \underset{j=1}{\overset{M-2}\sum}\delta_{x}u_{j}^{q}\delta_{x}v_{j+\frac{1}{2}}^{q}+7\underset{j=2}{\overset{M-2}\sum}
    \delta_{x}u_{j}^{q}\delta_{x}v_{j-\frac{1}{2}}^{q}\right].
    \end{equation*}
    Multiplying this equation by $h$, utilizing the H\"{o}lder and Cauchy-Schwarz inequalities together with the definitions of $L^{2}$-norm and
    scalar product $\left(\cdot,\cdot\right)$ to get
    \begin{equation}\label{n8}
    \left(\delta_{x}^{4}u^{q},v^{q}\right)\leq\frac{16}{12}\|u^{q}\|_{2}\|\delta_{x}v^{q}\|_{2}.
     \end{equation}
     A combination of $(\ref{44})$ and estimates $(\ref{n7})$-$(\ref{n8})$ gives
     \begin{equation*}
    \left|\left(L_{h}u^{q},v^{q}\right)\right|=\left|\left((\delta_{2x}^{4}-\delta_{x}^{4})u^{q},v^{q}\right)\right|\leq
    \left|\left(\delta_{2x}^{4}u^{q},v^{q}\right)\right|+\left|\left(\delta_{x}^{4}u^{q},v^{q}\right)\right|\leq\frac{4}{3}[\|u^{q}\|_{2}
    \|\delta_{x}v^{q}\|_{2}+\|\delta_{x}u^{q}\|_{2}\|\delta_{x}v^{q}\|_{2}].
    \end{equation*}
    This completes the proof of the first estimate in $(\ref{45})$. The proof of the second estimate in $(\ref{45})$ is obtained thanks to the
    Poincar\'{e}-Friedrich inequality.
    \end{proof}

    \begin{lemma}\label{l6}
     Suppose $u\in\mathcal{C}_{D}^{2,0}$, be a function defined on $D=[L_{0},L]\times[0,T]$, satisfying $u(L_{0},t)=u(L,t)=0$, for any $t\in[0,T]$.
     Let $U(t)=(U_{0}(t),U_{1}(t),...,U_{M}(t))$ be a grid function such that, $U_{j}(t)=u(x_{j},t)$, for $j=0,1,...,M$. So, it holds
     \begin{equation}\label{46}
     \left(-Lu(t),u(t)\right)\geq\|\delta_{x}u(t)\|_{2}^{2},\|\delta_{x}U(t)\|_{2}^{2} \text{\,\,\,and\,\,\,}
     \left(-L_{h}U(t),U(t)\right)\geq\frac{1}{2}\|\delta_{x}U(t)\|_{2}^{2},
    \end{equation}
   for every $t\in[0,T]$.
   \end{lemma}

   \begin{proof}
   We should show that the operator $-Lu=-u_{2x}+u_{x}$ satisfies: $\left(-Lu(t),u(t)\right)\geq\|\delta_{x}u(t)\|_{2}^{2}$ and then, use the
   discrete $L^{2}$-norm defined in relation $(\ref{dn})$ to conclude. The application of the Taylor series expansion using backward difference
   formulation gives: $u_{x}(t)=\delta_{x}u_{j-\frac{1}{2}}(t)+O(h)$ and $u_{2x}(x_{j},t)=\delta_{x}^{2}u_{j}(t)+O(h^{2})$. Using the conditions
   $u_{1}(t)=u_{M-1}(t)=0$, for every $t\in[0,T]$ together with the summation by parts and the equality $a(a-b)=\frac{1}{2}[(a-b)^{2}+a^{2}-b^{2}]$,
   for any real numbers $a$ and $b$, direct computations yield
    \begin{equation*}
    (-Lu(t),u(t))=-h\underset{j=2}{\overset{M-2}\sum}[(\delta_{x}^{2}u_{j}(t)+O(h^{2}))-(\delta_{x}u_{j-\frac{1}{2}}(t)+O(h))u_{j}(t)=
    -h\underset{j=2}{\overset{M-2}\sum}(\delta_{x}^{2}u_{j}(t))u_{j}(t)+O(h^{2})+
   \end{equation*}
   \begin{equation*}
    h\underset{j=2}{\overset{M-2}\sum}(\delta_{x}u_{j-\frac{1}{2}}(t))u_{j}(t)+O(h)=-\left[(\delta_{x}u_{M-\frac{3}{2}}(t))u_{M-2}(t)-
    (\delta_{x}u_{\frac{3}{2}}(t))u_{2}(t)-\underset{j=2}{\overset{M-3}\sum}(\delta_{x}u_{j+\frac{1}{2}}(t))^{2}\right]+O(h^{2})
   \end{equation*}
   \begin{equation*}
    +\frac{1}{2}\left[h^{2}\underset{j=2}{\overset{M-2}\sum}(\delta_{x}u_{j-\frac{1}{2}}(t))^{2}+(u_{M-2}(t))^{2}-(u_{1}(t))^{2}\right]+O(h)
    =h\underset{j=1}{\overset{M-2}\sum}(\delta_{x}u_{j+\frac{1}{2}}(t))^{2}+O(h^{2})+
   \end{equation*}
   \begin{equation}\label{91a}
    \frac{1}{2}h^{2}\underset{j=2}{\overset{M-1}\sum}(\delta_{x}u_{j-\frac{1}{2}}(t))^{2}+O(h).
   \end{equation}
   The last equality follows from $h\delta_{x}u_{M-\frac{3}{2}}(t)=-u_{M-2}(t)$ and $\delta_{x}u_{\frac{3}{2}}(t)=\frac{1}{h}u_{2}(t)$, since
    $u_{M-1}(t)=u_{1}(t)=0$. For small values of $h$, it holds
   \begin{equation}\label{91b}
    h\underset{j=1}{\overset{M-2}\sum}(\delta_{x}u_{j+\frac{1}{2}}(t))^{2}+O(h^{2})=\|\delta_{x}u(t)\|_{2}^{2}+O(h^{2})
    \approx h\underset{j=2}{\overset{M-2}\sum}(u_{x,j}(t))^{2}=\|u_{x}(t)\|_{2}^{2}.
   \end{equation}
   Substituting approximation $(\ref{91b})$ into relation $(\ref{91a})$ to obtain
    \begin{equation*}
    (-Lu(t),u(t))\geq h\underset{j=2}{\overset{M-2}\sum}(u_{x,j}(t))^{2}=\|u_{x}(t)\|_{L^{2}}^{2}.
   \end{equation*}
   Furthermore, since $u_{j}(t)=U_{j}(t)$, for $j=0,1,...,M$, for $h$ sufficiently small, neglecting the infinitesimal terms $O(h^{2})$ and $O(h)$, 
   equation $(\ref{91a})$ implies
   \begin{equation*}
    (-Lu(t),u(t))\geq h\underset{j=1}{\overset{M-2}\sum}(\delta_{x}U_{j+\frac{1}{2}}(t))^{2}+\frac{1}{2}h^{2}\underset{j=2}{\overset{M-1}
    \sum}(\delta_{x}u_{j-\frac{1}{2}}(t))^{2}\geq\|\delta_{x}U(t)\|_{L^{2}}^{2}.
   \end{equation*}
   This ends the proof of the first estimate in Lemma $\ref{l6}$. Since $L_{h}=L-(L-L_{h})$ and $(L-L_{h})U_{j}(t)=O(h^{4})$, the proof of 
   Lemma $\ref{l6}$ is completed thanks to the first estimate in $(\ref{46})$ and the definition of the scalar product given by $(\ref{sp})$.

   \end{proof}

   Armed with Lemmas $\ref{l1}$-$\ref{l6}$, we should state and prove the main result of this work (Theorem $\ref{t}$).

   \begin{theorem}\label{t} (Unconditional stability and Error estimates).
   Suppose $U$ be the approximate solution provided by the proposed approach $(\ref{s1})$-$(\ref{s5})$ and let $u$ be the analytical solution
   of the initial-boundary value problem $(\ref{1})$-$(\ref{3})$. let $\beta(x,t)=\overline{\beta}\in(0,\frac{2}{3})$, for any $(x,t)\in D$,
   be a positive constant function, $0<\alpha<\frac{1}{2}$ be a parameter and let $(a_{\cdot,l}^{\alpha,\overline{\beta}})_{l}$ be the generalized
    sequences defined by relations $(\ref{12})$-$(\ref{13})$ and $(\ref{19})$-$(\ref{20})$. Thus, the following estimates are satisfied
    \begin{equation}\label{50}
     \underset{0\leq n\leq N-1}{\max}\|U^{n+\frac{1}{2}}\|_{2},\text{\,\,\,} \underset{0\leq n\leq N}{\max}\|U^{n}\|_{2}\leq
      \||u|\|_{\infty,2}+\sqrt{2\widehat{C}T}(k+k^{2-\overline{\beta}}+k^{2}+h^{4}).
   \end{equation}
    Furthermore, denote $e=u-U$ be the error term, it holds
    \begin{equation}\label{51}
     \underset{0\leq n\leq N-1}{\max}\|e^{n+\frac{1}{2}}\|_{2},\text{\,\,\,} \underset{0\leq n\leq N}{\max}\|e^{n}\|_{2}\leq
      3\sqrt{2\widehat{C}T}(k+h^{4}),
   \end{equation}
    where $\widehat{C}$ is a positive constants independent on the space size $h$ and time step $k$.
   \end{theorem}
    We recall that estimate $(\ref{50})$ suggests that the proposed technique $(\ref{s1})$-$(\ref{s5})$ is unconditionally stable whereas
    inequality $(\ref{51})$ shows that the developed numerical scheme is fourth-order spatial convergent and temporal accurate of order $O(k)$.

    \begin{proof}
    Let $e^{n+\frac{1}{2}}=u^{n+\frac{1}{2}}-U^{n+\frac{1}{2}}$ be the temporary error term and $e^{n+1}=u^{n+1}-U^{n+1}$ be the exact one at
    time level $n+1$. Subtracting equation $(\ref{30})$ from $(\ref{29})$ and utilizing $(\ref{35a})$ and $(\ref{44})$ provides
    \begin{equation*}
       e^{n+\frac{1}{2}}_{j}-\alpha kL_{h}e^{n+\frac{1}{2}}_{j}=e^{n}-k\underset{l=0}{\overset{n-\frac{1}{2}}\sum}
       \theta_{l+\frac{1}{2}}^{n+\alpha} \delta_{t}e_{j}^{l}+(\frac{1}{2}-\alpha)kL_{h}e^{n}_{j}-\frac{k}{2}J_{j}^{\overline{\beta}}+O(k^{2}+k^{3}+kh^{4}),
      \end{equation*}
      which is equivalent to
      \begin{equation*}
       e^{n+\frac{1}{2}}_{j}-e^{n}-\alpha kL_{h}e^{n+\frac{1}{2}}_{j}+k\underset{l=0}{\overset{n-\frac{1}{2}}\sum}\theta_{l+\frac{1}{2}}^{n+\alpha}
       \delta_{t}e_{j}^{l}=(\frac{1}{2}-\alpha)kL_{h}e^{n}_{j}-\frac{k}{2}J_{j}^{\overline{\beta}}+O(k^{2}+k^{3}+kh^{4}).
      \end{equation*}
      Multiplying both sides of this equation by $e^{n+\frac{1}{2}}_{j}$ yields
      \begin{equation*}
       (e^{n+\frac{1}{2}}_{j}-e^{n})e^{n+\frac{1}{2}}_{j}-\alpha ke^{n+\frac{1}{2}}_{j}L_{h}e^{n+\frac{1}{2}}_{j}+
       ke^{n+\frac{1}{2}}_{j}\underset{l=0}{\overset{n-\frac{1}{2}}\sum}\theta_{l+\frac{1}{2}}^{n+\alpha}\delta_{t}e_{j}^{l}=
       (\frac{1}{2}-\alpha)ke^{n+\frac{1}{2}}_{j}L_{h}e^{n}_{j}-\frac{k}{2}J_{j}^{\overline{\beta}}e^{n+\frac{1}{2}}_{j}+
       \end{equation*}
      \begin{equation}\label{52a}
      O(k^{2}+k^{3}+kh^{4})e^{n+\frac{1}{2}}_{j}.
      \end{equation}
     We introduce the generalized sequences $(\widehat{\theta}_{l}^{n+\alpha})_{l\leq n+\frac{1}{2}}$ and
     $(\widehat{\theta}_{l}^{n+\frac{1}{2}+\alpha})_{l\leq n+1}$, defined as
     \begin{equation*}
     \widehat{\theta}_{n+\frac{1}{2}}^{n+\alpha}=\theta_{n}^{n+\alpha}\text{\,\,\,and\,\,\,}\widehat{\theta}_{l}^{n+\alpha}=
     \theta_{l}^{n+\alpha},\text{\,\,\,for\,\,\,}l=\frac{1}{2},1,...,n,
      \end{equation*}
     \begin{equation}\label{64}
     \widehat{\theta}_{n+1}^{n+\frac{1}{2}+\alpha}=\theta_{n+\frac{1}{2}}^{n+\frac{1}{2}+\alpha}\text{\,\,\,and\,\,\,}\widehat{\theta}_{l}^{n+\frac{1}{2}
     +\alpha}=\theta_{l}^{n+\frac{1}{2}+\alpha},\text{\,\,\,for\,\,\,}l=\frac{1}{2},1,...,n+\frac{1}{2}.
     \end{equation}
     Thus, the sequences $(\widehat{\theta}_{l}^{n+s+\alpha})_{l\leq n+s+\frac{1}{2}}$ satisfy
     $\widehat{\theta}_{l}^{n+s+\alpha}\leq\widehat{\theta}_{l+\frac{1}{2}}^{n+s+\alpha}$, for $s=0,2^{-1}$ and $l=\frac{1}{2},1,...,n+s$.
     Furthermore, similar to the proof of estimate $(\ref{39})$ in Lemma $\ref{l3}$, it is easy to show that
     \begin{equation}\label{65}
     e^{n+s+\frac{1}{2}}_{j}\underset{l=0}{\overset{n+s}\sum}\widehat{\theta}_{l+\frac{1}{2}}^{n+s+\alpha}\delta_{t}e^{l}_{j}\geq
     \frac{1}{2}\underset{l=0}{\overset{n+s}\sum}\widehat{\theta}_{l+\frac{1}{2}}^{n+s+\alpha}\delta_{t}(e^{l}_{j})^{2}.
     \end{equation}
      Since $\delta_{t}e_{j}^{n}=\frac{2}{k}(e^{n+\frac{1}{2}}_{j}-e^{n}_{j})$ and $k\underset{l=0}{\overset{n-\frac{1}{2}}\sum}
      \widehat{\theta}_{l+\frac{1}{2}}^{n+\alpha}\delta_{t}e_{j}^{l}=k\underset{l=0}{\overset{n}\sum}
      \widehat{\theta}_{l+\frac{1}{2}}^{n+\alpha}\delta_{t}e_{j}^{l}-2\widehat{\theta}_{n+\frac{1}{2}}^{n+\alpha}(e^{n+\frac{1}{2}}_{j}-e^{n}_{j})$.
      Using this, equation $(\ref{52a})$ becomes
      \begin{equation*}
       (e^{n+\frac{1}{2}}_{j}-e^{n})e^{n+\frac{1}{2}}_{j}-\alpha ke^{n+\frac{1}{2}}_{j}L_{h}e^{n+\frac{1}{2}}_{j}+
       ke^{n+\frac{1}{2}}_{j}\underset{l=0}{\overset{n}\sum}\widehat{\theta}_{l+\frac{1}{2}}^{n+\alpha}\delta_{t}e_{j}^{l}-
       2\widehat{\theta}_{n+\frac{1}{2}}^{n+\alpha}(e^{n+\frac{1}{2}}_{j}-e^{n}_{j})e^{n+\frac{1}{2}}_{j}=
       \end{equation*}
      \begin{equation}\label{52}
       (\frac{1}{2}-\alpha)ke^{n+\frac{1}{2}}_{j}L_{h}e^{n}_{j}-\frac{k}{2}J_{j}^{\overline{\beta}}e^{n+\frac{1}{2}}_{j}+
       O(k^{2}+k^{3}+kh^{4})e^{n+\frac{1}{2}}_{j}.
      \end{equation}
      But $(a-b)a=\frac{1}{2}[a^{2}-b^{2}+(a-b)^{2}]$, for every real numbers $a$ and $b$. In addition, for $s=0$, it follows from $(\ref{65})$ that:
      $e^{n+\frac{1}{2}}_{j}\underset{l=0}{\overset{n}\sum}\widehat{\theta}_{l+\frac{1}{2}}^{n+\alpha}\delta_{t}e_{j}^{l}\geq
       \frac{1}{2}\underset{l=0}{\overset{n}\sum}\widehat{\theta}_{l+\frac{1}{2}}^{n+\alpha}\delta_{t}(e_{j}^{l})^{2}.$ These facts, together
       with $(\ref{52})$ result in
       \begin{equation*}
       \frac{1}{2}\left[(e^{n+\frac{1}{2}}_{j})^{2}+(e^{n+\frac{1}{2}}_{j}-e^{n})^{2}-(e^{n}_{j})^{2}\right]-
       \alpha ke^{n+\frac{1}{2}}_{j}L_{h}e^{n+\frac{1}{2}}_{j}+
       \frac{k}{2}\underset{l=0}{\overset{n}\sum}\widehat{\theta}_{l+\frac{1}{2}}^{n+\alpha}\delta_{t}(e_{j}^{l})^{2}-
       \widehat{\theta}_{n+\frac{1}{2}}^{n+\alpha}\left[(e^{n+\frac{1}{2}}_{j})^{2}+\right.
       \end{equation*}
      \begin{equation*}
       \left.(e^{n+\frac{1}{2}}_{j}-e^{n})^{2}-(e^{n}_{j})^{2}\right]\leq(\frac{1}{2}-\alpha)ke^{n+\frac{1}{2}}_{j}L_{h}e^{n}_{j}
       -\frac{k}{2}J_{j}^{\overline{\beta}}e^{n+\frac{1}{2}}_{j}+ O(k^{2}+k^{3}+kh^{4})e^{n+\frac{1}{2}}_{j}.
      \end{equation*}
      Setting $J^{\overline{\beta}}=(J^{\overline{\beta}}_{2},J^{\overline{\beta}}_{3},...,J^{\overline{\beta}}_{M-2})$ and
      $\overline{O}(k^{2}+k^{3}+kh^{4})=(O(k^{2}+k^{3}+kh^{4}),O(k^{2}+k^{3}+kh^{4}),...,O(k^{2}+k^{3}+kh^{4}))$, multiplying both sides of this
      estimate by $2h$, summing the obtained estimate up from $j=2,3,...,M-2$, and using the definition of the $L^{2}$-norm and scalar product
      given by relations $(\ref{dn})$ and $(\ref{sp})$, respectively, this provides
      \begin{equation*}
      \|e^{n+\frac{1}{2}}\|_{2}^{2}+\|e^{n+\frac{1}{2}}-e^{n}\|_{2}^{2}-\|e^{n}\|_{2}^{2}+2\alpha k\left(-L_{h}e^{n+\frac{1}{2}},e^{n+\frac{1}{2}}\right)
      +2\underset{l=0}{\overset{n}\sum}\widehat{\theta}_{l+\frac{1}{2}}^{n+\alpha}\left(\|e^{l+\frac{1}{2}}\|_{2}^{2}-\|e^{l}\|_{2}^{2}\right)-
      2\widehat{\theta}_{n+\frac{1}{2}}^{n+\alpha}\left[\|e^{n+\frac{1}{2}}\|_{2}^{2}+\right.
      \end{equation*}
     \begin{equation}\label{53}
      \left.\|e^{n+\frac{1}{2}}-e^{n}\|_{2}^{2}-\|e^{n}\|_{2}^{2}\right]\leq
      (1-2\alpha)k\left(L_{h}e^{n},e^{n+\frac{1}{2}}\right)-k\left(J^{\overline{\beta}},e^{n+\frac{1}{2}}\right)+
      \left(\overline{O}(k^{2}+k^{3}+kh^{4}),e^{n+\frac{1}{2}}\right).
     \end{equation}
     Now, utilizing the second estimate of $(\ref{46})$ together with the summation by parts $(\ref{42})$, it is not hard to observe that $(\ref{53})$
     implies
     \begin{equation*}
      \|e^{n+\frac{1}{2}}\|_{2}^{2}+\left(1-2\widehat{\theta}_{n+\frac{1}{2}}^{n+\alpha}\right)\|e^{n+\frac{1}{2}}-e^{n}\|_{2}^{2}+
      \alpha k\|\delta_{x}e^{n+\frac{1}{2}}\|_{2}^{2}\leq 2\widehat{\theta}_{\frac{1}{2}}^{n+\alpha}\|e^{0}\|_{2}^{2}+
      \left(1-2\widehat{\theta}_{n+\frac{1}{2}}^{n+\alpha}\right)\|e^{n}\|_{2}^{2}
      \end{equation*}
     \begin{equation}\label{54}
     +2\underset{l=0}{\overset{n-\frac{1}{2}}\sum}\left(\widehat{\theta}_{l+1}^{n+\alpha}-\widehat{\theta}_{l+\frac{1}{2}}^{n+\alpha}\right)
     \|e^{l}\|_{2}^{2}+(1-2\alpha)k\left(L_{h}e^{n},e^{n+\frac{1}{2}}\right)-k\left(J^{\overline{\beta}},e^{n+\frac{1}{2}}\right)+
      \left(\overline{O}(k^{2}+k^{3}+kh^{4}),e^{n+\frac{1}{2}}\right).
     \end{equation}
     It follows from the Poincar\'{e}-Friedrich inequality that $\|e^{n+\frac{1}{2}}\|_{2}^{2}\leq C_{p}\|\delta_{x}e^{n+\frac{1}{2}}\|_{2}^{2}$,
     where $C_{p}$ denotes a positive constant independent of $k$ and $h$. Using this, estimate $(\ref{45})$ and the H\"{o}lder inequality,
     straightforward calculations give
     \begin{equation*}
      (1-2\alpha)k\left(L_{h}e^{n},e^{n+\frac{1}{2}}\right)\leq C_{p}(1-2\alpha)k\|\delta_{x}e^{n}\|_{2}\|\delta_{x}e^{n+\frac{1}{2}}\|_{2}=
     2\left(\frac{1-2\alpha}{2}C_{p}\sqrt{\frac{6k}{\alpha}}\|\delta_{x}e^{n}\|_{2}\right)
     \left(\sqrt{\frac{\alpha k}{6}}\|\delta_{x}e^{n+\frac{1}{2}}\|_{2}\right)
      \end{equation*}
     \begin{equation}\label{55}
     \leq\frac{3(1-2\alpha)^{2}C_{p}^{2}}{2\alpha}k\|\delta_{x}e^{n}\|_{2}^{2}+\frac{\alpha k}{6}\|\delta_{x}e^{n+\frac{1}{2}}\|_{2}^{2}.
     \end{equation}
    \begin{equation}\label{56}
     -k\left(J^{\overline{\beta}},e^{n+\frac{1}{2}}\right)\leq\frac{3C_{p}^{2}}{2\alpha}k\|J^{\overline{\beta}}\|_{2}^{2}+
     \frac{\alpha k}{6}\|\delta_{x}e^{n+\frac{1}{2}}\|_{2}^{2},\text{\,}\left(\overline{O}(k^{2}+k^{3}+kh^{4}),e^{n+\frac{1}{2}}\right)\leq
     \frac{\alpha k}{6}\|\delta_{x}e^{n+\frac{1}{2}}\|_{2}^{2}+\widetilde{C}_{1}k(k+k^{2}+h^{4})^{2},
     \end{equation}
     where $\widetilde{C}_{1}>0$ is a constant which does not depend on the time step $k$ and the space step $h$. Setting
     $\widehat{C}_{\alpha}=\max\{\widetilde{C}_{1},\frac{3}{2}C_{p}^{2}\alpha^{-1}\}$, substituting estimates $(\ref{55})$ and $(\ref{56})$
     into relation $(\ref{54})$ and rearranging terms yields
     \begin{equation*}
      \|e^{n+\frac{1}{2}}\|_{2}^{2}+\left(1-2\widehat{\theta}_{n+\frac{1}{2}}^{n+\alpha}\right)\|e^{n+\frac{1}{2}}-e^{n}\|_{2}^{2}+\frac{\alpha k}{2}
      \|\delta_{x}e^{n+\frac{1}{2}}\|_{2}^{2}\leq 2\widehat{\theta}_{\frac{1}{2}}^{n+\alpha}\|e^{0}\|_{2}^{2}+
      \left(1-2\widehat{\theta}_{n+\frac{1}{2}}^{n+\alpha}\right)\|e^{n}\|_{2}^{2}
      \end{equation*}
     \begin{equation}\label{58}
     +2\underset{l=0}{\overset{n-\frac{1}{2}}\sum}\left(\widehat{\theta}_{l+1}^{n+\alpha}-\widehat{\theta}_{l+\frac{1}{2}}^{n+\alpha}\right)\|e^{l}\|_{2}^{2}+
      \frac{3(1-2\alpha)^{2}C_{p}^{2}}{2}\alpha^{-1}k\|\delta_{x}e^{n}\|_{2}^{2}+\widehat{C}_{\alpha}k[\|J^{\overline{\beta}}\|_{2}^{2}+(k+k^{2}+h^{4})^{2}].
     \end{equation}
     Estimate $(\ref{58})$ is satisfied for any $0<\alpha<2^{-1}$. For $n\geq1$, it comes from $(\ref{13})$-$(\ref{14})$ and $(\ref{20})$-
     $(\ref{21})$, that
     \begin{equation*}
     \alpha^{1-\overline{\beta}}=\widetilde{f}_{n,n}^{\alpha,\overline{\beta}}<a_{n,n}^{\alpha,\overline{\beta}}=
     \widetilde{f}_{n,n-1}^{\alpha,\overline{\beta}}+\widetilde{f}_{n,n}^{\alpha,\overline{\beta}}=\widetilde{f}_{n+\frac{1}{2},n-1}^{\alpha,\overline{\beta}}
     +\widetilde{f}_{n+\frac{1}{2},n}^{\alpha,\overline{\beta}}=a_{n+\frac{1}{2},n+\frac{1}{2}}^{\alpha,\overline{\beta}}=
      \end{equation*}
     \begin{equation}\label{59}
     \widetilde{f}_{n+1,n}^{\alpha,\overline{\beta}}+\widetilde{f}_{n+1,n+1}^{\alpha,\overline{\beta}}=a_{n+1,n+1}^{\alpha,\overline{\beta}}
     <\alpha^{1-\overline{\beta}}+\frac{2}{2-\overline{\beta}}\left[(1+\alpha)^{2-\overline{\beta}}-\alpha^{2-\overline{\beta}}\right].
     \end{equation}
    Since $\widehat{\theta}_{n+\frac{1}{2}}^{n+\alpha}=k^{1-\overline{\beta}}\Gamma(2-\overline{\beta})^{-1}a_{n,n}^{\alpha,\overline{\beta}}$,
    for small values of $k$, it follows from $(\ref{59})$ that: $1-2\widehat{\theta}_{n+\frac{1}{2}}^{n+\alpha}>0$. This fact and estimate
    $(\ref{58})$ imply
    \begin{equation*}
      \|e^{n+\frac{1}{2}}\|_{2}^{2}\leq 2\widehat{\theta}_{\frac{1}{2}}^{n+\alpha}\|e^{0}\|_{2}^{2}+
      \left(1-2\widehat{\theta}_{n+\frac{1}{2}}^{n+\alpha}\right)\|e^{n}\|_{2}^{2}+2\underset{l=0}{\overset{n-\frac{1}{2}}\sum}
      \left(\widehat{\theta}_{l+1}^{n+\alpha}-\widehat{\theta}_{l+\frac{1}{2}}^{n+\alpha}\right)
      \|e^{l}\|_{2}^{2}+\frac{3(1-2\alpha)^{2}C_{p}^{2}}{2}\alpha^{-1}k\|\delta_{x}e^{n}\|_{2}^{2}
      \end{equation*}
     \begin{equation*}
     +\widehat{C}_{\alpha}k[\|J^{\overline{\beta}}\|_{2}^{2}+(k+k^{2}+h^{4})^{2}]\leq \left[1+2\widehat{\theta}_{\frac{1}{2}}^{n+\alpha}
     -2\widehat{\theta}_{n+\frac{1}{2}}^{n+\alpha}+2\underset{l=0}{\overset{n-\frac{1}{2}}\sum}
     \left(\widehat{\theta}_{l+1}^{n+\alpha}-\widehat{\theta}_{l+\frac{1}{2}}^{n+\alpha}\right)\right]\underset{0\leq l\leq n}{\max}\|e^{l}\|_{2}^{2}+
     \end{equation*}
     \begin{equation}\label{60}
     \frac{3(1-2\alpha)^{2}C_{p}^{2}}{2}\alpha^{-1}k\|\delta_{x}e^{n}\|_{2}^{2}+\widehat{C}_{\alpha}k[\|J^{\overline{\beta}}\|_{2}^{2}+(k+k^{2}+h^{4})^{2}],
     \end{equation}
     where the summation index "l" varies in the range: $\frac{1}{2},1,\frac{3}{2},2,...,n-\frac{1}{2},n$. $(\ref{60})$ is equivalent to
     \begin{equation}\label{60a}
      \|e^{n+\frac{1}{2}}\|_{2}^{2}\leq \underset{0\leq l\leq n}{\max}\|e^{l}\|_{2}^{2}+\frac{3(1-2\alpha)^{2}C_{p}^{2}}{2}\alpha^{-1}k
      \|\delta_{x}e^{n}\|_{2}^{2}+\widehat{C}_{\alpha}k[\|J^{\overline{\beta}}\|_{2}^{2}+(k+k^{2}+h^{4})^{2}],
     \end{equation}
      for every $\alpha\in(0,2^{-1})$. Setting $\widehat{C}_{1}=\underset{\alpha\rightarrow\frac{1}{2}}{\lim}\widehat{C}_{\alpha}=
      \underset{\alpha\rightarrow\frac{1}{2}}{\lim}\max\{\widetilde{C}_{1},\frac{3}{2}C_{p}^{2}\alpha^{-1}\}=\max\{\widetilde{C}_{1},3C_{p}^{2}\}$.
      Taking the limit in estimate $(\ref{60a})$ when $\alpha$ approaches $\frac{1}{2}$ and combining $(\ref{17})$ and $(\ref{36})$, to obtain
      \begin{equation*}
      \|e^{n+\frac{1}{2}}\|_{2}^{2}\leq \underset{0\leq l\leq n}{\max}\|e^{l}\|_{2}^{2}+\widehat{C}_{1}k[C_{\frac{1}{2}}^{2}k^{4-2\overline{\beta}}+
      (k+k^{2}+h^{4})^{2}].
     \end{equation*}
     This implies
     \begin{equation}\label{61}
      \underset{0\leq l\leq n+\frac{1}{2}}{\max}\|e^{l}\|_{2}^{2}\leq \underset{0\leq l\leq n}{\max}\|e^{l}\|_{2}^{2}+
      \widehat{C}_{1}k[C_{\frac{1}{2}}^{2}k^{4-2\overline{\beta}}+(k+k^{2}+h^{4})^{2}]\leq\underset{0\leq l\leq n}{\max}\|e^{l}\|_{2}^{2}+
      \widehat{C}_{2}k(k+k^{2-\overline{\beta}}+k^{2}+h^{4})^{2},
     \end{equation}
     where $\widehat{C}_{2}=\widehat{C}_{1}(1+C_{\frac{1}{2}}^{2})$.\\

      In a similar way, a combination of equations $(\ref{32})$, $(\ref{34})$ and $(\ref{44})$ gives
      \begin{equation*}
       e^{n+1}_{j}-e^{n+\frac{1}{2}}-\frac{1+4\alpha}{4}kL_{h}e^{n+1}_{j}+\frac{k}{4}\left[\underset{l=0}{\overset{n+\frac{1}{2}}\sum}
       \theta_{l+\frac{1}{2}}^{n+1+\alpha}
       \delta_{t}e_{j}^{l}+\underset{l=0}{\overset{n}\sum}\theta_{l+\frac{1}{2}}^{n+\frac{1}{2}+\alpha}\delta_{t}e_{j}^{l}\right]
       =\frac{1-4\alpha}{4}kL_{h}e^{n+\frac{1}{2}}-
      \end{equation*}
      \begin{equation}\label{62}
      \frac{k}{4}(J_{j}^{\overline{\beta}}+I_{j}^{\overline{\beta}})+O(k^{2}+k^{3}+kh^{4}).
      \end{equation}
      Multiplying both sides of this equation by $2he^{n+1}_{j}$, summing this up from $j=2,3,...,M-2$, and using the definition of the $L^{2}$-norm
      and scalar product, it is not hard to observe that $(\ref{62})$ implies
       \begin{equation*}
      \|e^{n+1}\|_{2}^{2}+\|e^{n+1}-e^{n+\frac{1}{2}}\|_{2}^{2}-\|e^{n+\frac{1}{2}}\|_{2}^{2}+\frac{1+4\alpha}{2} k\left(-L_{h}e^{n+1},e^{n+1}\right)
      +\frac{k}{2}h\underset{j=2}{\overset{M-2}\sum}\left[e^{n+1}\underset{l=0}{\overset{n+\frac{1}{2}}\sum}\theta_{l+\frac{1}{2}}^{n+1+\alpha}
      \delta_{t}e_{j}^{l}+\right.
      \end{equation*}
      \begin{equation}\label{63}
      \left.e^{n+1}\underset{l=0}{\overset{n}\sum}\theta_{l+\frac{1}{2}}^{n+\frac{1}{2}+\alpha}\delta_{t}e_{j}^{l}\right]
       =\frac{1-4\alpha}{2}k\left(L_{h}e^{n+\frac{1}{2}},e^{n+1}\right)-\frac{k}{2}\left(J^{\overline{\beta}}+
      I^{\overline{\beta}},e^{n+1}\right)+\left(\overline{O}(k^{2}+k^{3}+kh^{4}),e^{n+1}\right),
      \end{equation}
     where $\overline{O}(k^{2}+k^{3}+kh^{4})=(O(k^{2}+k^{3}+kh^{4}),...,O(k^{2}+k^{3}+kh^{4}))$, $J^{\overline{\beta}}=
     (J^{\overline{\beta}}_{2},...,J^{\overline{\beta}}_{M-2})$ and $I^{\overline{\beta}}=(I^{\overline{\beta}}_{2},...,I^{\overline{\beta}}_{M-2})$.
     Utilizing the generalized sequence defined by equation $(\ref{64})$, relation $(\ref{63})$ becomes
     \begin{equation*}
      \|e^{n+1}\|_{2}^{2}+\|e^{n+1}-e^{n+\frac{1}{2}}\|_{2}^{2}-\|e^{n+\frac{1}{2}}\|_{2}^{2}+\frac{1+4\alpha}{2} k\left(-L_{h}e^{n+1},e^{n+1}\right)
      +\frac{k}{2}h\underset{j=2}{\overset{M-2}\sum}\left[e^{n+1}\underset{l=0}{\overset{n+\frac{1}{2}}\sum}\theta_{l+\frac{1}{2}}^{n+1+\alpha}
      \delta_{t}e_{j}^{l}+\right.
      \end{equation*}
      \begin{equation}\label{65a}
      \left.e^{n+1}\underset{l=0}{\overset{n}\sum}\widehat{\theta}_{l+\frac{1}{2}}^{n+\frac{1}{2}+\alpha}\delta_{t}e_{j}^{l}\right]
       =\frac{1-4\alpha}{2}k\left(L_{h}e^{n+\frac{1}{2}},e^{n+1}\right)-\frac{k}{2}\left(J^{\overline{\beta}}+
      I^{\overline{\beta}},e^{n+1}\right)+\left(\overline{O}(k^{2}+k^{3}+kh^{4}),e^{n+1}\right).
      \end{equation}
      For $s=2^{-1}$, using estimate $(\ref{65})$ and performing direct calculations, it holds
       \begin{equation*}
      e^{n+1}\underset{l=0}{\overset{n}\sum}\widehat{\theta}_{l+\frac{1}{2}}^{n+\frac{1}{2}+\alpha}\delta_{t}e_{j}^{l}=e^{n+1}
      \underset{l=0}{\overset{n+\frac{1}{2}}\sum}\widehat{\theta}_{l+\frac{1}{2}}^{n+\frac{1}{2}+\alpha}\delta_{t}e_{j}^{l}-
      e_{j}^{n+1}\widehat{\theta}_{n+1}^{n+\frac{1}{2}+\alpha}\delta_{t}e_{j}^{n+\frac{1}{2}}\geq \frac{1}{2}
      \underset{l=0}{\overset{n+\frac{1}{2}}\sum}\widehat{\theta}_{l+\frac{1}{2}}^{n+\frac{1}{2}+\alpha}\delta_{t}(e_{j}^{l})^{2}-
      \end{equation*}
      \begin{equation}\label{66}
       \frac{1}{k}\widehat{\theta}_{n+1}^{n+\frac{1}{2}+\alpha}\left[(e_{j}^{n+1})^{2}+(e_{j}^{n+1}-e_{j}^{n+\frac{1}{2}})^{2}-
       (e_{j}^{n+\frac{1}{2}})^{2}\right].
      \end{equation}
      Estimate $(\ref{66})$ combined with $(\ref{65a})$ and Lemmas $\ref{l3}$ and $\ref{l6}$ yield
      \begin{equation*}
      \left(1-\frac{1}{2}\widehat{\theta}_{n+1}^{n+\frac{1}{2}+\alpha}\right)\|e^{n+1}\|_{2}^{2}+
      \left(1-\frac{1}{2}\widehat{\theta}_{n+1}^{n+\frac{1}{2}+\alpha}\right)\|e^{n+1}-e^{n+\frac{1}{2}}\|_{2}^{2}
      +\frac{1+4\alpha}{4} k\|\delta_{x}e^{n+1}\|_{2}^{2}+
      \end{equation*}
       \begin{equation*}
      \frac{1}{2}\underset{l=0}{\overset{n+\frac{1}{2}}\sum}\left(\theta_{l+\frac{1}{2}}^{n+1+\alpha}+
      \widehat{\theta}_{l+\frac{1}{2}}^{n+\frac{1}{2}+\alpha}\right)\left(\|e^{l+\frac{1}{2}}\|_{2}^{2}-\|e^{l}\|_{2}^{2}\right)
       \leq\left(1-\frac{1}{2}\widehat{\theta}_{n+1}^{n+\frac{1}{2}+\alpha}\right)\|e^{n+\frac{1}{2}}\|_{2}^{2}+
      \end{equation*}
      \begin{equation*}
      \frac{1-4\alpha}{2}k\left(L_{h}e^{n+\frac{1}{2}},e^{n+1}\right)-\frac{k}{2}\left(J^{\overline{\beta}}+ I^{\overline{\beta}},e^{n+1}\right)
      +\left(\overline{O}(k^{2}+k^{3}+kh^{4}),e^{n+1}\right).
      \end{equation*}
      Applying the summation by parts and rearranging terms, this becomes
       \begin{equation*}
      \left(1+\frac{1}{2}\theta_{n+1}^{n+1+\alpha}\right)\|e^{n+1}\|_{2}^{2}+
      \left(1-\frac{1}{2}\widehat{\theta}_{n+1}^{n+\frac{1}{2}+\alpha}\right)\|e^{n+1}-e^{n+\frac{1}{2}}\|_{2}^{2}
      +\frac{1+4\alpha}{4} k\|\delta_{x}e^{n+1}\|_{2}^{2}+
      \end{equation*}
       \begin{equation*}
      \frac{1}{2}\underset{l=0}{\overset{n}\sum}\left[\left(\theta_{l+\frac{1}{2}}^{n+1+\alpha}+\widehat{\theta}_{l+\frac{1}{2}}^{n+\frac{1}{2}
       +\alpha}\right)-\left(\theta_{l+1}^{n+1+\alpha}+\widehat{\theta}_{l+1}^{n+\frac{1}{2}+\alpha}\right)\right]\|e^{l+\frac{1}{2}}\|_{2}^{2}
       \leq \frac{1}{2}\left(\theta_{\frac{1}{2}}^{n+1+\alpha}+\widehat{\theta}_{\frac{1}{2}}^{n+\frac{1}{2}+\alpha}\right)\|e^{0}\|_{2}^{2}+
      \end{equation*}
      \begin{equation}\label{67}
      \left(1-\frac{1}{2}\widehat{\theta}_{n+1}^{n+\frac{1}{2}+\alpha}\right)\|e^{n+\frac{1}{2}}\|_{2}^{2}+
      \frac{1-4\alpha}{2}k\left(L_{h}e^{n+\frac{1}{2}},e^{n+1}\right)-\frac{k}{2}\left(J^{\overline{\beta}}+ I^{\overline{\beta}},e^{n+1}\right)
      +\left(\overline{O}(k^{2}+k^{3}+kh^{4}),e^{n+1}\right).
      \end{equation}
      Performing direct computations, using the H\"{o}lder and Poincar\'{e}-Friedrich inequalities, equations $(\ref{9})$, $(\ref{17})$ and the
      second estimate in $(\ref{45})$, it is not difficult to show that
      \begin{equation*}
      \frac{1-4\alpha}{2}k\left(L_{h}e^{n+\frac{1}{2}},e^{n+1}\right)\leq \frac{|1-4\alpha|}{2}C_{p}k
      \|\delta_{x}e^{n+\frac{1}{2}}\|_{2}\|\delta_{x}e^{n+1}\|_{2}=
     2\left(\frac{|1-4\alpha|}{2}C_{p}\sqrt{\frac{3k}{1+4\alpha}}\|\delta_{x}e^{n+\frac{1}{2}}\|_{2}\right)
      \end{equation*}
     \begin{equation}\label{68}
     \left(\sqrt{\frac{k(1+4\alpha)}{12}}\|\delta_{x}e^{n+1}\|_{2}\right)\leq\frac{3(1-4\alpha)^{2}}{4(1+4\alpha)}C_{p}^{2}
     k\|\delta_{x}e^{n+\frac{1}{2}}\|_{2}^{2}+\frac{(1+4\alpha)k}{12}\|\delta_{x}e^{n+1}\|_{2}^{2},
     \end{equation}
    \begin{equation}\label{69}
     -k\left(J^{\overline{\beta}}+I^{\overline{\beta}},e^{n+1}\right)\leq\frac{C_{p}}{2}k\|J^{\overline{\beta}}+I^{\overline{\beta}}\|_{2}^{2}
     \|\delta_{x}e^{n+1}\|_{2}^{2}\leq \frac{3C_{p}^{2}}{2(1+4\alpha)}k
     (\|J^{\overline{\beta}}\|_{2}^{2}+\|I^{\overline{\beta}}\|_{2}^{2})+\frac{(1+4\alpha)k}{12}\|\delta_{x}e^{n+1}\|_{2}^{2},
     \end{equation}
     \begin{equation}\label{70}
     \left(\overline{O}(k^{2}+k^{3}+kh^{4}),e^{n+1}\right)=k\left(\overline{O}(k+k^{2}+h^{4}),e^{n+1}\right)\leq \frac{(1+4\alpha) k}{12}
     \|\delta_{x}e^{n+1}\|_{2}^{2}+C_{2}k(k+k^{2}+h^{4})^{2},
     \end{equation}
     where $C_{2}>0$, is a constant independent of the time step $k$ and mesh size $h$. A combination of estimates $(\ref{67})$-$(\ref{70})$ results in
     \begin{equation*}
      \left(1+\frac{1}{2}\theta_{n+1}^{n+1+\alpha}\right)\|e^{n+1}\|_{2}^{2}+
      \left(1-\frac{1}{2}\widehat{\theta}_{n+1}^{n+\frac{1}{2}+\alpha}\right)\|e^{n+1}-e^{n+\frac{1}{2}}\|_{2}^{2}\leq
      \frac{1}{2}\underset{l=0}{\overset{n}\sum}\left[\left(\theta_{l+1}^{n+1+\alpha}+\widehat{\theta}_{l+1}^{n+\frac{1}{2}+\alpha}\right)-\right.
      \end{equation*}
       \begin{equation*}
      \left.\left(\theta_{l+\frac{1}{2}}^{n+1+\alpha}+\widehat{\theta}_{l+\frac{1}{2}}^{n+\frac{1}{2}+\alpha}\right)\right]\|e^{l+\frac{1}{2}}\|_{2}^{2}
      +\frac{1}{2}\left(\theta_{\frac{1}{2}}^{n+1+\alpha}+\widehat{\theta}_{\frac{1}{2}}^{n+\frac{1}{2}+\alpha}\right)\|e^{0}\|_{2}^{2}+
      \left(1-\frac{1}{2}\widehat{\theta}_{n+1}^{n+\frac{1}{2}+\alpha}\right)\|e^{n+\frac{1}{2}}\|_{2}^{2}+
      \end{equation*}
      \begin{equation}\label{71}
      \frac{3(1-4\alpha)^{2}}{4(1+4\alpha)}C_{p}^{2}k\|\delta_{x}e^{n+\frac{1}{2}}\|_{2}^{2}+\frac{3C_{p}^{2}}{2(1+4\alpha)}k(\|J^{\overline{\beta}}
      \|_{2}^{2}+\|I^{\overline{\beta}}\|_{2}^{2})+C_{2}k(k+k^{2}+h^{4})^{2}.
      \end{equation}
      For small values of $k$, $1-\frac{1}{2}\widehat{\theta}_{n+1}^{n+\frac{1}{2}+\alpha}>0$. Utilizing this and $(\ref{36})$, relation
      $(\ref{71})$ becomes
      \begin{equation*}
      \left(1+\frac{1}{2}\theta_{n+1}^{n+1+\alpha}\right)\|e^{n+1}\|_{2}^{2}\leq \frac{1}{2}\left\{
      \left(\theta_{\frac{1}{2}}^{n+1+\alpha}+\widehat{\theta}_{\frac{1}{2}}^{n+\frac{1}{2}+\alpha}\right)+
      \underset{l=0}{\overset{n}\sum}\left[\left(\theta_{l+1}^{n+1+\alpha}+\widehat{\theta}_{l+1}^{n+\frac{1}{2}+\alpha}\right)-\right.\right.
      \end{equation*}
      \begin{equation*}
       \left.\left(\theta_{l+\frac{1}{2}}^{n+1+\alpha}+\widehat{\theta}_{l+\frac{1}{2}}^{n+\frac{1}{2}+\alpha}\right)\right]+
       \left.2-\widehat{\theta}_{n+1}^{n+\frac{1}{2}+\alpha}\right\}\underset{0\leq l\leq n+\frac{1}{2}}{\max}\|e^{l}\|_{2}^{2}+
       \frac{3(1-4\alpha)^{2}}{4(1+4\alpha)}C_{p}^{2}k\|\delta_{x}e^{n+\frac{1}{2}}\|_{2}^{2}+
      \end{equation*}
      \begin{equation*}
       \frac{3C_{p}^{2}}{1+4\alpha}(C_{\frac{1}{2}}^{2}+C_{1}^{2})k^{5-2\overline{\beta}}+C_{2}k(k+k^{2}+h^{4})^{2}.
      \end{equation*}
      Letting $\widehat{C}_{3,\alpha}=\max\left\{\frac{3C_{p}^{2}}{1+4\alpha}(C_{\frac{1}{2}}^{2}+C_{1}^{2}),C_{2}\right\}$, this estimate implies
      \begin{equation*}
      \left(1+\frac{1}{2}\theta_{n+1}^{n+1+\alpha}\right)\|e^{n+1}\|_{2}^{2}\leq
       \left(1+\frac{1}{2}\widehat{\theta}_{n+1}^{n+\frac{1}{2}+\alpha}\right)\underset{0\leq l\leq n+\frac{1}{2}}{\max}\|e^{l}\|_{2}^{2}+
       \frac{3(1-4\alpha)^{2}}{4(1+4\alpha)}C_{p}^{2}k\|\delta_{x}e^{n+\frac{1}{2}}\|_{2}^{2}+
      \end{equation*}
      \begin{equation}\label{72}
       \widehat{C}_{3,\alpha}k(k+k^{2-\overline{\beta}}+k^{2}+h^{4})^{2}.
      \end{equation}
       It follows from $(\ref{59})$ that $a_{n+\frac{1}{2},n+\frac{1}{2}}^{\alpha,\overline{\beta}}=a_{n+1,n+1}^{\alpha,\overline{\beta}}$.
      This fact, together with $(\ref{35a})$ and $(\ref{64})$ give $\theta_{n+1}^{n+1+\alpha}=\widehat{\theta}_{n+1}^{n+\frac{1}{2}+\alpha}$.
      Since, $\left(1+\frac{1}{2}\theta_{n+1}^{n+1+\alpha}\right)^{-1}<1$, multiplying both sides of $(\ref{72})$ by
      $\left(1+\frac{1}{2}\theta_{n+1}^{n+1+\alpha}\right)^{-1}$ to get
      \begin{equation}\label{73}
      \|e^{n+1}\|_{2}^{2}\leq \underset{0\leq l\leq n+\frac{1}{2}}{\max}\|e^{l}\|_{2}^{2}+\frac{3(1-4\alpha)^{2}}{4(1+4\alpha)}
       C_{p}^{2}k\|\delta_{x}e^{n+\frac{1}{2}}\|_{2}^{2}+\widehat{C}_{3,\alpha}k(k+k^{2-\overline{\beta}}+k^{2}+h^{4})^{2}.
      \end{equation}
      Taking the limit when $\alpha$ approaches $\frac{1}{4}$, $(\ref{73})$ provides
      \begin{equation*}
      \|e^{n+1}\|_{2}^{2}\leq\underset{0\leq l\leq n+\frac{1}{2}}{\max}\|e^{l}\|_{2}^{2}+\widehat{C}_{3}k(k+k^{2-\overline{\beta}}+k^{2}+h^{4})^{2},
      \end{equation*}
      where $\widehat{C}_{3}=\underset{\alpha\rightarrow\frac{1}{4}}{\lim}\widehat{C}_{3,\alpha}=\underset{\alpha\rightarrow\frac{1}{4}}{\lim}
      \max\left\{\frac{3C_{p}^{2}}{1+4\alpha}(C_{\frac{1}{2}}^{2}+C_{1}^{2}),C_{2}\right\}=
      \max\left\{\frac{3C_{p}^{2}}{2}(C_{\frac{1}{2}}^{2}+C_{1}^{2}),C_{2}\right\}$. This estimate implies
       \begin{equation}\label{74}
      \underset{0\leq l\leq n+1}{\max}\|e^{l}\|_{2}^{2}\leq\underset{0\leq l\leq n+\frac{1}{2}}{\max}\|e^{l}\|_{2}^{2}+\widehat{C}_{3}
      k(k+k^{2-\overline{\beta}}+k^{2}+h^{4})^{2}.
      \end{equation}
      It is worth noticing to remind that the summation index $"l"$ varies in the range: $l=0,\frac{1}{2},1,...,n+\frac{1}{2},n+1$. Now, setting
       \begin{equation}\label{75}
      Z^{q}=\underset{0\leq l\leq q}{\max}\|e^{l}\|_{2}^{2}\text{\,\,\,and\,\,\,}\widehat{C}_{4}=\max\{\widehat{C}_{2},\widehat{C}_{3}\},
      \end{equation}
      estimates $(\ref{61})$ and $(\ref{74})$ can be rewritten as
      \begin{equation*}
      Z^{n+\frac{1}{2}}\leq Z^{n}+\widehat{C}_{4}k(k+k^{2-\overline{\beta}}+k^{2}+h^{4})^{2},\text{\,\,\,and\,\,\,}
      Z^{n+1}\leq Z^{n+\frac{1}{2}}+\widehat{C}_{4}k(k+k^{2-\overline{\beta}}+k^{2}+h^{4})^{2}.
      \end{equation*}
      Substituting the first estimate into the second one gives
       \begin{equation*}
       Z^{n+1}\leq Z^{n}+2\widehat{C}_{4}k(k+k^{2-\overline{\beta}}+k^{2}+h^{4})^{2}.
      \end{equation*}
      Summing this up from $n=0,1,2,...,N-1$, to obtain
      \begin{equation}\label{78}
       Z^{N}\leq Z^{0}+2\widehat{C}_{4}Nk(k+k^{2-\overline{\beta}}+k^{2}+h^{4})^{2}.
      \end{equation}
      But, it comes from the initial condition $(\ref{s5})$ that $e_{j}^{0}=u_{j}^{0}-U_{j}^{0}=0$, for $j=0,1,2,...,M$. Furthermore, since
      $k=\frac{T}{N}$ so, $Nk=T$. These facts together with $(\ref{75})$ and $(\ref{78})$ yield
      \begin{equation*}
      \underset{0\leq l\leq N}{\max}\|e^{l}\|_{2}^{2}\leq \|e^{0}\|_{2}^{2}+2\widehat{C}_{4}T(k+k^{2-\overline{\beta}}+k^{2}+h^{4})^{2}.
      \end{equation*}
      This is equivalent to
      \begin{equation*}
      \underset{0\leq n\leq N-1}{\max}\|e^{n+\frac{1}{2}}\|_{2}^{2},\text{\,\,\,\,}\underset{0\leq n\leq N}{\max}\|e^{n}\|_{2}^{2}\leq
       2\widehat{C}_{4}T(k+k^{2-\overline{\beta}}+k^{2}+h^{4})^{2}.
      \end{equation*}
      Taking the square root to obtain
      \begin{equation*}
      \underset{0\leq n\leq N-1}{\max}\|e^{n+\frac{1}{2}}\|_{2},\text{\,\,\,\,}\underset{0\leq n\leq N}{\max}\|e^{n}\|_{2}\leq
      \sqrt{2\widehat{C}_{4}T}(k+k^{2-\overline{\beta}}+k^{2}+h^{4}).
      \end{equation*}
      These estimates imply
      \begin{equation}\label{79}
      \|e^{n+\frac{1}{2}}\|_{2},\text{\,\,\,\,}\|e^{n}\|_{2}\leq \sqrt{2\widehat{C}_{4}T}(k+k^{2-\overline{\beta}}+k^{2}+h^{4}),
      \end{equation}
      for $n=0,1,2,...,N-1$ (resp., $N$). But $|\|u^{q}\|_{2}-\|U^{q}\|_{2}|\leq \|u^{q}-U^{q}\|_{2}=\|e^{q}\|_{2}$, thus
      \begin{equation*}
      \|U^{n+\frac{1}{2}}\|_{2},\text{\,\,\,\,}\|U^{n}\|_{2}\leq \||u|\|_{\infty,2}+\sqrt{2\widehat{C}_{4}T}(k+k^{2-\overline{\beta}}+k^{2}+h^{4}),
      \end{equation*}
      for $n=0,1,2,...,N-1$ (resp., $N$), which imply
      \begin{equation*}
      \underset{0\leq n\leq N-1}{\max}\|U^{n+\frac{1}{2}}\|_{2},\text{\,\,\,\,}\underset{0\leq n\leq N}{\max}\|U^{n}\|_{2}\leq \||u|\|_{\infty,2}
      + \sqrt{2\widehat{C}_{4}T}(k+k^{2-\overline{\beta}}+k^{2}+h^{4}).
      \end{equation*}
      This ends the proof of estimate $(\ref{50})$ in Theorem $\ref{t}$. Now, since $k<1$ and $2-\overline{\beta}>1$, so $k^{2-\overline{\beta}},k^{2}
      \leq k$, thus $k+k^{2-\overline{\beta}}+k^{2}+h^{4}\leq3(k+h^{4})$. Using this, $(\ref{79})$ implies
       \begin{equation*}
      \underset{0\leq n\leq N-1}{\max}\|e^{n+\frac{1}{2}}\|_{2},\text{\,\,\,\,}\underset{0\leq n\leq N}{\max}\|e^{n}\|_{2}\leq
       3\sqrt{2\widehat{C}_{4}T}(k+h^{4}).
      \end{equation*}
      This completes the proof of Theorem $\ref{t}$.
   \end{proof}

   \section{Numerical experiments and Convergence rate}\label{sec4}
   This section considers some computational results to show the unconditional stability and the convergence order of the new approach
    $(\ref{s1})$-$(\ref{s5})$ applied to time-variable fractional mobile-immobile advection-dispersion equation $(\ref{1})$ subjects to initial and
    boundary value conditions $(\ref{2})$ and $(\ref{3})$, respectively. To demonstrate the efficiency and accuracy of the proposed algorithm, two
    examples are taken in \cite{27mc}. We set $k=h^{4}$, where $h\in\{2^{-i},\text{\,\,\,}i=1,2,3,4\}$ so, $k=2^{-4r}$, $r=1,2,3,4$, and we compute
    the $L^{\infty}$-norm of the numerical solution: $U^{n}$, the exact one: $u^{n}$, and the corresponding error: $e^{n}$, at time level $n$, using
    the following formulas
   \begin{equation*}
        \||U|\|_{\infty,2}=\underset{0\leq n\leq N}{\max}\left(h\underset{j=2}{\overset{M-2}\sum}|U_{j}^{n}|^{2}\right)^{\frac{1}{2}},\text{\,\,\,}
        \||u|\|_{\infty,2}=\underset{0\leq n\leq N}{\max}\left(h\underset{j=2}{\overset{M-2}\sum}|u_{j}^{n}|^{2}\right)^{\frac{1}{2}},
       \end{equation*}
     and
     \begin{equation*}
       \||e(h)|\|_{\infty,2}=\underset{0\leq n\leq N}{\max}\left(h\underset{j=2}{\overset{M-2}\sum}|u_{j}^{n}-U_{j}^{n}|^{2}\right)^{\frac{1}{2}}.
       \end{equation*}
       In each example, the numerical evidences are performed with two different order functions: $\beta_{1}(x,t)=1-2^{-1}e^{-xt}$ and
      $\beta_{2}(x,t)=\frac{4}{3}-5.10^{-3}\cos(xt)\sin(xt)$. Furthermore, the convergence rate $R(k,h)$ of the new algorithm is estimate using the
       formula
       \begin{equation*}
        R(k,h)=\log_{2}(\||e(2h)|\|_{\infty,2}/\||e(h)|\|_{\infty,2}),
       \end{equation*}
       where we set $k=h^{4}$. Finally, the numerical computations are carried out by the use of MATLAB R$2013b$.\\

    Figures $\ref{figure1}$-$\ref{figure4}$ suggest that the proposed two-step technique $(\ref{s1})$-$(\ref{s5})$ is unconditionally stable whereas
    \textbf{Tables} $1$-$4$ indicate that the developed numerical method is temporal first-order accurate and spatial fourth-order convergent.
    These numerical studies confirm the theoretical results provided in Section $\ref{sec3}$, Theorem $\ref{t}$.\\

   $\bullet$ \textbf{Example $1.$} Let $D=[0,1]\times[0,1]$ be the domain. The parameter $\alpha=0.25$, $0.49$ and the function $\beta$ is defined as
   $\beta(x,t)=1-2^{-1}e^{-xt}$. Consider the following time-variable fractional mobile-immobile defined in \cite{27mc} by
   \begin{equation*}
    \left\{
      \begin{array}{ll}
        u_{t}(x,t)+cD_{0t}^{\beta(x,t)}u(x,t)=-u_{x}(x,t)+u_{2x}(x,t)+f(x,t) & \hbox{on\text{\,\,\,\,}D,} \\
       \text{\,}\\
        u(x,0)=u_{0}(x)=10x^{2}(1-x)^{2} & \hbox{\text{\,\,\,for\,\,\,}$0\leq x\leq1$,} \\
        \text{\,}\\
        u(0,t)=u(1,t)=0 & \hbox{\text{\,\,\,for\,\,\,}$0\leq t\leq1$,} \\
      \end{array}
    \right.
   \end{equation*}
     where $f(x,t)=10x^{2}(1-x)^{2}+\frac{10x^{2}(1-x)^{2}t^{1-\beta(x,t)}}{\Gamma(2-\beta(x,t))}+10(1+t)(2x-6x^{2}+4x^{3})-10(1+t)(2-12x+12x^{2})$.
    The analytical solution $u$ is given by
      \begin{equation*}
     u(x,t)=10(1+t)x^{2}(1-x)^{2}.
     \end{equation*}

       \textbf{Table 1} $\label{T1}$. Stability and convergence rate $R(k,h)$ of the two-step fourth-order approach with varying
      space step $h$ and time step $k$. We take $\alpha=0.25$, $\beta(x,t)=1-2^{-1}e^{-xt}$ and $k=h^{4}$.
          \begin{equation*}
           \begin{tabular}{|c|c|c|c|c|}
            \hline
            $k$ & $\||u|\|_{\infty,2}$ &$\||U|\|_{\infty,2}$ & $\||E(h)|\|_{\infty,2}$ & R(k,h)\\
            \hline
            $h^{-1}$ & $6.1049\times10^{-1}$ & $6.011\times10^{-1}$  &  $6.4483\times10^{-2}$ &   --  \\
            \hline
            $h^{-2}$ & $3.9871\times10^{-1}$ & $3.9279\times10^{-1}$  & $5.4102\times10^{-3}$ & 3.5671 \\
            \hline
            $h^{-3}$ & $2.8163\times10^{-1}$ & $2.7747\times10^{-1}$  &  $4.2155\times10^{-4}$ & 3.6819 \\
            \hline
            $h^{-4}$ & $1.9920\times10^{-1}$ & $1.9629\times10^{-1}$   &  $2.6665\times10^{-5}$ & 3.9827 \\
           \hline
          \end{tabular}
            \end{equation*}

         \textbf{Table 2} $\label{T2}$. Stability and Convergence rate $R(k,h),$ of the new technique with varying spacing $h$ and
          time step $k$. Here we take $\alpha=0.49$, $\beta(x,t)=1-2^{-1}e^{-xt}$ and $k=h^{4}.$
         \begin{equation*}
         \begin{tabular}{|c|c|c|c|c|}
            \hline
            $k$ & $\||u|\|_{\infty,2}$ &$\||U|\|_{\infty,2}$ & $\||E(h)|\|_{\infty,2}$ & R(k,h)\\
            \hline
            $h^{-1}$ & $6.1503\times10^{-1}$  & $6.0899\times10^{-1}$  &  $4.9235\times10^{-3}$ &   --  \\
            \hline
            $h^{-2}$ & $3.9891\times10^{-1}$  & $3.9796\times10^{-1}$  &  $3.3821\times10^{-4}$ & 3.8637 \\
            \hline
            $h^{-3}$ & $2.8176\times10^{-1}$  & $2.7889\times10^{-1}$  &  $1.9746\times10^{-5}$ & 4.0983\\
            \hline
            $h^{-4}$ &  $1.9925\times10^{-2}$ & $1.9727\times10^{-1}$  &  $1.1488\times10^{-6}$ & 4.1029 \\
           \hline
          \end{tabular}
            \end{equation*}

      $\bullet$ \textbf{Example $2.$} Let $D$ be the bounded domain $[0,1]\times[0,1]$. We assume that the parameters $\alpha\in\{0.25,0.49\}$ and the
     function $\beta$ is given by $\beta(x,t)=\frac{4}{3}-5.10^{-3}\cos(xt)\sin(xt)$. We consider the following time-variable fractional
     mobile-immobile advection-dispersion model defined in \cite{27mc} as
    \begin{equation*}
    \left\{
      \begin{array}{ll}
        u_{t}(x,t)+cD_{0t}^{\beta(x,t)}u(x,t)=-u_{x}(x,t)+u_{2x}(x,t)+f(x,t) & \hbox{on\text{\,\,\,\,}D,} \\
       \text{\,}\\
        u(x,0)=u_{0}(x)=5\sin(\pi x) & \hbox{\text{\,\,\,for\,\,\,}$0\leq x\leq1$,} \\
        \text{\,}\\
        u(0,t)=u(1,t)=0 & \hbox{\text{\,\,\,for\,\,\,}$0\leq t\leq1$,} \\
      \end{array}
    \right.
   \end{equation*}
     where $f(x,t)=5(1+\pi^{2}(1+t))\sin(\pi x)+\frac{5\sin(\pi x) t^{1-\beta(x,t)}}{\Gamma(2-\beta(x,t))}-5\pi(1+t)\cos(\pi x)$.
    The exact solution $u$ is defined as
      \begin{equation*}
     u(x,t)=5(1+t)x^{2}\sin(\pi x).
     \end{equation*}

       \textbf{Table 3} $\label{T3}$. Unconditional stability and convergence rate $R(k,h),$ for the two-level approach with varying time
      step $k$ and space step $h$. Here we take $\alpha=0.25$, $\beta(x,t)=\frac{4}{3}-5.10^{-3}\cos(xt)\sin(xt)$ and $k=h^{4}.$
      \begin{equation*}
            \begin{tabular}{|c|c|c|c|c|}
            \hline
            $k$ & $\||u|\|_{\infty,2}$ &$\||U|\|_{\infty,2}$ & $\||E(h)|\|_{\infty,2}$ & R(k,h)\\
            \hline
            $h^{-1}$ & $1.3672$ & $1.3534$  &  $1.0918\times10^{-2}$ &   --  \\
            \hline
            $h^{-2}$ & $1.3248$ & $1.3115$  & $7.3334\times10^{-4}$ & 3.8937\\
            \hline
            $h^{-3}$ & $9.4892\times10^{-1}$ & $9.4356\times10^{-1}$  &  $4.5602\times10^{-5}$ & 4.0073\\
            \hline
           $h^{-4}$ & $6.7163\times10^{-1}$ & $6.6689\times10^{-1}$   &  $2.5995\times10^{-6}$ & 4.1328 \\
           \hline
          \end{tabular}
            \end{equation*}

            \textbf{Table 4} $\label{t4}$. Stability and accuracy $R(k,h),$ of the new technique with mesh size $h$ and time step $k$. In this example
         we set $\alpha=0.49$, $\beta(x,t)=\frac{4}{3}-5.10^{-3}\cos(xt)\sin(xt)$ and $k=h^{4}.$
          \begin{equation*}
            \begin{tabular}{|c|c|c|c|c|}
            \hline
            $k$ & $\||u|\|_{\infty,2}$ &$\||U|\|_{\infty,2}$ & $\||E(h)|\|_{\infty,2}$ & R(k,h)\\
            \hline
            $h^{-1}$ & $1.5172$ & $1.5096$  & $7.6372\times10^{-3}$ &   --  \\
            \hline
            $h^{-2}$ & $1.4664$ & $1.4590$  & $4.8459\times10^{-4}$ & 3.9782 \\
            \hline
            $h^{-3}$ & $1.0502$ & $1.0449$  & $2.8188\times10^{-5}$ & 4.1036 \\
            \hline
           $h^{-4}$ & $7.4327\times10^{-1}$ & $7.3959\times10^{-1}$   &  $1.6242\times10^{-6}$ & 4.1173 \\
           \hline
          \end{tabular}
            \end{equation*}
          We observe from this table that the proposed method is temporal second order convergent and spatial fourth order accurate.

     \section{General conclusions and future works}\label{sec5}
     This paper has developed a two-step fourth-order modified explicit Euler/Crank-Nicolson formulation for solving the time-variable fractional
     mobile-immobile advection-dispersion model subjects to suitable initial and boundary value conditions. Both stability and error estimates of the
     new technique have been deeply analyzed in $L^{\infty}(0,T;L^{2})$-norm. The theory has shown that the proposed approach is unconditionally stable,
     first-order convergence in time and fourth-order accurate in space (see Theorem $\ref{t}$). This theoretical analysis is confirmed by two numerical
    examples. Especially, the graphs (Figures $\ref{figure1}$-$\ref{figure4}$) show that the new procedure is both unconditionally stable and convergent
    whereas \textbf{Tables 1-4} indicate the convergence rate of the developed algorithm (convergence with order $O(k)$ in time and fourth-order accurate
    in space). Furthermore, the theoretical and numerical studies suggested that the proposed numerical scheme $(\ref{s1})$-$(\ref{s5})$ is faster and
    efficient than a wide set of numerical methods \cite{27mc,28ll,29ll,30zjl,22zjl,10zjl} developed for the considered problem
    $(\ref{1})$-$(\ref{3})$. Solving the two-dimensional time-variable fractional problems by the use of a two-step fourth-order explicit/implicit
    scheme will be the topic of our future work. In addition, we will also be interested in the analysis of the Preconditioned Generalized Minimal
    Residual Method in an efficient solution of linear systems of equations $(\ref{s1})$-$(\ref{s5})$. Specifically, in view to construct efficient
    preconditioners for such systems of linear equations, our study will consider the spectral behavior of the sequences of coefficient matrices
    $\{A:=A_{M}\}_{M}$ and $\{A_{0}:=A_{0M}\}_{M}$ of size $(M-2)\times(M-2)$ in the sense of eigenvalues and clustering. For more details about the
    eigenvalues/singular values distribution and clustering, the readers can consult the works \cite{nss,enss,eng1,eng2} and references therein.\\

     \textbf{Acknowledgment.} This work has been partially supported by the Deanship of Scientific Research of Imam Mohammad Ibn Saud Islamic University
    (IMSIU) under the Grant No. $331203.$

    \newpage

          \begin{figure}
         \begin{center}
       Analysis of stability and convergence of a two-step Euler/Crank-Nicolson technique for time-variable fractional mobile-immobile with 
       $\alpha=0.25$ and $k=h^{4}.$
         \begin{tabular}{c c}
         \psfig{file=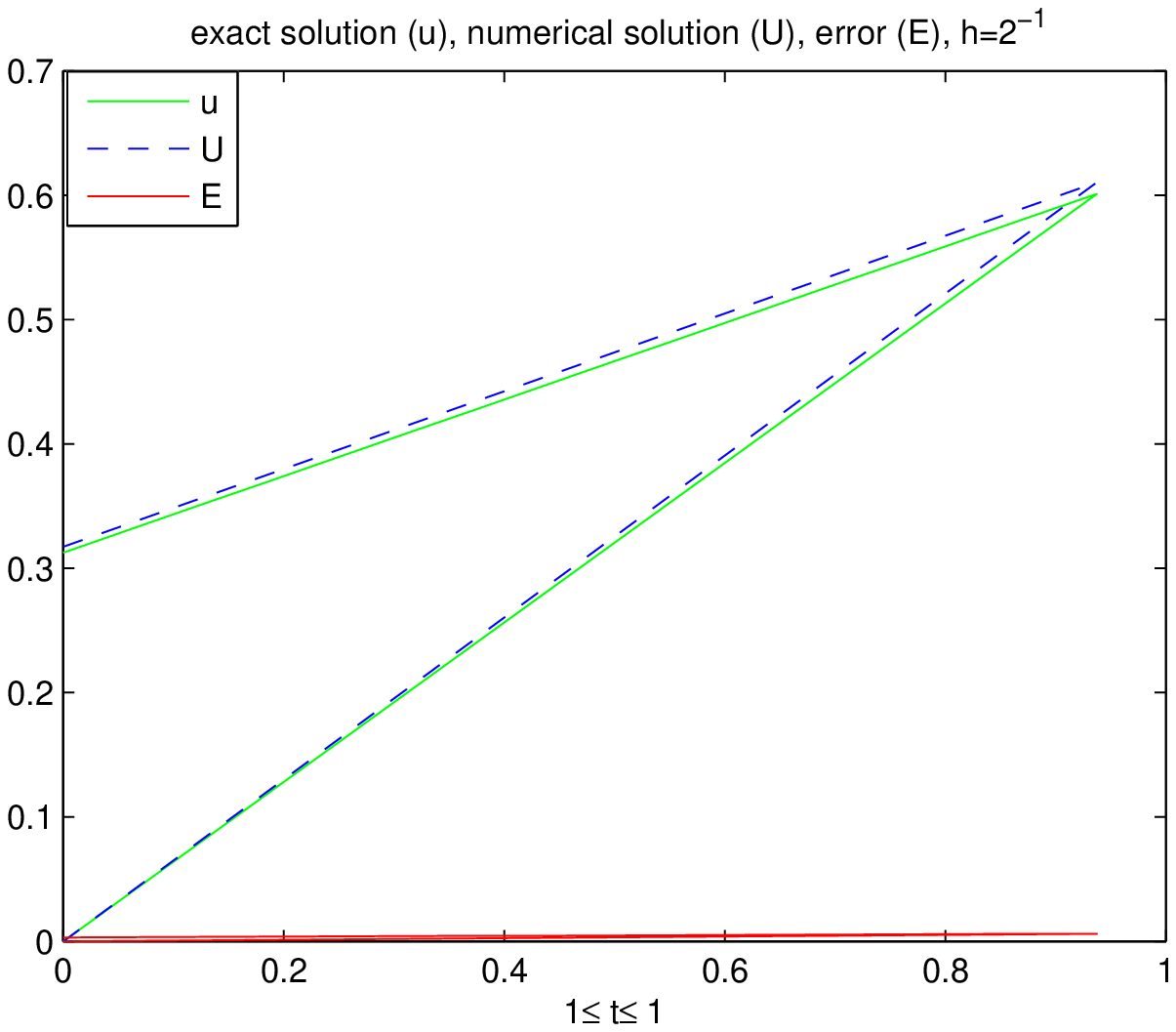,width=7cm} & \psfig{file=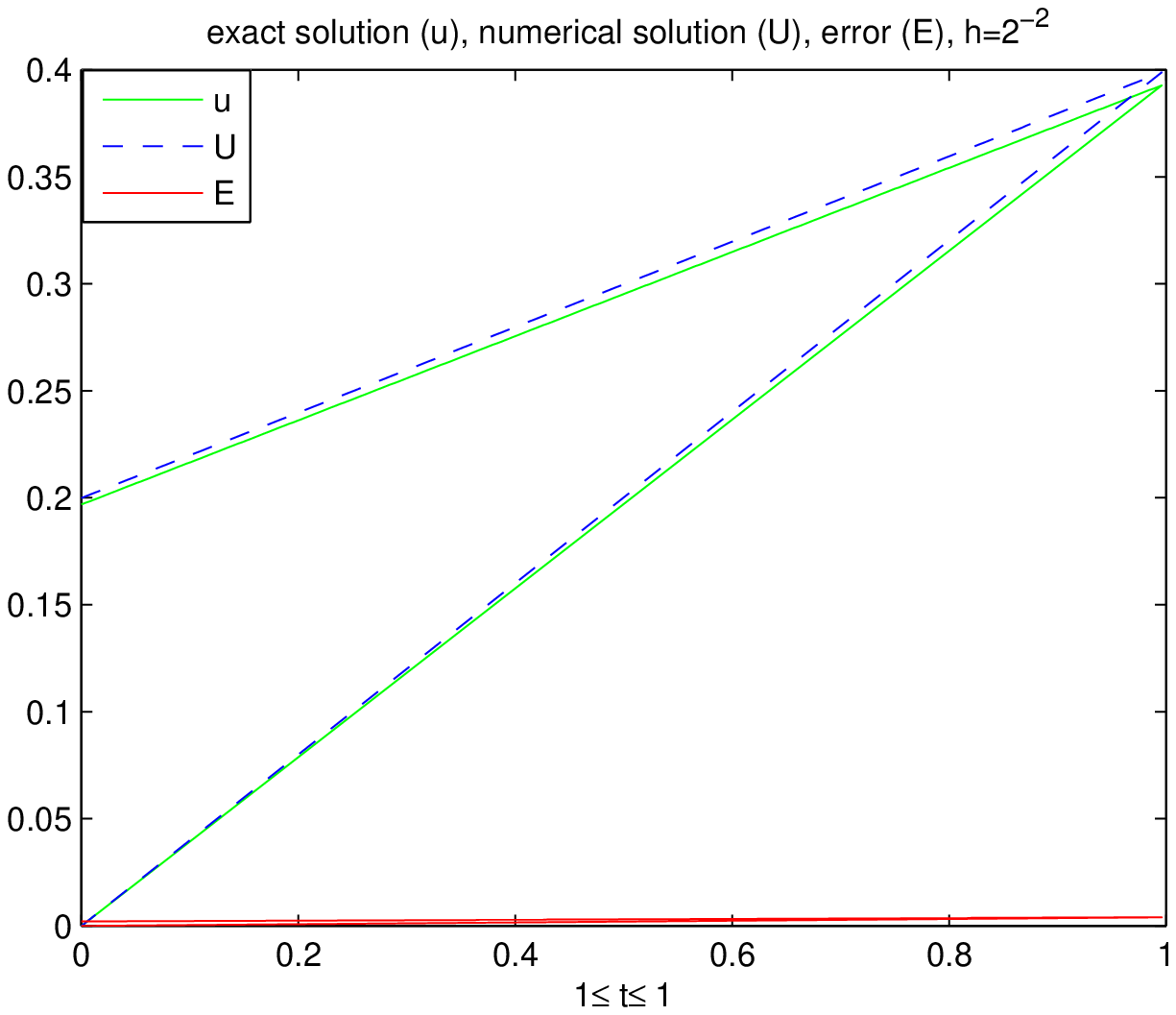,width=7cm}\\
         \psfig{file=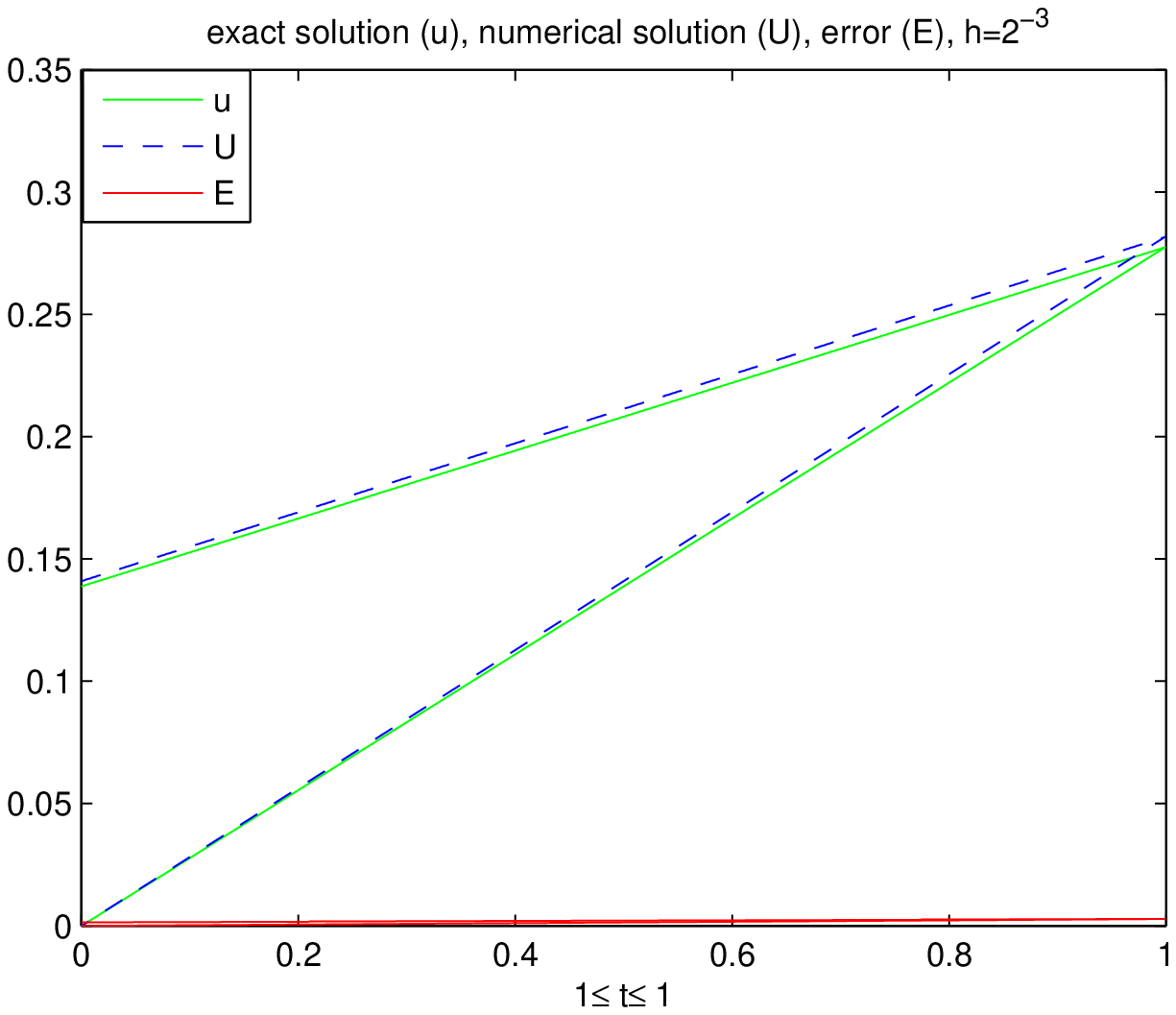,width=7cm} & \psfig{file=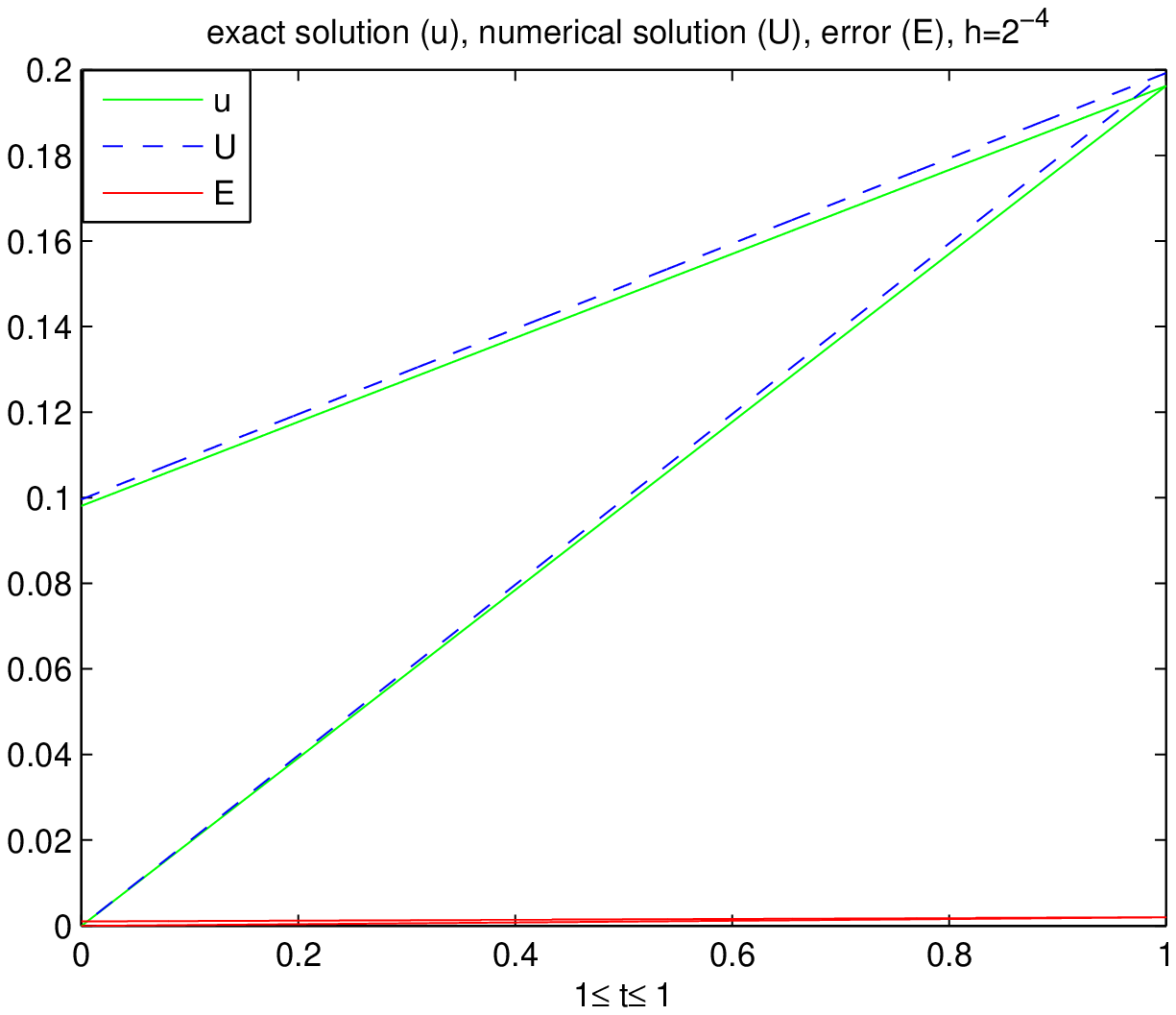,width=7cm}\\
         \end{tabular}
        \end{center}
        \caption{Exact solution (u: in green), Numerical solution (U: in blue) and Error (E: in red) for Problem 1}
        \label{figure1}
        \end{figure}

           \begin{figure}
         \begin{center}
         Stability and convergence of a two-step Euler/Crank-Nicolson approach for time-variable fractional mobile-immobile with $\alpha=0.49$ and
        $k=h^{4}.$
         \begin{tabular}{c c}
         \psfig{file=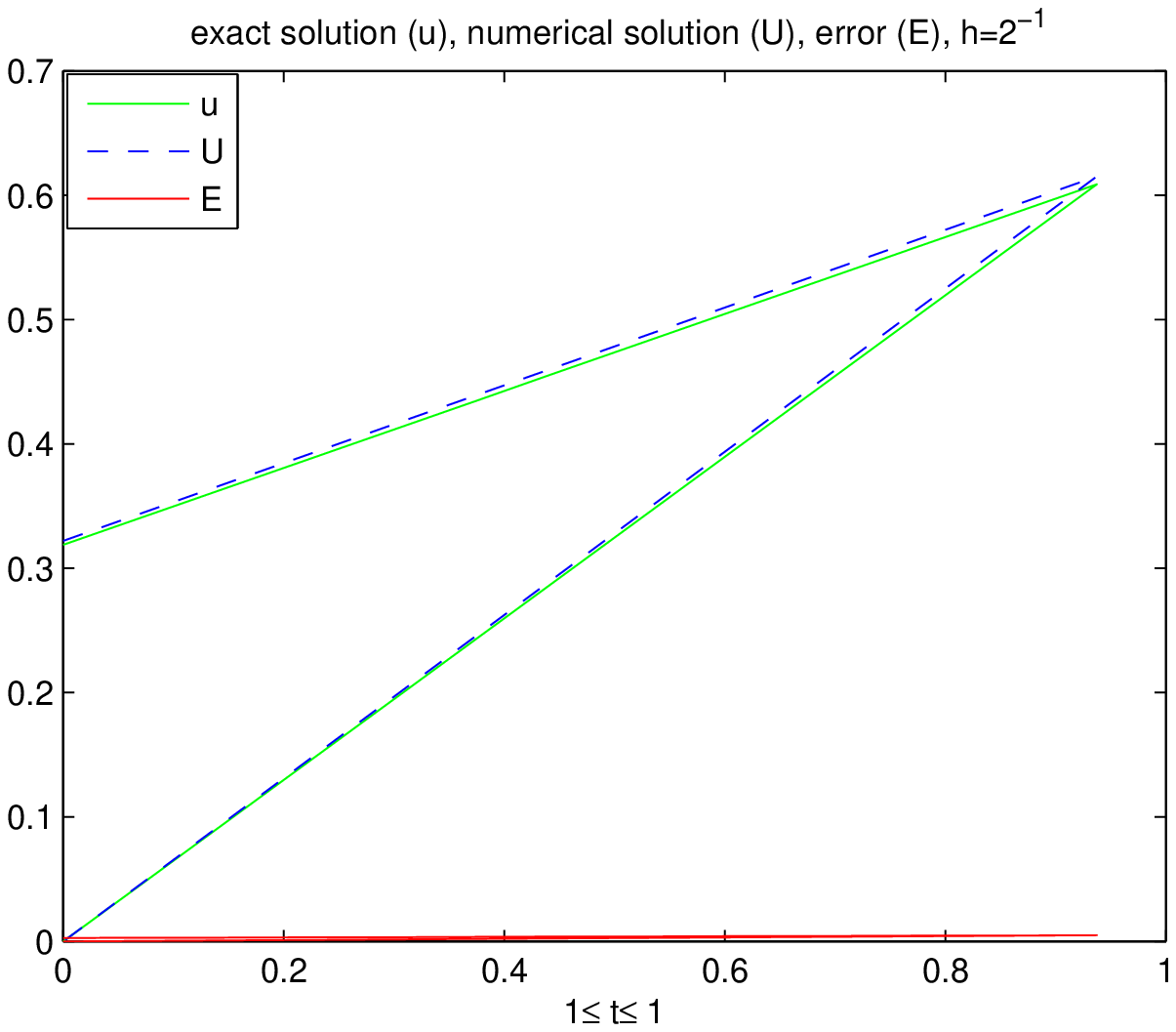,width=7cm} & \psfig{file=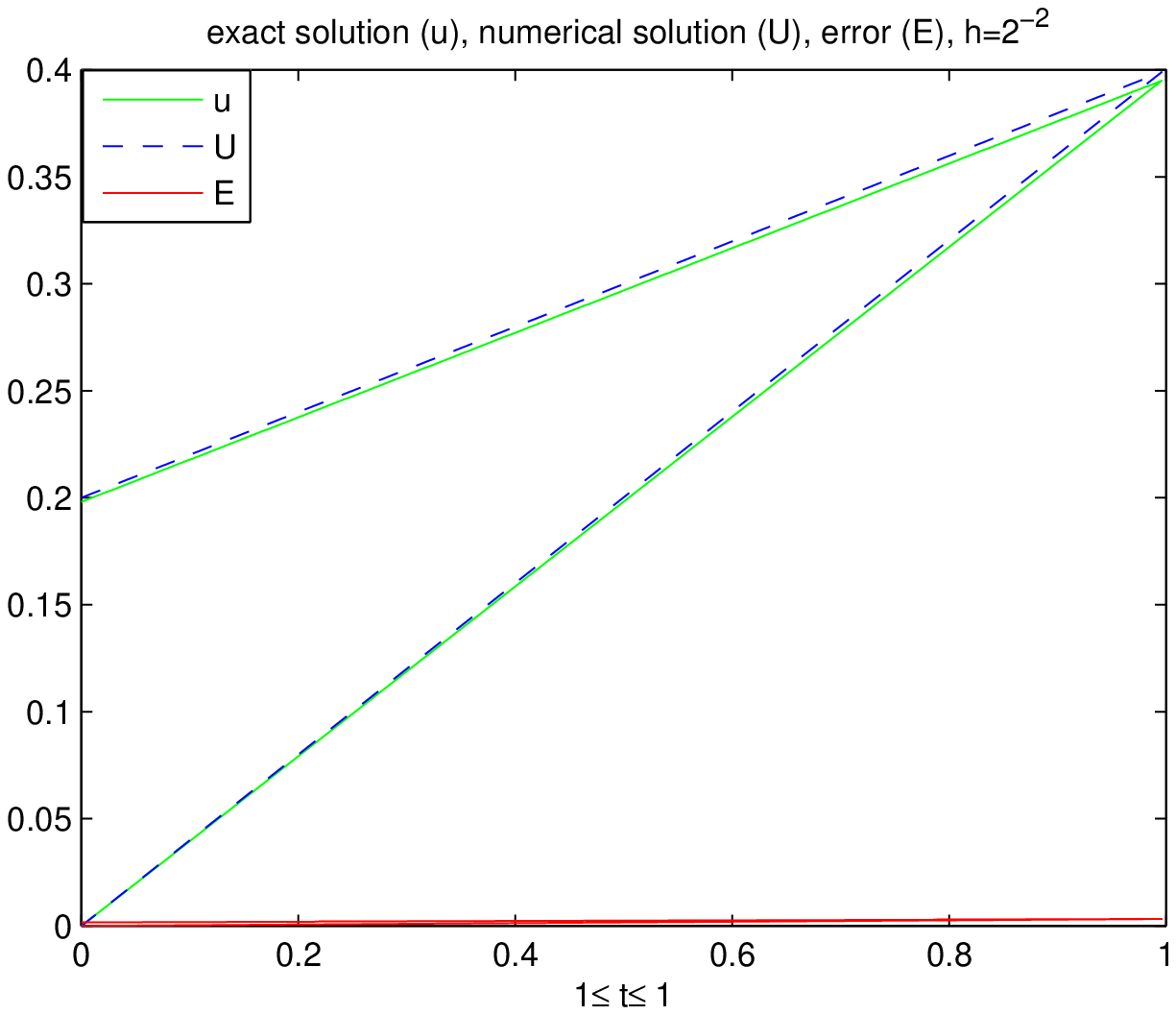,width=7cm}\\
         \psfig{file=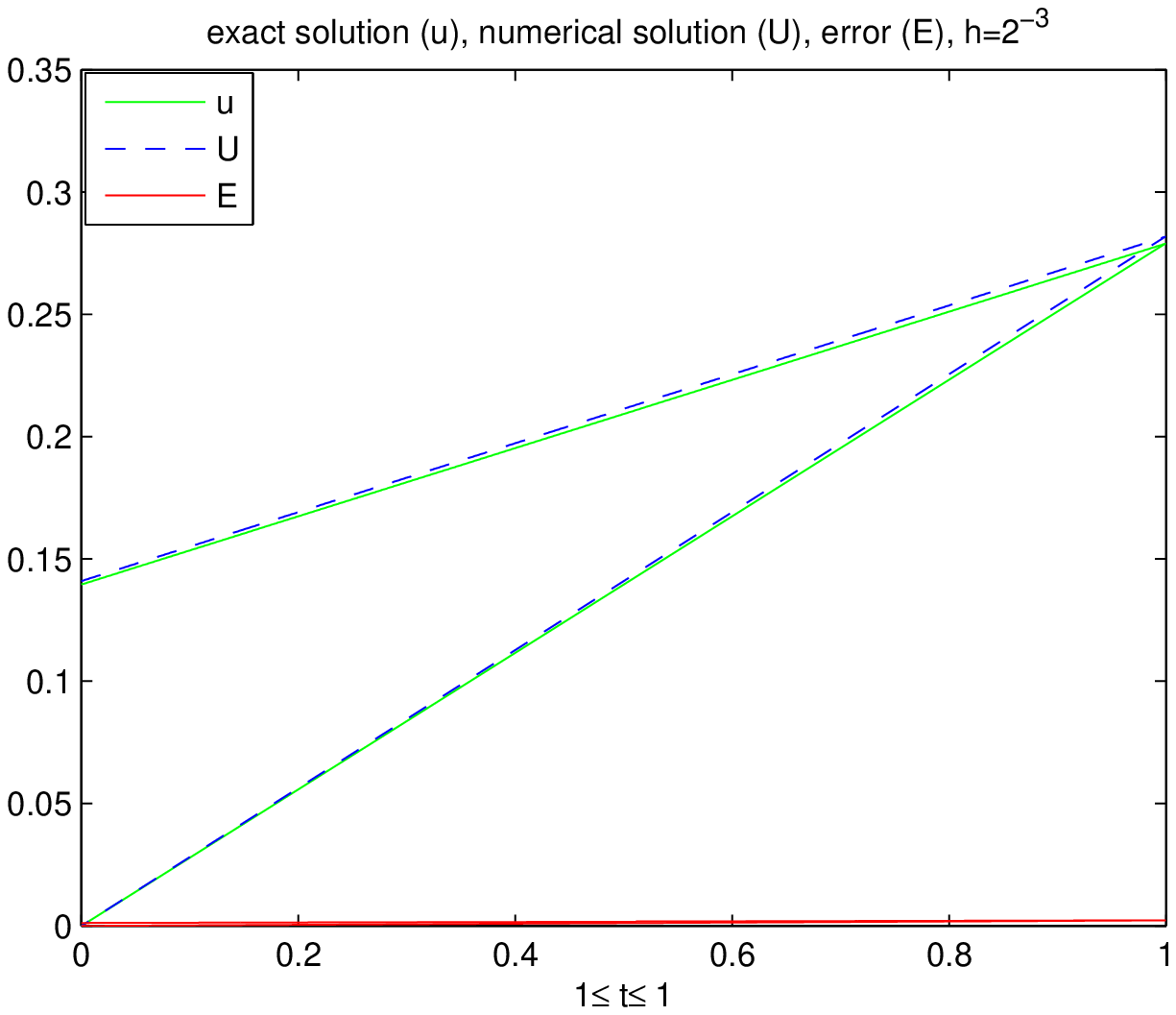,width=7cm} & \psfig{file=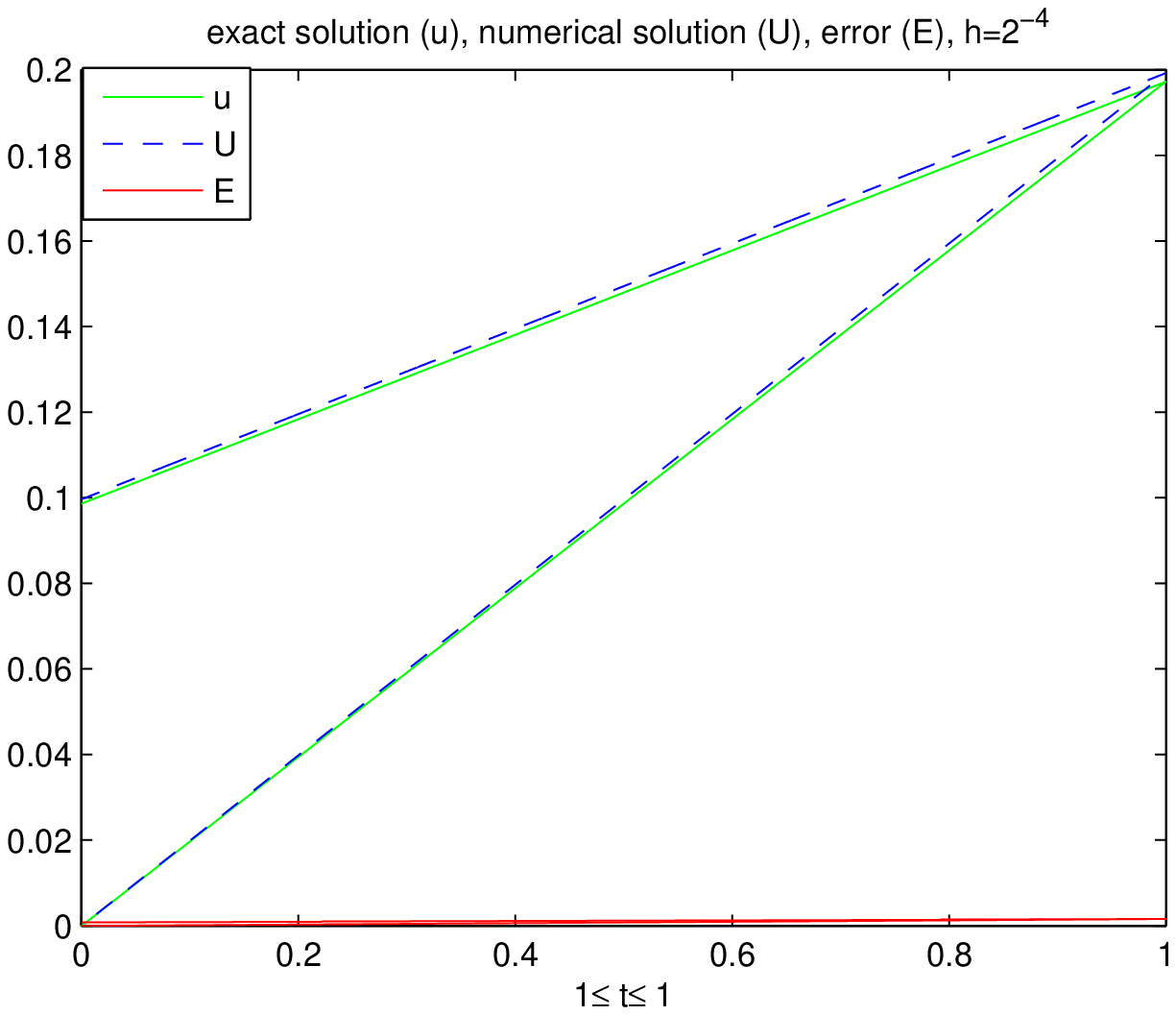,width=7cm}\\
         \end{tabular}
        \end{center}
        \caption{Exact solution (u: in green), Numerical solution (U: in blue) and Error (E: in red) for Problem 1}
        \label{figure2}
        \end{figure}

       \begin{figure}
         \begin{center}
          Stability and convergence of a two-step Euler/Crank-Nicolson approach for time-variable fractional mobile-immobile with $\lambda=0.25$ and
        $k=h^{4}.$
         \begin{tabular}{c c}
         \psfig{file=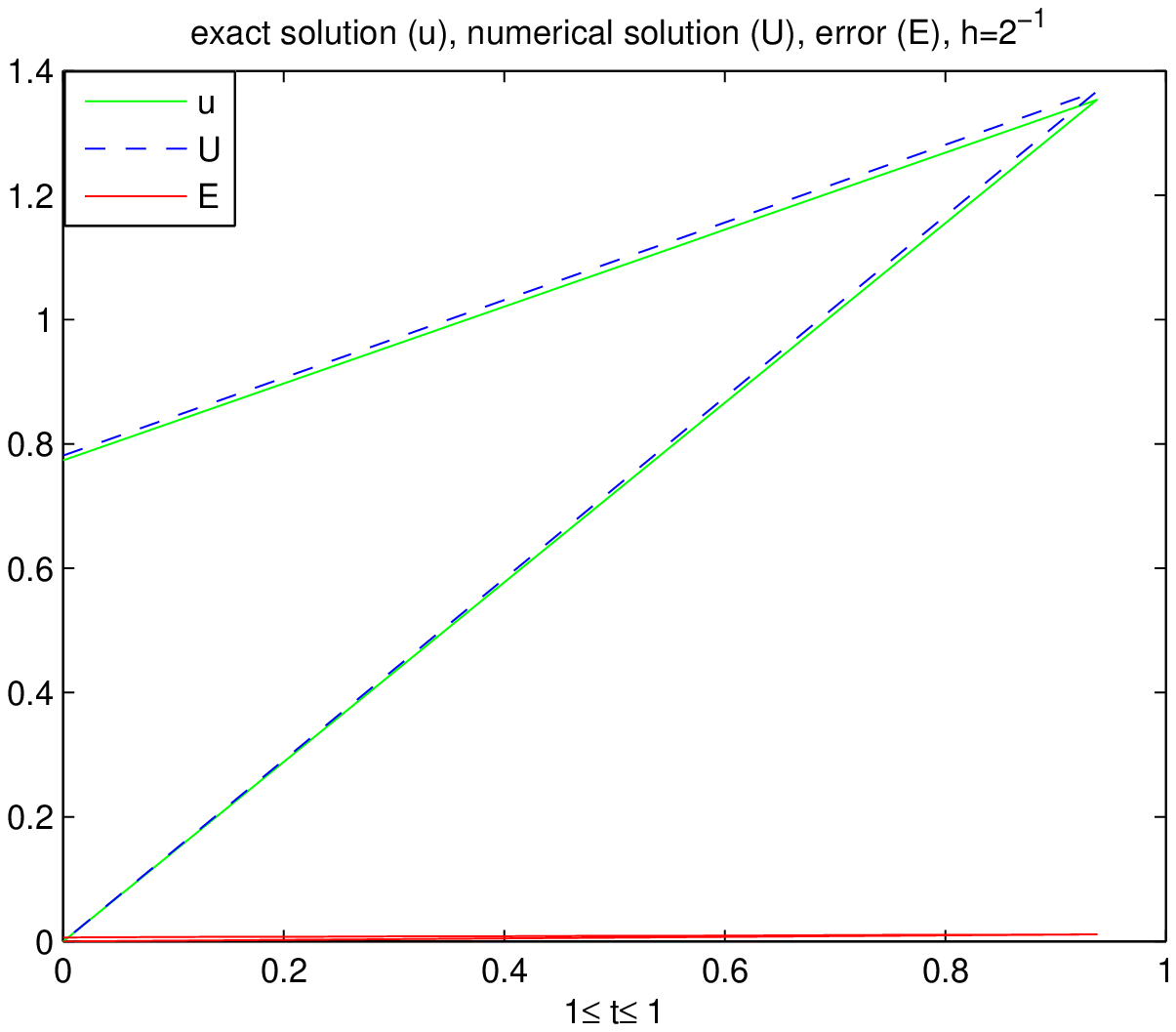,width=7cm} & \psfig{file=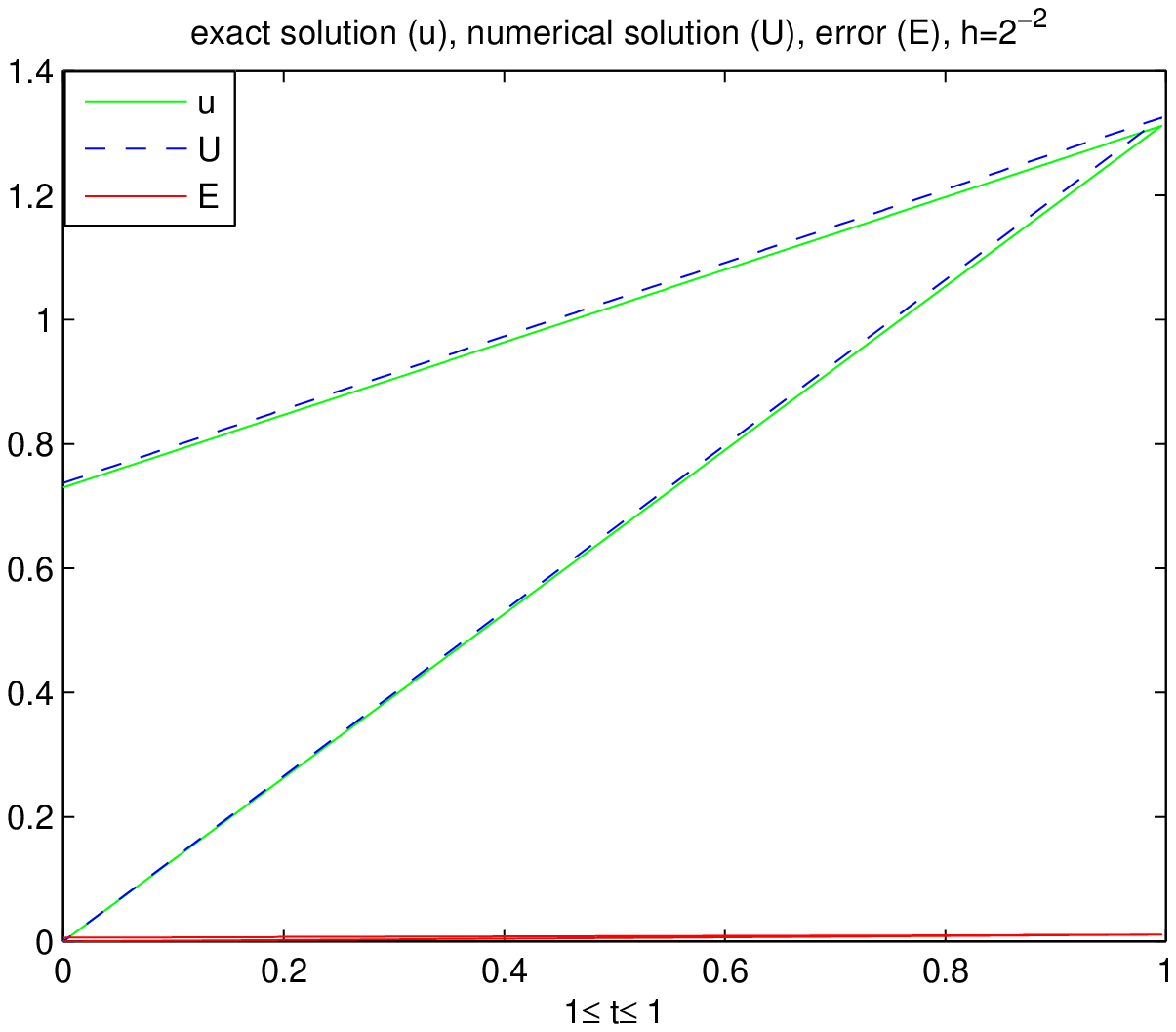,width=7cm}\\
         \psfig{file=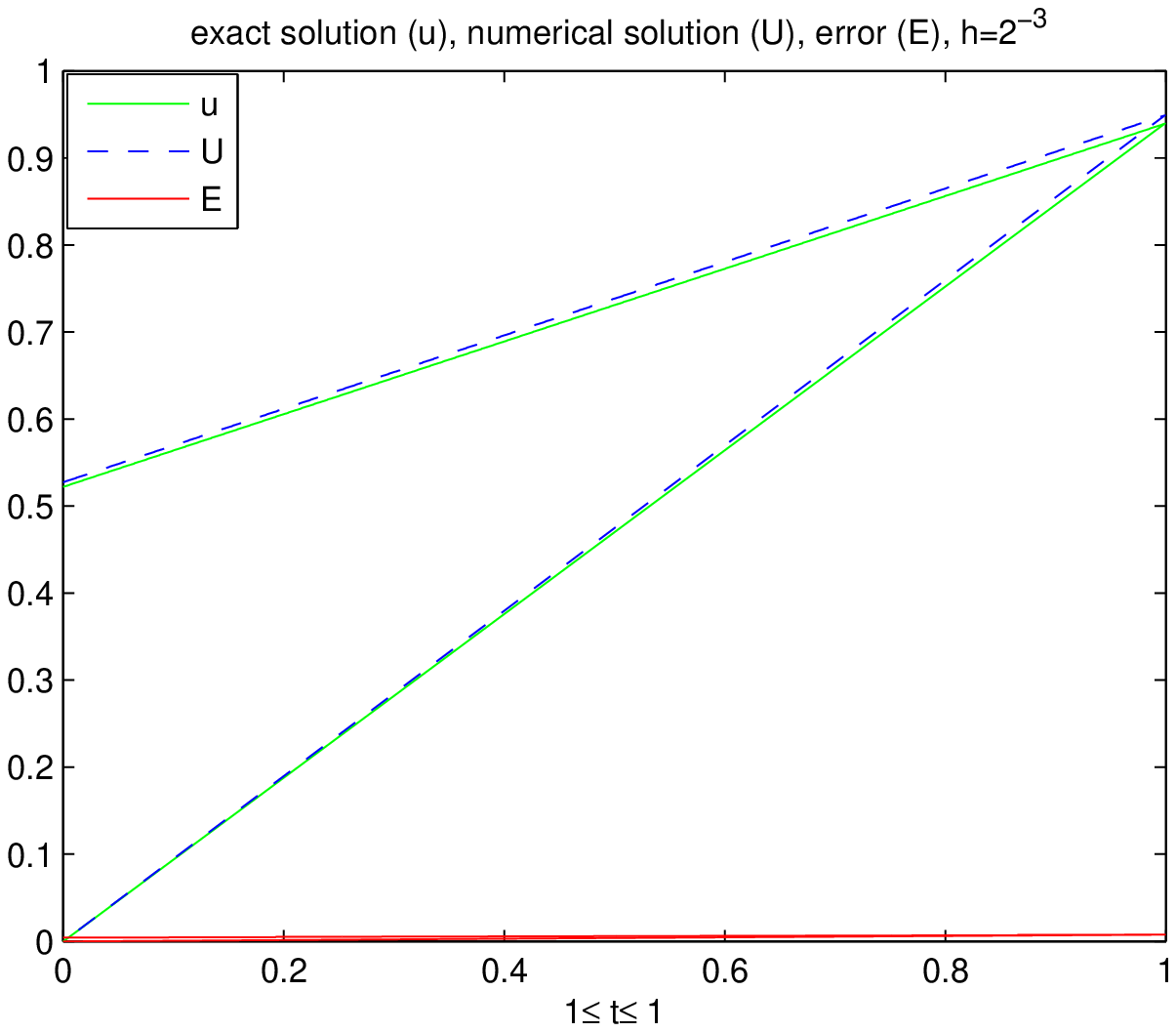,width=7cm} & \psfig{file=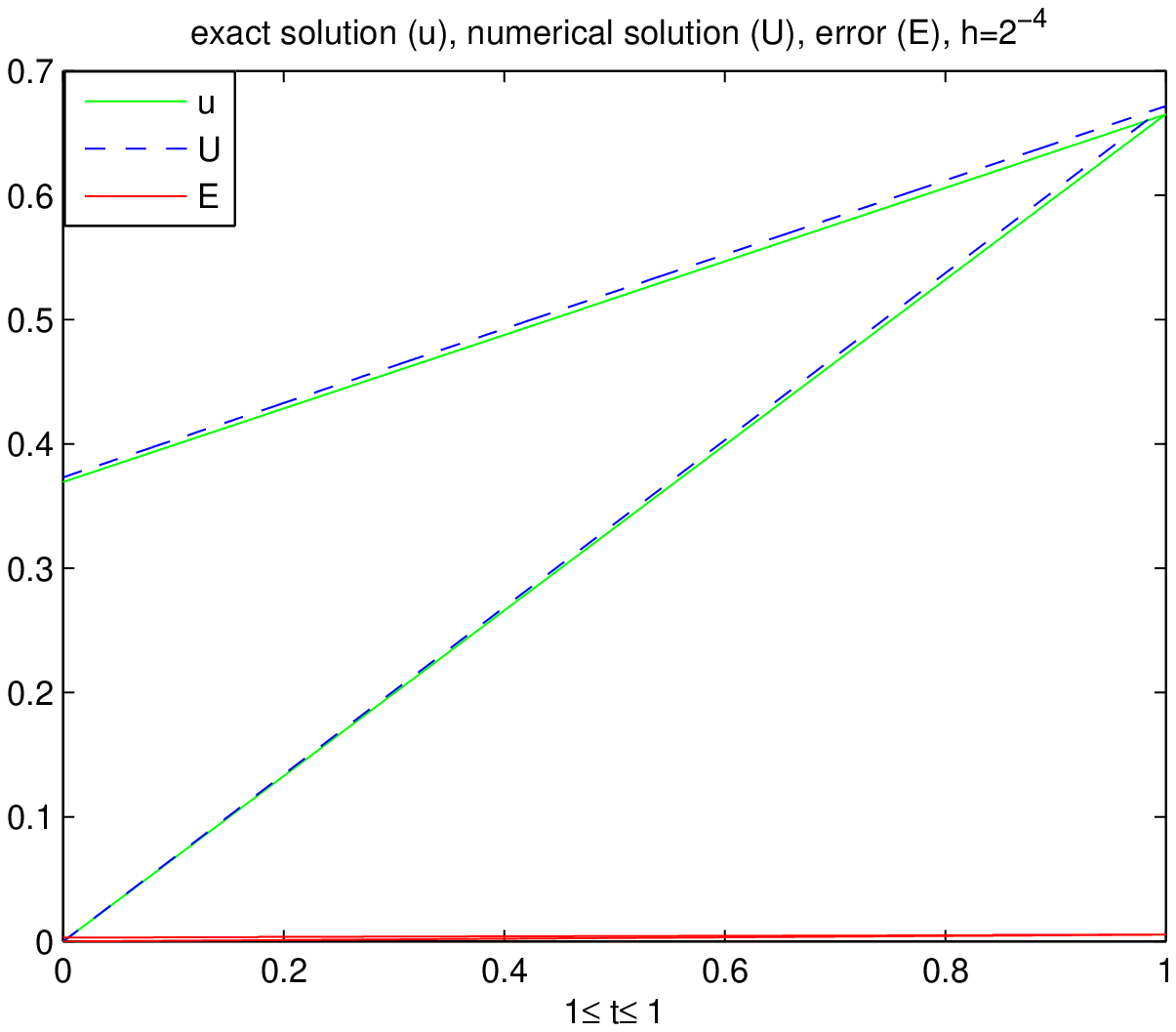,width=7cm}\\
         \end{tabular}
        \end{center}
        \caption{Exact solution (u: in green), Numerical solution (U: in blue) and Error (E: in red) for Problem 2}
        \label{figure3}
        \end{figure}

         \begin{figure}
         \begin{center}
       Analysis of stability and convergence of a two-step Euler/Crank-Nicolson numerical scheme for time-variable fractional mobile-immobile with 
       $\alpha=0.49$ and $k=h^{4}.$
         \begin{tabular}{c c}
         \psfig{file=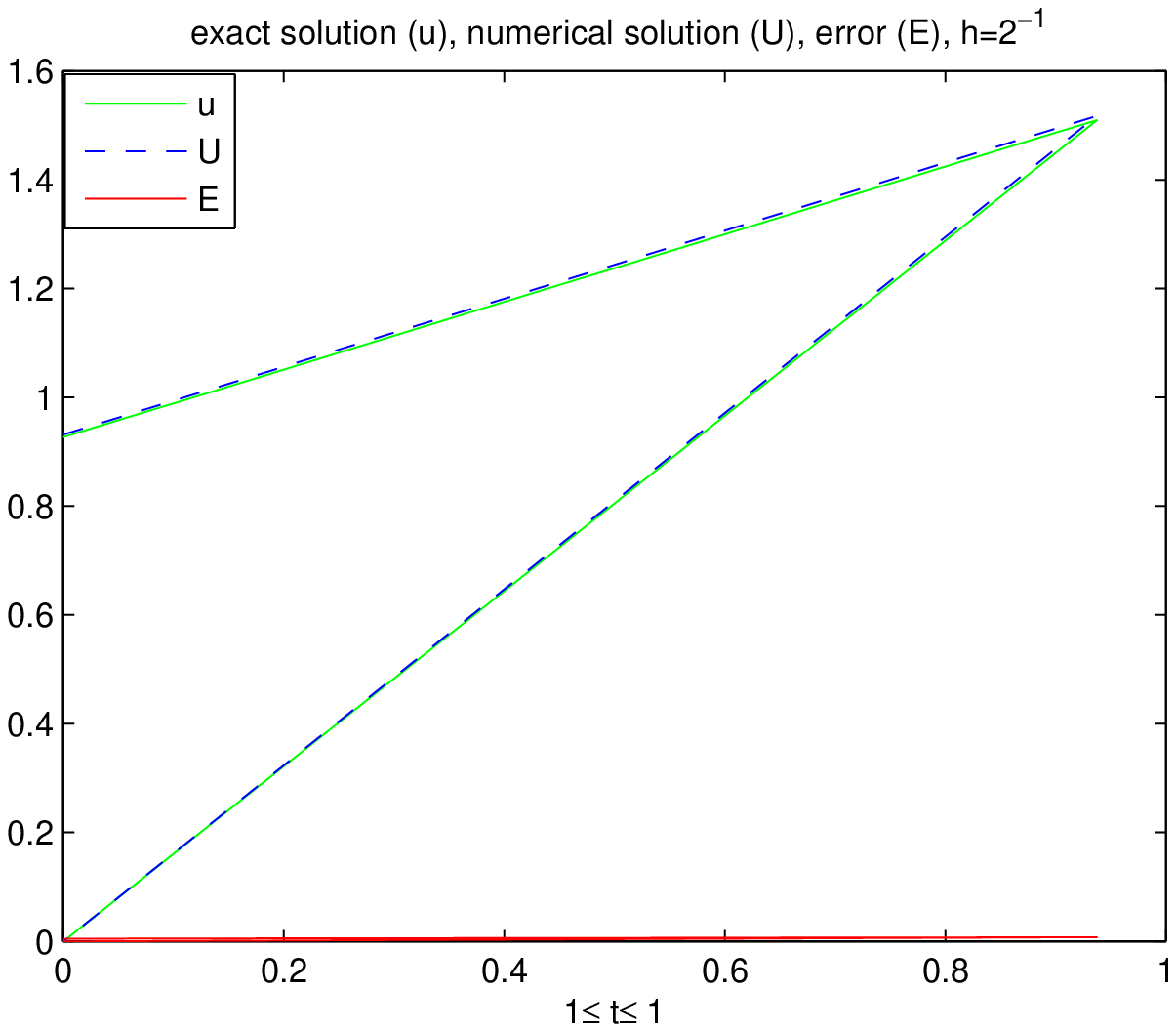,width=7cm} & \psfig{file=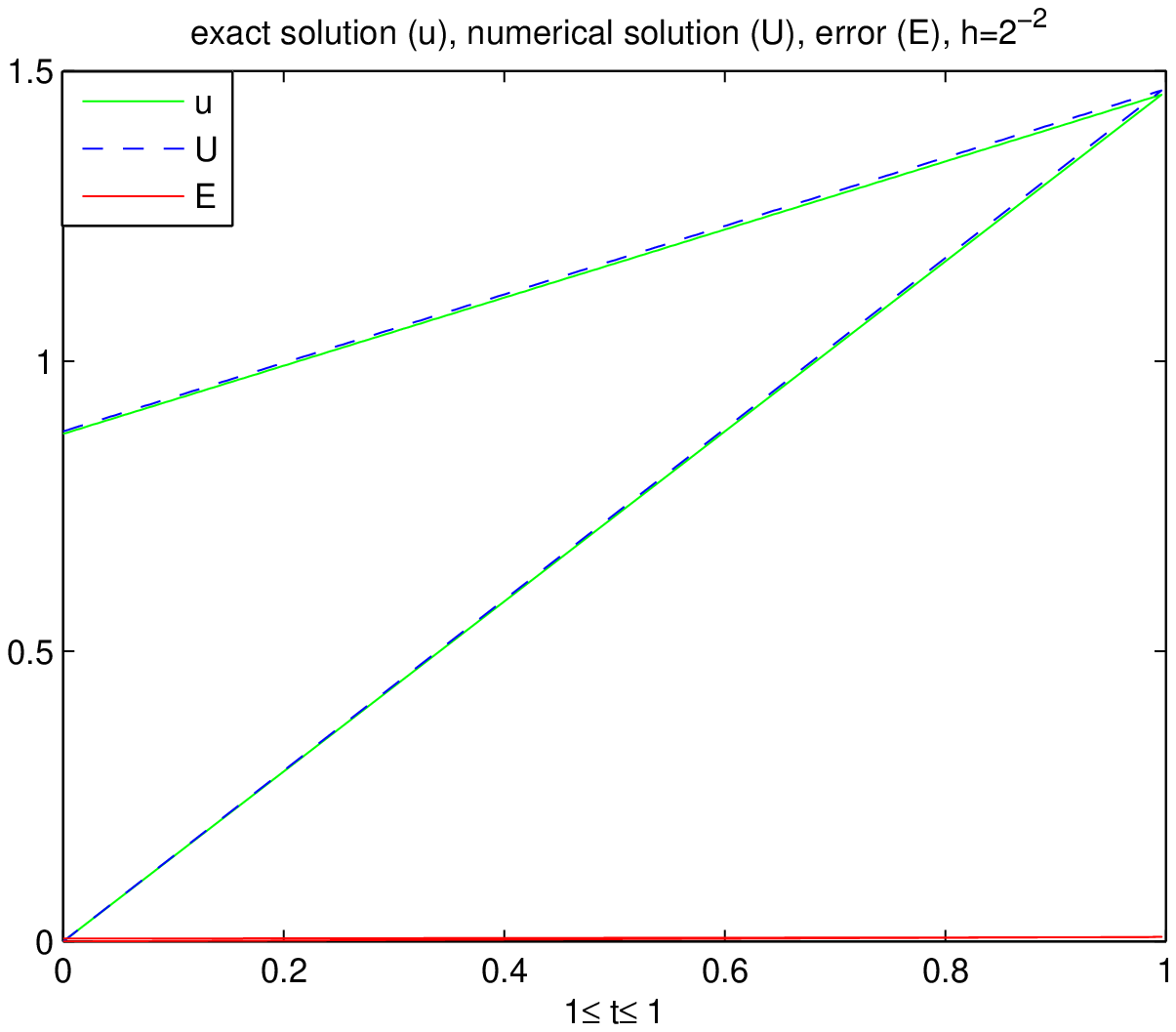,width=7cm}\\
         \psfig{file=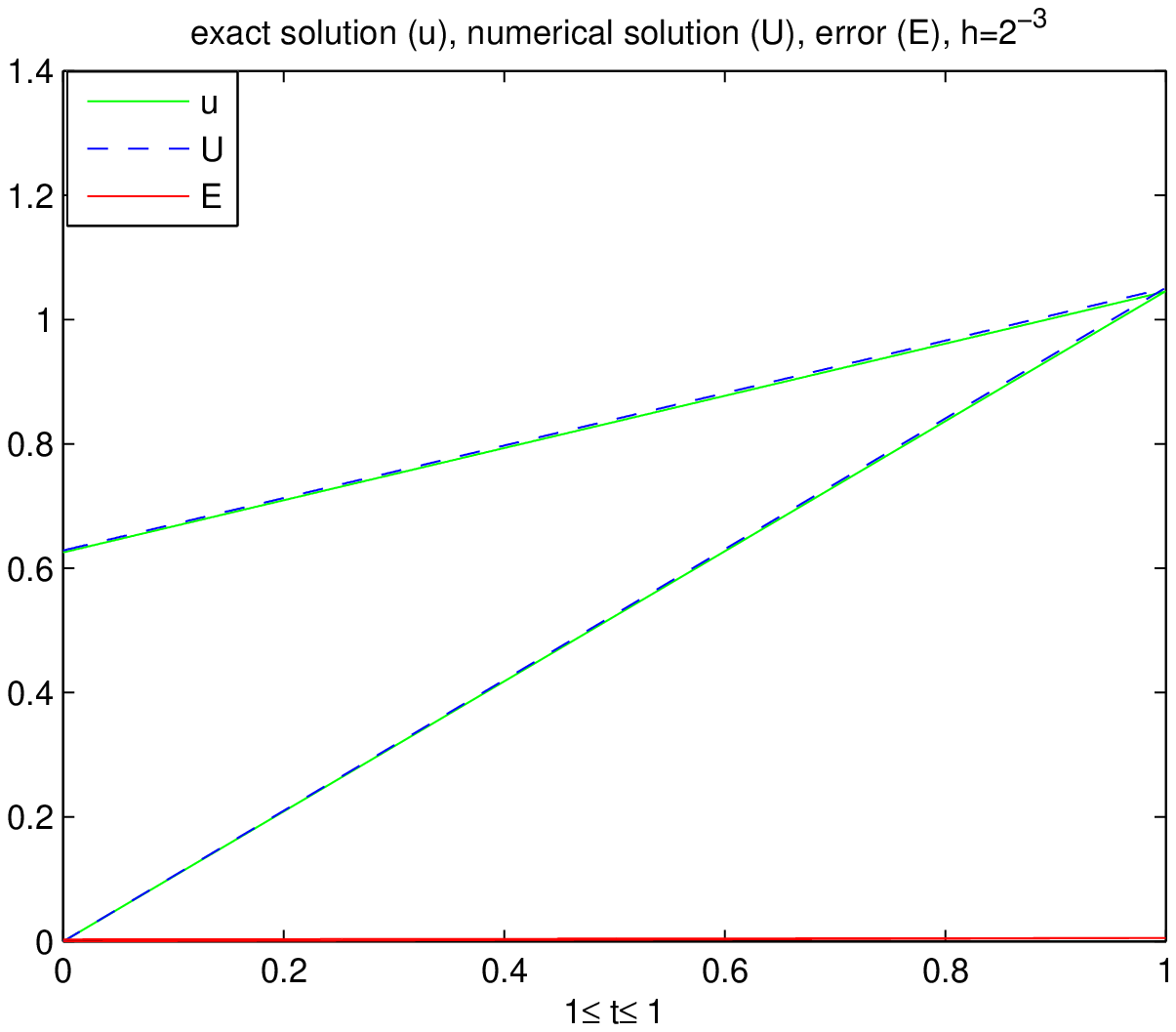,width=7cm} & \psfig{file=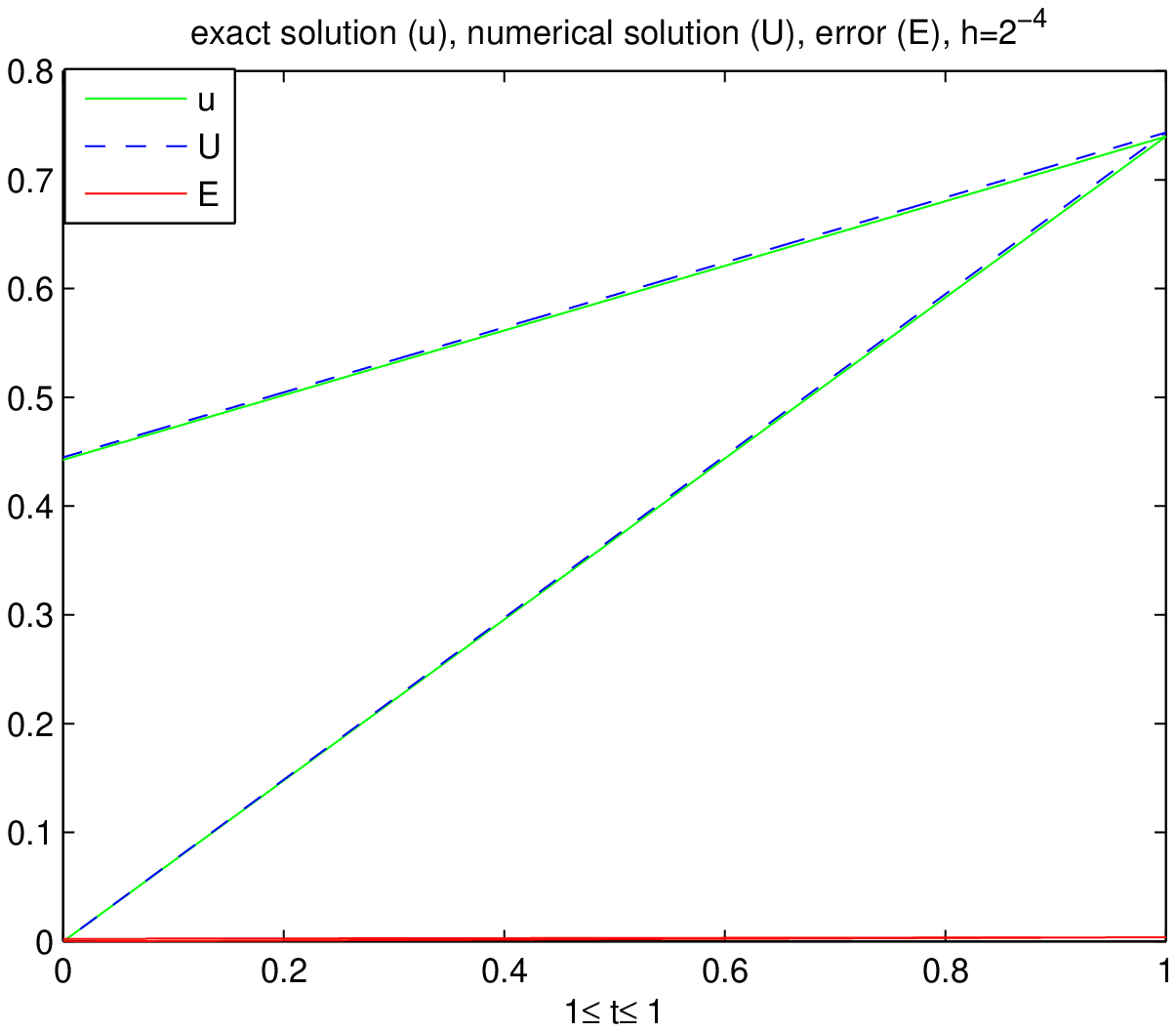,width=7cm}\\
         \end{tabular}
        \end{center}
        \caption{Exact solution (u: in green), Numerical solution (U: in blue) and Error (E: in red) for Problem 2}
        \label{figure4}
        \end{figure}
     \end{document}